\documentclass[11pt,letterpaper]{amsart}
\usepackage{amssymb,latexsym,amsmath, amsthm}
\usepackage{amsfonts,amssymb, latexsym, amsmath, amsthm,mathrsfs, verbatim,geometry}

\usepackage{amsfonts, graphics}
\usepackage{mathtools}
\usepackage{hyperref} 
\usepackage{csquotes} 
\usepackage{enumitem,moreenum} 
\usepackage{float}
\usepackage{dsfont,bbm} 
\usepackage{tikz}
\usepackage{tikz-cd}

\usetikzlibrary{shapes,backgrounds,shadows,arrows} 
\usetikzlibrary{calc,decorations.markings,math}
\numberwithin{equation}{section}
\def\R{\mathbb R}
\def\C{\mathbb C}
\def\N{\mathbb N}

\hypersetup{
    colorlinks,
    citecolor=red,
    filecolor=green,
    linkcolor=blue,
    urlcolor=black
}

\def\R{\mathbb{R}}
\def\C{\mathbb C}
\def\N{\mathbb N}
\def\Z{\mathbb Z}

\def\A{\mathbb{A}}

\def\T{\mathbb{T}}

\def\be{\begin{equation}}
\def\ee{\end{equation}}
\def\bea{\begin{eqnarray}}
\def\eea{\end{eqnarray}}
\def\beas{\begin{eqnarray*}}
\def\eeas{\end{eqnarray*}}
\def\supp{\mathrm{supp}\,}

\def\pa{\partial }
\def\S{\mathcal S}

\def\Ga{\Gamma}

\def\supp{\mathrm{supp}\,}

\def\T{{\mathcal T}}
\def\P{{\mathcal P}}

\def\J{{\mathcal J}}

\def\pa{\partial}
\def\l{\lambda}

\def\al{\alpha}
\def\de{\delta}
\def\ga{\gamma}
\def\la{\lambda}

\def\F{\mathcal F}

\def\H{\mathcal H}
\def\lv{\left\vert}
\def\rv{\right\vert}
\def\Emin{E_{\text{min}}}
\def\k{\varepsilon}

\def\dE{\mathcal E}

\def\Rmax{R_{\textup{max}}}
\def\Rmin{R_{\textup{min}}}
\def\la{\lambda}
\def\Rem{\mathcal{R}^{\textup{Ta}}}
\def\Poly{\mathcal{P}^{\textup{Ta}}}

\def\Ade{\mathbb{A}_\de}
\def\Ide{\I_\de}
\def\Jde{\J_{R,\de}}

\def\T{\mathcal{T}}

\def\C{\mathcal C}
\def\al{\alpha}
\def\de{\delta}
\def\ga{\gamma}
\def\la{\lambda}

\def\dif{\textup{d}}

\newcommand{\fhat}{\widehat{f}}
\newcommand{\ghat}{\widehat{g}}

\let\th\relax
\newcommand{\th}{\theta}
\newcommand{\eps}{\varepsilon}
\newcommand\om{\omega}

\def\Lmax{L_{\mathrm{max}}}
\def\EminL{E^L_{\min}}

\def\Res{\textup{Res}}

\newcommand{\beq}{\begin{equation}}
\newcommand{\eeq}{\end{equation}}
\newcommand{\beqs}{\begin{equation*}}
\newcommand{\eeqs}{\end{equation*}}
\newcommand{\beqa}{\begin{equation}\begin{aligned}}
\newcommand{\eeqa}{\end{aligned}\end{equation}}
\newcommand{\beqas}{\begin{equation*}\begin{aligned}}
\newcommand{\eeqas}{\end{aligned}\end{equation*}}

\let\originalleft\left
\let\originalright\right
\renewcommand{\left}{\mathopen{}\mathclose\bgroup\originalleft}
\renewcommand{\right}{\aftergroup\egroup\originalright}

\newcommand*\diff{\mathop{}\!\mathrm{d}} 

\newtheorem{theorem}{Theorem}[section]
\newtheorem{definition}[theorem]{Definition}
\newtheorem{prop}[theorem]{Proposition}

\newtheorem{corollary}[theorem]{Corollary}

\newtheorem{lemma}[theorem]{Lemma}

\newtheorem{remark}[theorem]{Remark}

\def\bcr{\begin{color}{red}}
\def\bcb{\begin{color}{blue}}
\def\bcc{\begin{color}{violet}}
\def\ec{\end{color}}
\def\be{\begin{equation}}
\def\ee{\end{equation}}

\def\L{\mathcal{L}}
\def\Om{\Omega}
\def\Ueps{U^\pm_{\k}}
\def\S{\mathbb S}

\def\dif{\textup{d}}

\def\fin{f_0}

\def\Lmax{L_{\text{max}}}
\def\I{\mathcal I}

\def\Pl{\text{{\bf Pl}}}
\def\n{k}
\def\Jvac{\pa{\mathcal J}_R^{\text{vac}}}
\def\plog{p^\eps_{\l,m}}
\def\bI{{\bf I}}

\def\wUeps{\widehat{U^\pm_{\eps}}}
\def\Uepsplus{{U_\k^{+}}}
\def\Uepsminus{{U_\k^{-}}}
\def\wUepsm{\widehat{U^{\pm}_{\eps}}\hspace{-1.6mm}{}_{,m}}
\def\wUepsplusm{\widehat{U^{+}_{\eps}}\hspace{-1.6mm}{}_{,m}}
\def\wUepsminusm{\widehat{U^{-}_{\eps}}\hspace{-1.6mm}{}_{,m}}
\def\Vhat{\widehat{V}}
\def\Ude{\Upsilon_\delta}
\def\PL{\Pl_{m}^{\pm,\l,\eps}}
\def\PLplus{\Pl_{m}^{+,\l,\eps}}
\def\PLminus{\Pl_{m}^{-,\l,\eps}}
\def\Uhatplus\widehat{\Uepsplus}
\def\Uhatminus\widehat{\Uepsminus}
\def\FPM{\mathcal F_\l^{\pm,\eps}}
\def\Fplus{\mathcal F_\l^{+,\eps}}
\def\Fminus{\mathcal F_\l^{-,\eps}}
\newcommand{\set}{\@ifstar{\setb}{\seti}}
\newcommand{\seti}[1]{\{#1\}}
\newcommand{\setb}[1]{\left\{ #1\right\}}

\def\Rplus{\mathcal R_{1,b}^{+,\l,\eps}}
\def\Rplusa{\mathcal R_{a,b}^{+,\l,\eps}}
\def\Rminus{\mathcal R_{1,b}^{-,\l,\eps}}
\def\Rminusa{\mathcal R_{a,b}^{-,\l,\eps}}
\def\fmn{f_\eta^{\mu,\nu}}
\def\lmin{\lambda_{\text{min}}}
\def\LL{\mathbb L}
\def\r{k}
\def\reg{K}
\def\ommin{\omega_{\text{min}}}
\def\ommax{\omega_{\text{max}}}
\def\PP{{\bf P}}
\def\FPT{f^{\text{PT}}}
\def\chardelta{\mathbbm{1}_{\J_{R,\de}}}
\def\oom{\overline{\Omega}}
\def\fhatm{\widehat{f}_{0,m}}
\def\sL{\sigma_{\mathcal{L}}} 
\def\C{\mathcal{C}}

\definecolor{ubtgreen}{RGB}{3,138,94}

\allowdisplaybreaks

\begin{document}

\title{On quantitative linear gravitational relaxation}
\author{Mahir Had\v zi\'c}%
\address{Department of Mathematics, University College London, UK}
\email{m.hadzic@ucl.ac.uk}

\author{Matthew Schrecker }%
\address{Department of Mathematics, University of Bath, UK}
\email{mris21@bath.ac.uk}

\begin{abstract}

We prove quantitative decay rates for the linearised Vlasov-Poisson system around compactly supported equilibria. More precisely, we prove decay of the gravitational potential induced by the radial dynamics of this system in the presence of a point mass source. Our result can be interpreted as the gravitational version of linear Landau damping in the radial setting and hence the first  linear asymptotic stability result around such equilibria. 
We face fundamental obstacles to decay caused by the presence of stable trapping in the problem. To overcome these issues we introduce several new ideas.   
We use different tools, including the Birkhoff-Poincar\'e normal form, action-angle type variables, and delicate resolvent bounds to prove a suitable version of the limiting absorption principle and obtain the decay-in-time.
\end{abstract}

\maketitle

\tableofcontents


\section{Introduction}


The basic astrophysical model of galaxy evolution is given by the gravitational Vlasov-Poisson system. 
This system possesses a continuum of steady states representing static compactly supported galaxies whose dynamic stability is one of the 
central topics in the kinetic theory of self-gravitating gases (e.g.~\cite{BiTr}). Since the seminal works of Antonov~\cite{An1961} and Lynden-Bell~\cite{LB1962,LB1967} in the 1960s, the asymptotic stability  of a large family of radially symmetric steady states has been expected. Such a gravitational relaxation mechanism has also been termed gravitational Landau damping~\cite{LB1962,BiTr}  by analogy to the well-known stabilisation mechanism 
from plasma physics~\cite{Landau1946}. The first rigorous mathematical results supporting this claim have appeared only recently due to some central analytic difficulties that distinguish spatially homogeneous and inhomogeneous equilibria. For most of the physically interesting gravitational steady states, the associated Hamiltonian flow in the phase-plane contains an elliptic point (actually a smooth curve of such points), which leads to the phenomenon of \emph{trapping}, a major obstruction to the phase mixing that one hopes will produce decay. 

The main result of this work is the first proof of quantitative decay for the linearised Vlasov-Poisson dynamics around a family of 
nontrivial equilibria, which is also an essential step towards the full nonlinear asymptotic stability theory. We address the fundamental new challenges that arise from the interaction of stable trapping with the resonant structures 
associated with the Hamiltonian geometry of the problem. In addition, frequency analysis is naturally phrased with respect to the angle parametrising the radial motion of the stars inside the steady galaxy,
which leads to the mixing of all frequencies in the effective dispersion relation. These problems cannot be addressed by a straightforward adaptation of available techniques, for example from
the setting of Landau damping around space-homogeneous equilibria or inviscid damping around shear flows.
To handle these difficulties we develop a new approach using a delicate combination of techniques from spectral theory and
dynamical systems, adapted to the specific geometry of this problem. We believe that our theory will be applicable to other problems featuring similar qualitative difficulties.

We consider the 3-d radially symmetric gravitational Vlasov-Poisson (VP) system with a fixed point mass potential (also known as the Kepler potential) at the origin.
The basic unknown is the phase-space density $f(t,r,w,L)$\footnote{In general a phase-space density is called radially symmetric if $f(t,Ax,Av)=f(t,x,v)$ for all rotation matrices $A\in SO(3)$.
This reduces the phase-space to a 3-dimensional space parametrised by coordinates $r=|x|$, $w=\frac{x\cdot v}{r}$, $L=|x\times v|^2$.} 
which solves
\begin{align}
\pa_tf+ w\pa_r f - \pa_r (U(t,r)-\frac{M}{r}+\frac L{2r^2})\pa_w f & = 0, \label{E:VLASOV}\\
\Delta U & = 4\pi \rho, \ \ \lim_{r\to\infty}U(t,r)=0, \label{E:POTENTIAL}\\
\rho(t,r) & = \frac{\pi}{r^2}\int_0^\infty\int_{\mathbb R} f(t,r,w,L)\diff w\diff L.\label{E:DENSITY}
\end{align}

The point mass models the presence of a heavy attractive body at the centre of a galaxy; such potentials play an important role in galaxy dynamics~\cite{BiTr}.  
We point out that global-in-time nonlinear well-posedness for the VP-system, with or without point mass, is well-known and goes back to~\cite{Pf1992,LiPe1991}.


\subsection{Steady states}

Jeans' theorem states that stationary solutions $\bar f$ of the VP-system~\eqref{E:VLASOV}--\eqref{E:DENSITY} 
are in general given as functions of $L$ and
the local particle energy $E$, given by
\begin{align}
E=E(r,w,L) = \frac12 w^2 + \Psi_L(r), \ \ \Psi_L(r) = \bar U(r) -\frac Mr + \frac{L}{2r^2},\label{DEF:EPSIL}
\end{align}
where $\bar U$ is the gravitational potential induced by $\bar f$. $\Psi_L(\cdot)$ is referred to as the {\em effective potential}.

Although not essential, for specificity we
consider a class of {\em polytropic} steady states of the form
\begin{align}\label{E:SS}
\fmn(r,w,L) = \eta \varphi(E,L):= \eta \, (E_0-E)^\mu_+(L-L_0)^\nu_+,
\end{align}
where $0\le \eta\ll1$ is a small constant and we assume
\begin{align}\label{E:SS2}
\mu>2, \ \nu>1. 
\end{align}
The constants $E_0<0$ and $L_0>0$ represent the cut-off energy and cut-off angular momentum. An ansatz of the form~\eqref{E:SS} is
referred to as a polytrope~\cite{BiTr}. The ansatz function $\varphi$ corresponds to a microscopic equation of state;
clearly, the steady state regularity is limited by the exponents $\mu,\nu$. As shown, for example, in~\cite{HRSS2023,Straub24}, for any $\eta>0$
there exists a 1-parameter family of steady states of the  form~\eqref{E:SS}, where the parameter is given by
\begin{align}
\kappa:=E_0-\bar U(0)<0,\label{E:KAPPA}
\end{align}
which represents the relative gravitational potential at the origin.

The steady states $\fmn$  correspond to spherical shells supported away from the origin. 
Their  phase-space support is fixed by the choice of $\kappa$, and we denote it  by 
\beq\label{E:SUPPORT}
\Om:=\supp \fmn\subset [0,\infty)\times \mathbb R\times [0,\infty),
\eeq 
where we neglect the explicit dependence on $\kappa$.
We denote the spatial extent of the steady galaxy
by $[\Rmin,\Rmax]$ and the support of the angular momentum by $[L_0,\Lmax]$.  From~\cite[Proposition 6.2.6]{Straub24}, as $\eta\to0$, the support converges   $(\Rmin,\Rmax)\to(\Rmin^0,\Rmax^0)$, $E_0\to\kappa$, and $\Lmax\to\Lmax^0$, where $0<\Rmin^0<\Rmax^0$ and $\Lmax^0>L_0$ are defined in Lemma~\ref{L:PUREPOINT}.
Moreover, the associated macroscopic density $\bar\rho$ is of characteristic size $O_{\eta\to0}(\eta)$. This and other uniform-in-$\eta$ properties are stated in more detail in Section~\ref{S:AA}. The dependence on $\kappa$ will from now on be omitted 
as we fix $\kappa$ to be any value satisfying the {\em single-gap} assumption 
\begin{align}\label{E:NOGAP}
-2^{-\frac23}\frac{M^2}{2L_0}<\kappa<0.
\end{align}
As shown in~\cite{HRSS2023,Straub24}, assumption~\eqref{E:NOGAP} ensures that the spectrum of the linearised operator around above equilibria possesses a single, so-called principal gap, see Proposition~\ref{P:SPECTRUM}.

To better describe the influence of the steady state regularity on
gravitational relaxation, we introduce a regularity parameter given by
\beq\label{def:N}
N=\max\{n\in\N\,|\,n<\mu+\nu+\frac32\}.
\eeq
Due to~\eqref{E:SS2} we necessarily have $N\ge4$.



\subsection{Main result}


The linearised VP-system takes the form (e.g.,~\cite{HRS2022,HRSS2023})
\be\label{E:VLASOVLIN}
\pa_tf + w\pa_r f - \Psi_L'(r)\pa_wf - \eta w\pa_r U_{\varphi'f} = 0,
\ee
where $U_g$ solves the Poisson equation $\Delta U_g=4\pi \int g\diff v$. Here $\varphi'(E,L):=\pa_E\varphi(E,L)<0$ from the polytropic ansatz~\eqref{E:SS}.
From~\cite{An1961}, the natural Hilbert space for analysis is the weighted space
\begin{align}
\mathscr H : = \Big\{f:\Om\to\mathbb R\,\Big|\, \int_{\Omega} f(r,w,L)^2 |\varphi'(E(r,w,L),L)|\diff(r,w,L)<\infty\Big\},
\end{align}
equipped with the corresponding weighted inner product.  It is not hard to see,~\cite{HRSS2023}, that the
 associated initial value linearised problem can be recast in the form
\begin{align}
\pa_t f + \mathscr L f & = 0, \label{E:FULLLIN}\\
f(0)&=f_0\in\mathscr H, \label{E:FULLLININITIAL}
\end{align}
where the operator $\mathscr L$ factorises into a product of the form
\begin{align}
\mathscr L = \T \LL. 
\end{align}
Here $\mathcal T$ is a skew-adjoint  transport operator induced by the steady state $\fmn$:
\begin{align}
\T & : =  w\pa_r - \Psi_L'(r)\pa_w, \\
D(\T) & : = \{f\in\mathscr H\,\big|\, \T f \ \text{is weakly in } \mathscr H \}.
\end{align}
The self-adjoint bounded operator $\LL:\mathscr H\to\mathscr H$ takes the form
\begin{align}
\LL f = f -  \eta U_{\varphi'f}.\label{E:BBLOPERATOR}
\end{align}

To  speak meaningfully of decay we must choose an initial datum $f_0$ that is orthogonal to the kernel of $\mathscr L$ in a suitable sense. This condition (see Section~\ref{S:HAMILTON})  reads
\begin{align}\label{E:FINORTH}
(f_0,\phi)_{\mathscr H}=0
\ \ \forall \ \phi(E,L)\in\mathscr H.
\end{align}

\begin{theorem}\label{T:MAIN}
Let $\mu,\nu$ satisfy~\eqref{E:SS2}. Then there exists $0<\eta_0\ll1$  such that the following holds. Let $\fmn$ be a steady state satisfying~\eqref{E:SS} with $\eta\in[0,\eta_0)$.
Let $1\leq \n\leq N+1$ and $f_0\in C^{\n}(\overline{\Om})$ be an initial datum satisfying the orthogonality condition~\eqref{E:FINORTH}. Let $f(t,\cdot,\cdot)$ be the 
unique solution to the initial value problem~\eqref{E:FULLLIN}--\eqref{E:FULLLININITIAL}. Then there exists a constant $C=C(\n,\mu,\nu)>0$ such that 
\begin{align}
\|\pa_RU_{|\varphi'|f}(t,\cdot)\|_{C^{\reg-1-b}([\Rmin,\Rmax])} \le \frac {C\|f_0\|_{C^{\reg}(\oom)}}{(1+t)^{b}},  \ \ 0\le b \le \reg-1,
\end{align}
where 
\begin{align}\label{E:REGDEFMAIN}
\reg:= \min\{\mu-1,\nu,\n\}.
\end{align}
\end{theorem}

The proof of this theorem is contained in Section~\ref{S:PROOF}.


\begin{remark}
Since $4\pi \rho_f = \Delta_R U_f + \frac2R\pa_R U_f $, Theorem~\ref{T:MAIN} implies, when $k\geq 2$, that the spatial density also decays-in-$t$ with the bound
\[
\|\rho_{|\varphi'|f}(t,\cdot)\|_{C^{\reg-2-b}([\Rmin,\Rmax])} \le \frac {C\|f_0\|_{C^{\reg}(\oom)}}{(1+t)^{b}},  \ \ 0\le b \le \reg-2.
\]
\end{remark}


\begin{remark}
Our theory works for a general class of equilibria of the form $\eta \varphi(E,L)$, $0\le\eta\ll1$, which are assumed to behave like~\eqref{E:SS} to the leading order
at the vacuum boundary $\{E=E_0\}\cup\{L=L_0\}$, smooth in the interior, and satisfy the Antonov stability condition $\varphi'(E,L)<0$. The latter plays a role analogous to the 
Penrose stability condition in plasma physics~\cite{MoVi2011}.
\end{remark}

\begin{remark}
An easy consequence of our theorem is that there exists a scattering profile $f_\infty\in\mathscr H$ such that 
$f$ scatters to $f_\infty$ as $t\to\infty$ in a suitable sense. For details, see Corollary~\ref{C:SCATTERING}.
\end{remark}


This result offers the first quantitative decay estimates for the full linearised flow~\eqref{E:FULLLIN}--\eqref{E:FULLLININITIAL} around nontrivial equilibria. 
Previously, Rein, Straub, and the authors,~\cite{HRSS2023,Straub24}, showed a non-quantitative version of decay for such steady states, based on absence of eigenvalues and the RAGE theorem. More recently in~\cite{We2025}, the author studied so-called plane-symmetric equilibria  of the gravitational VP-system and proved pointwise decay-in-time of macroscopic quantities (with no rates) for data supported in the absolutely continuous subspace of the linearised operator. In the physics literature, the discussion of damping usually takes place in the framework of the Kalnajs matrix method~\cite{BiTr,Kalnajs1977}, which is tied to ideas from scattering theory -- see for example the numerical work~\cite{BaYa}.

Our result captures several essential difficulties associated with the problem of stability of the Schwarzschild black hole viewed as a solution of the massive Einstein-Vlasov system. In particular, imparting nontrivial mass to a Vlasov perturbation of the Schwarzschild solution is expected to drive the dynamics to a ``nearby" equilibrium, which entails the same conceptual issues, such as elliptic particle trapping and a need for new resolvent estimates.

We note that our result applies to and is new already in the special case of a trivial background $\fmn\equiv0$ (i.e. $\eta=0$), where linearised dynamics collapse to pure transport dynamics.
For the pure transport part of the dynamics, $\pa_t f + \T f=0$, limited algebraic decay rates for gravitational force and density were obtained in~\cite{HRSS2024} allowing data supported near elliptic points and steady states that do not have to feature a central point mass. This pure transport equation was also studied in~\cite{CL2024} with data supported away from elliptic points, a property that is preserved for the pure transport flow but not the full linearised flow.  For this special data, both~\cite{HRSS2024,CL2024} show arbitrary improvement to algebraic decay rates by taking sufficiently smooth data. In~\cite{CL2024} the authors also show that the phase-mixing mechanism for such small data around the trivial equilibrium leads to  existence  for the nonlinear dynamics on an improved time scale longer than the naive scale predicted by the local existence arguments. We also note that certain (non-macroscopic) action-averages of solutions to pure transport are known to  decay in time; this observation essentially goes back to Lynden-Bell~\cite{LB1962}, see also more recent works~\cite{RiSa2020,MoRiVa2022}. 

Theorem~\ref{T:MAIN} shows that the obstructions to decay come from two sources: the regularity of the steady state and the regularity of the initial data. In particular, for arbitrarily smooth steady states, the linearised dynamics 
give
arbitrarily fast decay with sufficiently regular initial data. Surprisingly, this derivative-to-decay relation is consistent with the philosophy of Landau damping in plasmas~\cite{MoVi2011}, where the decay rate improves with the regularity of data. In this gravitational setting, this effect  reflects the 3D nature of the problem as well as the presence of the point mass source.
Our results for the simpler pure transport problem~\cite{HRSS2024} suggest that this regularity-decay improvement mechanism is unlikely to be true for $(1+1)$-D problems (see also~\cite{ChLu2022}) and it is also unclear whether the $t^{-2}$ rate proved in the radial case~\cite{HRSS2024} can be improved in the situation without the point mass.

Much more is known about the {\em orbital} stability of radial equilibria (with or without the point mass potential). Under the Antonov stability condition $\varphi'<0$ (and modulo some technical assumptions),
the equilibria are known to be nonlinearly orbitally stable; we refer the reader to key contributions~\cite{GuRe2007,LeMeRa2012} and references therein.  
However, already at the linear level, orbital stability does not necessarily imply asymptotic stability. In~\cite{HRS2022}, the first author and collaborators proved the existence of steady states for which the linearised operator contains a purely oscillatory eigenvalue, producing oscillations that obstruct decay in the linear dynamics (see also~\cite{HRSS2023,Straub24,MK}). Such oscillatory modes exist only if the steady state is not regular enough at the vacuum boundary~\cite{HRSS2023,Straub24}, which is excluded in our setting by imposing sufficient regularity, as in~\eqref{E:SS2}.  


Nonlinear Landau damping around spatially homogeneous equilibria for the VP-system was shown by Mouhot and Villani in the seminal
work~\cite{MoVi2011}. Following this work, there has been a lot of progress in this subject, see for example~\cite{BeMaMo2016,GrNgRo2020} for shorter proofs.
In the setting of homogeneous equilibria on the whole of $\mathbb R^3$ we refer to~\cite{GlSc94,BeMaMo2018,HKNgRo2021,HKNgRo2022,BeMaMo2022,IoPaWaWi2022,Ng2024}
and, for the case of massless electrons,~\cite{GaIa}. We also mention the stability results~\cite{PaWi2021,PaWiYa2022} in the presence of a repulsive point charge, which make use of action-angle formalism
to prove asymptotic stability of the vacuum steady state for the plasma VP-system. The linearised problem in particular allows for explicit formulas for the action-angles. For further works on stability and scattering
around the vacuum, see~\cite{Sm2016,FlOuPaWi2023}.

By contrast there are very few results about damping around spatially inhomogeneous plasmas, the so-called BGK waves.
Guo and Lin~\cite{GuLi2017} constructed examples of spectrally stable BGK waves with constant ion backgrounds and no embedded eigenvalues in the continuous spectrum.
A nonquantitative damping result for so-called Boltzmannian BGK waves with a trapping region and non-constant ion backgrounds was shown by Despr\'es~\cite{Despres2019}. In a recent work,~\cite{HaMo2024},
by the first author and Moreno,
a nonquantitative version of damping was shown for a large class of equilibria solving the VP-system with non-constant ion backgrounds, using ideas developed in the gravitational case,~\cite{HRSS2023}.
We mention a very nice result by Faou, Horsin, and Rousset~\cite{FaHoRo2021}, wherein the authors prove a  quantitative version of linear Landau damping in a $1+1$-dimensional setting around inhomogeneous equilibria  
of the Vlasov--Hamiltonian Mean Field (HMF) model. This is a toy model for the gravitational interactions allowing for explicit formulas for the action-angle variables, which play an important role in the proof.


\subsection{Strategy of the proof}


A fundamental challenge in the analysis of the linearised decay is the inhomogeneity of the steady galaxies in both $x$ and $v$. This precludes 
any direct use of Fourier techniques in the derivation of the dispersion relation, by contrast to the well-understood setting
of the linear Landau damping around space-homogeneous steady plasmas. We therefore switch to a more geometric frequency variable, the angle $\th\in\mathbb S^1$ which can be used 
together with the conserved energy $E$ and momentum $L$  as a new set of coordinates to parametrise the phase-space, working now for $(E,L)\in\I :=\supp\varphi$. 
In these coordinates, the transport operator $\mathcal T$ is parallel to $\pa_\th$ and takes the form
\begin{align}
\T = \om(E,L) \pa_\th, \quad \om(E,L):=\frac1{T(E,L)},\label{E:TRANSPINTRO}
\end{align}
where $(E,L)\mapsto T(E,L)$ is the period function associated with the periodic trajectory of a star with a given energy-momentum pair $(E,L)$. Passing to the Fourier
variables in $\th$ and letting $\bI:=(E,L)$, the linearised VP-system takes the form
\begin{align}\label{E:LINVPINTRO}
\pa_t \fhat(t,m,\bI) + 2\pi i m\om(\bI)\big(\fhat(t,m,\bI) + \eta\widehat{U_{|\varphi'|f}}(t,m,\bI)\big)=0.
\end{align}
A fundamental difficulty is that the $m$-th Fourier mode of the gravitational potential $\widehat{U_f}(t,m,\bI)$ depends on  the full vector $\langle \fhat(t,\ell,\bI)\rangle_{\ell\in\mathbb Z}$ and
therefore does not separate frequencies, unlike the analogous situation in Landau damping for plasmas around homogeneous steady states.
This is a manifestation of the inhomogeneity of the steady state and highlights an antagonism between the angle coordinate and the gravitational response, which causes many difficulties in
our analysis. 

In order to understand the decay mechanism for~\eqref{E:LINVPINTRO} we must first understand dispersion induced by the leading order dynamics
\begin{align}\label{E:PTINTRO0}
\pa_t f + \mathcal T f = 0,
\end{align}
which we refer to as {\em pure transport}. Solving this equation on the Fourier side, using~\eqref{E:TRANSPINTRO},
 leads to a general solution formula 
\begin{align}\label{E:FPTINTRO}
\fhat(t,m,\bI) = e^{-2\pi i m t \om(\bI)} \widehat{f_0}(m,\bI),
\end{align}
for an initial profile $f_0$. 
For any smooth test function $g:\I\to\mathbb R$ a  simple nonstationary-phase argument then gives decay-in-time of action averages $g\mapsto \int_\I \fhat(t,m,\bI) g(\bI) \diff \bI =  \int_\I e^{-2\pi i m t \om(\bI)} \widehat{f_0}(m,\bI) g(\bI) \diff \bI$,
under the key monotonicity assumption $\om'(\bI)\neq0$~\cite{RiSa2020,MoRiVa2022}.
However, such averages bear unclear  physical relevance, as they do not correspond to  {\em macroscopic} quantities, such as the gravitational field or spatial density.

A second severe limitation of the argument for higher decay-in-$t$ is that it requires the assumption that $f_0$  vanishes at the trapping portion of the boundary $\pa\I$. While such an assumption
is propagated by the pure transport flow~\eqref{E:PTINTRO0}, it is {\em not} propagated by the linearised flow~\eqref{E:FULLLIN}, nor by the full nonlinear flow. In particular, any linear theory that aims to address the
asymptotic convergence to a ``nearby" nontrivial steady state must take into account initial data supported around the elliptic points. 

The first key result of this work is to prove that the gravitational field (and its derivatives) along the solutions of the pure transport~\eqref{E:PTINTRO0} satisfies the decay bound
\begin{align}\label{E:PTDECAYINTRO}
|\pa_R U_{|\varphi'|\FPT}(t,R)| \lesssim \|f_0\|_{C^{\reg}(\oom)} (1+t)^{-\reg}.
\end{align} 
We view this as a remarkable feature of the presence of the point mass potential and simultaneously the 3-D character of the problem; the decay rate improves with the regularity of initial data, assuming for example that the steady state is $C^\infty$ at the vacuum boundary.
This is a subtle effect and stands in contrast to $1+1$-dimensional problems, where the decay bound appears limited irrespective of the initial data/boundary regularity, compare~\cite{HRSS2024,ChLu2022}.

To address the decay of the gravitational force, we first derive a (remarkably succinct) formula that expresses it in terms of action-angle variables:
\begin{align}
\pa_R U_{|\varphi'|\FPT}(t,R) = \frac{4\pi}{R^2} \sum_{m\in\Z_\ast}\frac1m \int_{\I\cap\{E\geq \Psi_L(R)\}} e^{-2\pi mi \om(\bI) t} |\varphi'(\bI)| \fhat_0(m,\bI) S_m(\th(R,\bI))  T(\bI)\diff\bI,
\end{align}
where $\Z_\ast: = \Z\setminus\{0\}$,  we have written 
$S_m(\th) : = \sin(2\pi m\th)$, for $(m,\th)\in\mathbb Z\times\mathbb S^1,$ $\th(R,\bI)$ is the angle written in the mixed $(R,E,L)$ coordinates, and $T(\bI)$ is the period of radial motion for star with a given energy-momentum pair $\bI$. Considering a non-stationary phase argument, recalling $\om'(\bI)\neq 0$, one would like to integrate by parts in $E$ to derive decay. However, with only a little work, one finds that
\[ |\pa_E\th(R,\bI)| \lesssim (E-\EminL)^{-\frac12} (E-\Psi_L(R))^{-\frac12},\]
with higher derivatives picking up stronger singularities.  Here $\EminL$ is the minimum energy attained by the steady state at each $L$. 
Integrating out the singularities using the angular momentum variable $L$, which is a strategy implemented in~\cite{HRSS2024}, 
hits its limit after two integrations-by-parts, thus leading to the decay rate $(1+t)^{-2}$ of~\cite{HRSS2024}, irrespective of the regularity of 
the initial data and the vacuum boundary.

{\em A new foliation.}
To improve upon the above decay rate, one is forced to look for cancellation structures inside the product $\fhat_0(m,\bI) \sin(2\pi m \th(R,\bI))$, which is very challenging 
due to the lack of explicit expressions for the angle function $\th(R,\bI)$. Our first key idea in this direction is to introduce a new foliation of the action support. 
For any given $R\in[\Rmin,\Rmax]$, we consider the map
\begin{align}
\I\cap \{E\ge\Psi_L(R)\} \ni (E,L) \ \mapsto (y,z)\in\J_R, \ \ \ (y,z) = (\om(\bI), E-\Psi_L(R));
\end{align}
see Figure~\ref{F:ELtriangle}.
This map is locally invertible -- to prove this we use the smallness of $\eta\ll1$ in the choice of our steady state. At this stage we use the presence of the 
point mass potential in a fundamental way. In particular, to the leading order in $\eta$, the
frequency function $\om(\bI)$ coincides with the particle frequency in a potential well induced by the pure point mass potential, i.e. $\om(E,L)\sim |E|^{\frac32}$.

\begin{figure}[h!]
\begin{center}
		\begin{tikzpicture}
			\tikzmath{\xscale = 2.5; \yscale = 3.;}
			
			\draw[-{>},black] (0.,-.1*\yscale) -- (0.,1.15*\yscale) node[right] {$L$};
			\draw[-{>},black] (-2.6*\xscale,0.) -- (.3*\xscale,0.) node[below] {$E$};

			\fill [gray, domain=.2:1.,smooth, variable=\x, samples=20, very thin]
			(-2.5*\xscale,.2*\yscale)
			-- plot ({-\xscale*.5/\x}, {\yscale*\x})
			-- (-.5*\xscale,.2*\yscale)
			-- cycle;
			
			\node[white,align=center] at (-.75*\xscale,.4*\yscale) {$\I$};
			
			\draw[ubtgreen, densely dotted, thick] (-.5*\xscale,1.*\yscale) -- (-.5*\xscale,-.05*\yscale) node[below right,ubtgreen,align=center,xshift=-4.5*\xscale] {$E_0<0$};
			\draw[ubtgreen, densely dotted, thick] (-2.5*\xscale,.2*\yscale) -- (.05*\xscale,.2*\yscale) node[right,ubtgreen] {$L_0\geq0$}; 
			\draw[ubtgreen, densely dotted, thick] (-.5*\xscale,\yscale) -- (.05*\xscale,\yscale) node[right,ubtgreen] {$\Lmax$}; 
			
			\draw[magenta,domain=.2:1.,smooth, variable=\x, samples=20, ultra thick]  plot ({-\xscale*.5/\x}, {\yscale*\x});
			
			\node[purple,align=center] at (-1.9*\xscale,.4*\yscale) {$\pa \I_{0}$};
			
			\draw[ubtgreen, ultra thick] (-2.5*\xscale,.2*\yscale) -- (-.5*\xscale,.2*\yscale);
			\draw[ubtgreen, ultra thick] (-.5*\xscale,.2*\yscale) -- (-.5*\xscale,\yscale);
			\node[ubtgreen,align=center] at (-.275*\xscale,.4*\yscale) {$\pa \I_{\text{vac}}$};

			
			\draw[-{>},black] (8.,-.1*\yscale) -- (8.,1.15*\yscale) node[right] {$L$}; 
			\draw[-{>},black] (8+-2.6*\xscale,0.) -- (8+.3*\xscale,0.) node[below] {$E$};
			\fill [lightgray, domain=.2:1.,smooth, variable=\x, samples=20, very thin] 
			(8-2.5*\xscale,.2*\yscale)
			-- plot ({8-\xscale*.5/\x}, {\yscale*\x})
			-- (8-.5*\xscale,.2*\yscale)
			-- cycle;
			\draw[magenta,domain=.2:1.,smooth, variable=\x, samples=20, ultra thick]  plot ({8-\xscale*.5/\x}, {\yscale*\x});

			\node[magenta,align=center] at (8-1.9*\xscale,.4*\yscale) {$\pa \I_{0}$};
			\draw[ubtgreen, ultra thick] (8-2.5*\xscale,.2*\yscale) -- (8-.5*\xscale,.2*\yscale);
			\draw[ubtgreen, ultra thick] (8-.5*\xscale,.2*\yscale) -- (8-.5*\xscale,\yscale);
			\node[ubtgreen,align=center] at (8-.275*\xscale,.4*\yscale) {$\pa \I_{\text{vac}}$};

			\draw[red] (8-1.55*\xscale,.2*\yscale) .. controls +(1,0.95) and +(-.4,0.05) .. (8-0.5*\xscale,.59*\yscale); 
			\draw[red] (8-1.3*\xscale,.2*\yscale) .. controls +(1,1) and +(-.6,0.05) .. (8-0.5*\xscale,.49*\yscale); 
			\draw[red] (8-1.1*\xscale,.2*\yscale) .. controls +(0.08,0.12) and +(-1,0.4) .. (8-0.5*\xscale,.39*\yscale); 
			\draw[red] (8-0.8*\xscale,.2*\yscale) .. controls +(0.1,0.2) and +(-0.1,0.08) .. (8-0.5*\xscale,.3*\yscale); 
			
			\draw[blue] ({8-\xscale*.5/0.25},{\yscale*0.2}) .. controls +(0.05,0.10) and +(-0.05,-0.10) .. ({8-\xscale*.5/0.255},{\yscale*0.255}); 
			\draw[blue] ({8-\xscale*.5/0.3},{\yscale*0.2}) .. controls +(0.1,0.15) and +(-0.1,-0.15) .. ({8-\xscale*.5/0.3},{\yscale*0.3}); 
			\draw[blue] ({8-\xscale*.5/0.35},{\yscale*0.2}) .. controls +(0.1,0.2) and +(-0.1,-0.2) .. ({8-\xscale*.5/0.35},{\yscale*0.35}); 
			\draw[blue] ({8-\xscale*.5/0.4},{\yscale*0.2}) .. controls +(0.1,0.25) and +(-0.2,-0.2) .. ({8-\xscale*.5/0.42},{\yscale*0.42}); 
			\draw[blue] ({8-\xscale*.5/0.46},{\yscale*0.2}) .. controls +(0.1,0.25) and +(-0.2,-0.2) .. ({8-\xscale*.5/0.48},{\yscale*0.48}); 
			\draw[blue] ({8-\xscale*.5/0.54},{\yscale*0.2}) .. controls +(0.1,0.25) and +(0.2,-0.2) .. ({8-\xscale*.5/0.57},{\yscale*0.57}); 
			\draw[blue] ({8-\xscale*.5/0.54},{\yscale*0.2}) .. controls +(0.1,0.25) and +(0.2,-0.2) .. ({8-\xscale*.5/0.57},{\yscale*0.57}); 
			\draw[blue] ({8-\xscale*.5/0.75},{\yscale*0.2}) .. controls +(0.1,0.25) and +(0.2,-0.2) .. ({8-\xscale*.5/0.72},{\yscale*0.72}); 

		\end{tikzpicture}
	\end{center}
	     \caption{Left: the $(E,L)$-support~$\I $ of a spherically symmetric steady state. The magenta and green parts of the boundary of~$\I$ depict~$\pa \I_{0}$ and $\pa \I_{\text{vac}}$, respectively (see~\eqref{E:IBDRYDEF}). Right: schematic depiction of $(y,z)$-level sets for some fixed value of $R\in[\Rmin,\Rmax]$. Blue lines are the $y$-level sets and are almost vertical, the red lines are the $z$-level sets.}
	\label{F:ELtriangle}
\end{figure}
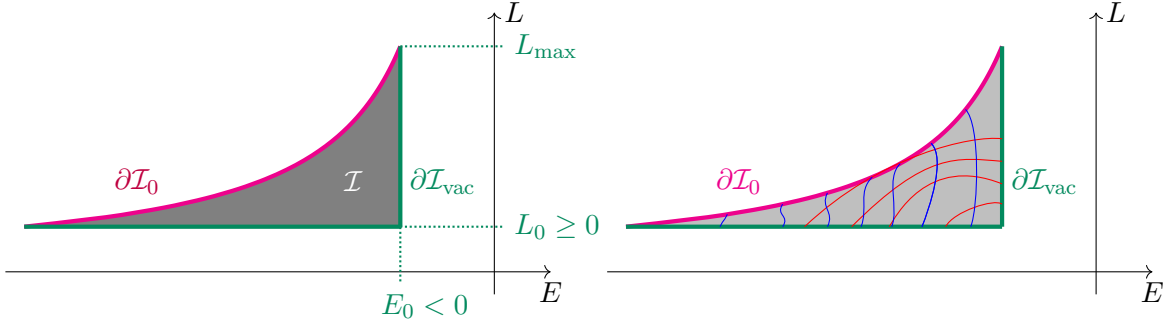


This new coordinate system effectively separates derivatives on the `very bad' singular powers $E-\Psi_L(R)$ from the singularities at the trapping set, $E=\EminL$. This is still insufficient, however, to achieve arbitrarily high decay, and so we need to find a more sophisticated cancellation structure in the product $\fhat_0(m,\bI)S_m(\th(R,\bI))$. To see  this cancellation structure, we need a second ingredient coming from
dynamical systems.

{\em Canonical Normal Form.} 
The Birkhoff-Poincar\'e normal form allows us to state that 
\begin{align}\label{E:CNFINTRO1}
r(\th,\bI)-r_L = G_1(\dE^{\frac12} \cos\th,L), \quad w(\th,\bI) = G_2 (\dE^{\frac12} \sin\th,L),
\end{align}
for some analytic functions $G_1$ and $G_2$, locally around the elliptic point $(r_L,\EminL,L)$, where we have written 
$\dE=E-\EminL$ for short. This identity gives a building block for the Fourier coefficients $\fhat_0(m,\bI)$ for `regular' $f_0$. In particular, we prove that $\fhat_0(m,\bI)$ factorises into the product of $\dE^{\frac{|m|}{2}}$ with a regular function of $E$ and $L$. In addition, inverting~\eqref{E:CNFINTRO1}, we also obtain stronger control over the Green's function, as we derive
\begin{align}\label{E:CNFINTRO2}
\sin(2\pi m \th) = \dE^{-\frac{|m|}{2}} g(r-r_L,w,L)
\end{align}
for some analytic function $g$, locally around $(r_L,0,L)$.  Thus, for $\th=\th(R,\bI)$, we find that the Green's function $S_m(\th(R,\bI))$ takes the form
\begin{align}
\dE^{\frac{|m|}{2}} \sin(2\pi m \th(R,\bI)) =  g(r(\th(R,\bI),\bI)-r_L,w(\th(R,\bI),\bI),L) = g (R-r_L, \sqrt{2z}, L).
\end{align}
This relation is remarkable, as it points to an effective separation of variables: after the change to $(y,z)$ coordinates, the right-hand side
depends on $y$ only through the momentum $L=L(R,y,z)$. In particular, the singular behaviour 
of $\sin(2\pi m \th(R,\bI))$ is well-approximated by  $\dE^{-\frac{|m|}{2}}$, which cancels with the corresponding factor in $\fhat_0(m,\bI)$. 
 In fact, in Proposition~\ref{P:KEY}, we prove a 
quantitative estimate of the form (for $\ell\le\n$)
\begin{align}\label{E:HORINTRO}
\big|\frac{d^\ell}{dy^\ell} \big(\fhat_0(m,\bI) \sin(2\pi m \th(R,\bI))\big)\big| \leq C\|f_0\|_{C^{\n}(\oom)}\big(C^{|m|}|m|^{\ell}\dE^{\max\{\frac{|m|-\ell}{2},0\}}+|m|^{-\n+\ell}\dE^{\frac{\n-\ell}{2}}\big),
\end{align}
which is the basic ingredient leading to~\eqref{E:PTDECAYINTRO}. Bound~\eqref{E:HORINTRO} relies on a splitting of $\fhat_0$ into a polynomial part and a Taylor remainder (see~\eqref{eq:gTaylorsplitting}), about the point $(R,0,r_L)$; this splitting
is responsible for the two contributions on the right-hand side of~\eqref{E:HORINTRO}.  It however leads to an artificial ``loss" of regularity in $R$, which a priori causes difficulties in upgrading~\eqref{E:HORINTRO} to handle higher derivatives in $R$. To get around this difficulty, we make the crucial observation that, for functions such as $\fhat(m,\bI(R,y,z))$, the $\pa_R$ and $\pa_z$ derivatives are parallel, enabling us to exchange $R$ derivatives for $z$ derivatives while still avoiding application of $\pa_z$ to $S_m(R,\bI)$, which would generate further singularities. This necessitates the careful development of a product rule (Lemma~\ref{L:PARPRODUCTSPLITTING}), which due to the algebraic structures in this problem necessitates the use of Bell polynomials to correctly count all the cancellations near trapping. The final argument is then given
in Proposition~\ref{P:KEY}.

{\em Resolvent Estimates.}
Having obtained a satisfactory decay result for the pure transport dynamics, we now discuss the strategy to
obtain quantitative decay for the true linearised problem~\eqref{E:FULLLIN} or equivalently~\eqref{E:LINVPINTRO}.
To that end we prove a version of the limiting absorption principle; our starting point is  Stone's formula, which reads
\begin{align}\label{E:RESOLVENTINTRO1}
\fhat(t,m,\bI) = \lim_{\eps\to0} \int_{\sigma_{\mathcal L}} e^{i\l t}(\widehat{f}_\eps^{-} - \widehat{f}_\eps^+) \diff \l, \text{ where }\widehat{f}^\pm_\eps:=(\L+\la\pm i\eps)^{-1}\widehat{f}_0
\end{align}
 and $\L$ is defined through~\eqref{E:LINVPINTRO}, see~\eqref{DEF:L} below.  Using the Green's function representation of the gravitational force,~\eqref{E:RESOLVENTINTRO1}, and the notation $\Ueps:=U_{|\varphi'| f_\eps^\pm}$, we obtain the following formula (compare Lemma~\ref{L:LAP0}):
\begin{align}
&\pa_R U_{|\varphi'|f}(t,R)  = \pa_RU_{|\varphi'|\FPT}(t,R)\notag\\
& -\frac{2i\eta}{R^2}\sum_{m\in\Z_\ast}\frac1{m}\lim_{\eps\to0} \int_{\sigma(\mathcal L)} \int_{\I} e^{i\la t}   |\varphi'(\bI)|\Big(\frac{\widehat{\Uepsplus}(m,\bI;\lambda)}{\om(\bI)+\frac{\la+ i\epsilon}{2\pi m}}-\frac{\widehat{\Uepsminus}(m,\bI;\lambda)}{\om(\bI) +\frac{\lambda- i\epsilon}{2\pi m}}\Big)S_m(\th(R,\bI)) \diff \bI\dif \lambda. \label{E:FORCEINTRO}
\end{align}

To obtain time-decay, we must integrate-by-parts  in~\eqref{E:FORCEINTRO} with respect to $\l$. It is immediately obvious that such derivatives can either hit the singularity in the singular integral, or they can land on $\Ueps$. Recalling $y=\om(\bI)$, we observe that 
\[ \pa_\la^\ell\big(y+\frac{\la\pm i\eps}{2\pi m}\big)^{-1}=\frac{1}{(2\pi m)^\ell}\pa_y^\ell\big(y+\frac{\la\pm i\eps}{2\pi m}\big)^{-1}\]
in order to integrate by parts again in $y$ and then seek to derive good bounds on $\pa_y^\ell\pa_\la^j \Ueps$. The integration by parts with respect to $y$ produces only vanishing boundary terms due to the structure of the $(y,z)$ coordinates as long as  $\pa_y^\ell \varphi'$  vanishes at the vacuum boundary.

At the heart of the proof of uniform-in-$\eps$ bounds for $\pa_\l^j\Ueps$
is the formula~\eqref{E:PARTIALRUEPSOP}, which gives the important relation
\begin{align}
\FPM[\Ueps] = \mathcal S^\pm[f_0].
\end{align}
Here the frequency-delocalised operator $\FPM$ is given by
\begin{align}\label{E:FORCEINTRO2}
\FPM[V](R) = V(R) -\eta\sum_{m\in\Z_\ast}\frac1m\int_R^{\Rmax}\frac{4\pi}{r^2} \int_{\J_R}\frac{\frac{|\varphi'|}{|\P_R\om|} \widehat V_m S_m}{y+\frac{\la\pm i\eps}{2\pi m}} \diff(y,z) \diff r,
\end{align}
whereas the source operator $\mathcal S^\pm$ acts only on the initial datum $f_0$. We stress that the operator $\FPM$ depends on the {\em full} frequency vector $\langle\Vhat(\ell,y,z)\rangle_{\ell\in\Z_*}$, and therefore
reflects the frequency-mixing feature that is not present in damping around space-homogeneous structures. We show in Proposition~\ref{P:INVERT} that for small steady states $\eta\ll1$ the operator $\FPM:C^1\to C^1$
is invertible, which is a building stone for the inductive procedure leading to bounds on $\pa_\l^j\Ueps$.

The main remaining ingredient is to understand the generalisation of the classical Plemelj operator, which we write 
\begin{align}\label{E:PLEMELJINTRO}
 \int_{\J_R} \frac{Q(R,y,z)\pa_y^j\pa_R^k\big(\Vhat(m,y,z) \sin(2\pi m \th(R,y,z))\big)}{y+\frac{\la\pm i\epsilon}{2\pi m}} \diff (y,z),
\end{align}
which already arises in~\eqref{E:FORCEINTRO} and~\eqref{E:FORCEINTRO2} once we perform integrations by parts.

This is a variant of the classical Plemelj singular integral, which contains potentially dangerous resonant interactions in the frequency region where $|y+\frac{\la}{2\pi m}|\ll1$. Inspired by the approach 
developed in~\cite{HRSS2023}, wherein the absence of embedded eigenvalues for the linearised dynamics was shown, we separate the estimates into a near-resonant region (where $|y+\frac{\l}{2\pi m}|<\delta$) and a non-resonant region. For the more dangerous near-resonant frequencies, we re-write the singular integrand as the derivative of a logarithm, integrating by parts and applying the cancellation structure of~\eqref{E:HORINTRO} to close estimates. By carefully tracking summability in $m$, we finally obtain the key estimates
  contained in Theorems~\ref{T:RESOLVENT1} and~\ref{T:RESOLVENT2},
which for any $\l\in\sigma(\mathcal L)$ and  $b\le\reg-1$ respectively give the a priori bounds
\begin{align}
\|\pa_\l^b\Ueps(\cdot;\l)\|_{C^{\reg-b}([\Rmin,\Rmax])} & \le C\eta |\l|^{-1} \|f_0\|_{C^{\reg}(\oom)}, \\
\|\pa_\l^b(\Uepsplus-\Uepsminus)(\cdot;\l)\|_{C^{\reg-b}([\Rmin,\Rmax])}&\le C\eta |\l|^{-2}  \|f_0\|_{C^{\reg}(\oom)}.
\end{align}
These bounds are then sufficient to pass to the limit in Stone's formula as $\eps\to0$ and conclude the 
decay claimed in Theorem~\ref{T:MAIN}.

{\bf Plan of the paper.} Section~\ref{S:AA} sets the stage for the rest of the paper: we introduce the action-angle variables, period and area functions, and state their basic regularity properties. In
Section~\ref{S:HAMILTON} we explain basic spectral properties of the linearised operator, and in Section~\ref{SS:STONE} we write down
the Stone formula. Section~\ref{S:CNF} introduces the Birkhoff-Poincar\'e normal form and a new foliation of the steady state support. This is then used in Section~\ref{S:REGULARITY} to develop
an effective regularity theory near the trapping set. These two ingredients are used to prove the decay of the gravitational potential along the solutions of the pure transport flow in Section~\ref{S:PT}. 
Finally, Section~\ref{S:RESOLVENTBOUNDS}
is devoted to the uniform resolvent estimates, which rely heavily on the regularity theory developed in Sections~\ref{S:CNF}--\ref{S:REGULARITY}, and a sharp use of 
ellipticity of the Poisson equation in the action-angle variables. Finally, the main theorem is proved in Section~\ref{S:MAIN}. 

\medskip

{\bf Acknowledgments.}
M. Had\v zi\'c acknowledges the support of the EPSRC Early Career Fellowship EP/S02218X/1.
M. Schrecker acknowledges the support of the EPSRC Post-doctoral Research Fellowship EP/W001888/1. 
The authors thank G. Rein and C. Straub for fruitful discussions.


\section{Action-angle variables and the limiting absorption principle}\label{S:AA}


\subsection{Action-angle variables and trapping}

To parametrise the phase-space, we construct the action angle variables associated to the completely integrable particle motion traced by the  stars 
inside the steady state support:
\begin{align}
\dot r & =w, \qquad \dot w  = - \Psi_L'(r),\label{E:RWDOT}
\end{align}
where $\Psi_L$ is defined in~\eqref{DEF:EPSIL} and satisfies the following lemma.

\begin{lemma}\label{lemma:rL}
Let $\mu,\nu$ satisfy~\eqref{E:SS2}. There exists an $\eta_0(\mu,\nu)>0$ such that for all steady states $\fmn$ as in~\eqref{E:SS}, the following holds, uniformly for $\eta\in[0,\eta_0]$.
For each $L\in[L_0,\Lmax]$, the map $r\mapsto \Psi_L(r)\in C^2(0,\infty)$ has a unique, non-degenerate minimum $r_L\in(\Rmin,\Rmax)$ with $\Psi_L'(r_L)=0$, $\Psi_L''(r_L)=\al_L\geq c_0>0$. Moreover, the map $L\mapsto r_L$ is $C^1$ and uniformly monotone increasing. Thus the set $\{r_L\,|\,L\in[L_0,\Lmax]\}=[r_{L_0},r_{\Lmax}]$ is compactly contained within $(\Rmin,\Rmax)$. Finally, as $\eta\to0$, the following convergences hold locally uniformly in $L>0$: $\Psi_L\to\Psi_L^0$ in $C^2$, $r_L\to r_L^0$, where the limiting functions are defined in Lemma~\ref{L:PUREPOINT}.
\end{lemma}

\begin{proof}
The existence, uniqueness, and non-degeneracy of $r_L$ has been shown previously, for example  in~\cite{Straub24}. The convergence as $\eta\to0$ follows from~\cite[Proposition 6.2.6]{Straub24}. From the relation
$\Psi_L'(r_L)=0$ and the at least $C^2$  regularity of $\Psi_L$, we see immediately that $r_L$ is differentiable and
$$\Psi_L''(r_L)\pa_Lr_L+(\pa_L\Psi_L'(r_L))=0,$$
from which we deduce 
$\pa_Lr_L=\frac{1}{\al_Lr_L^3}\geq c_0>0.$
\end{proof}



\subsubsection{Energy-momentum support and notation.}


 The set of particle energy-angular momenta pairs supported by the steady state is compact and corresponds to the set 
 \begin{align}\label{def:I}
 \mathcal I = \cup_{L\in[L_0,\Lmax]}[\EminL,E_0],
 \end{align} 
 where $E_0<0$ is the maximal particle energy and $L=L_0$ the minimal (modulus square of the) angular momentum.
For each $L\in[L_0,\Lmax]$ the value $\EminL=\Psi_L(r_L)$ is the minimal particle energy with a fixed momentum $L$ attained at the critical point $r=r_L$ in the interior of the galaxy, see Figure~\ref{F:PWL}. We note from~\cite[Proposition 6.2.6]{Straub24} that $\EminL\to E_{\min}^{L,0}$ as $\eta\to0$, where $E_{\min}^{L,0}$ is defined in Lemma~\ref{L:PUREPOINT}.

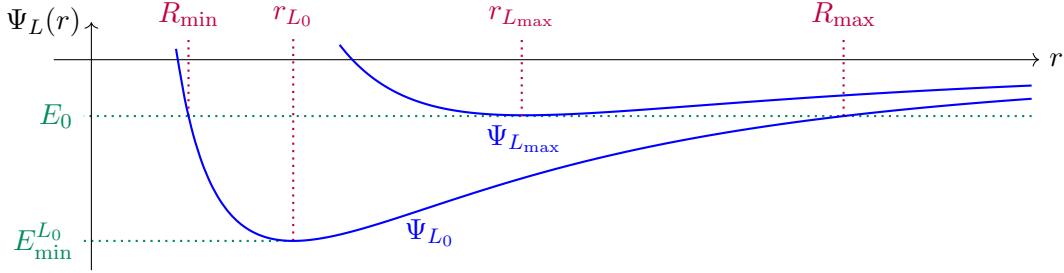
\begin{figure}[h!]
	\begin{tikzpicture}
				
		\draw[purple,  dotted,thick] (2.6833,-2.4112) -- (2.6833,.3) node[above,purple] {$r_{L_0}$};
		\draw[ubtgreen,  dotted,thick] (2.6833,-2.4112) -- (-.1,-2.4112) node[left,ubtgreen] {$\Emin^{L_0}$};
		
		\draw[ubtgreen, dotted,thick] (12.5,-.75) -- (-.1,-.75) node[left,ubtgreen] {$E_0$};
		\draw[purple,  dotted, thick] (1.29099,-.75) -- (1.29099,.3) node[above,purple] {$\Rmin$};
		\draw[purple,  dotted, thick] (10,-.75) -- (10,.3) node[above,purple] {$\Rmax$};
		
		\draw[-{>},black] (-.5,0.) -- (12.6,0.) node[right] {$r$};
		\draw[-{>},black] (0.,-2.8) -- (0.,.5) node[left] {$\Psi_L(r)$};

		\draw[domain=0.225:2.5, smooth, variable=\x, blue, samples=100,thick] plot ({5*\x}, {-4./(1+\x*\x) + .2/(\x*\x)});	
		
		\draw[domain=.66:2.5, smooth, variable=\x, blue, samples=100,thick] plot ({5*\x}, {-4./(1+\x*\x) + 1.3/(\x*\x)});
		\draw[purple,  dotted,thick] (5.7215,-.75) -- (5.7215,.3) node[above,purple] {$r_{\Lmax}$};
		\node[blue] at (4.5,-2.3) {$\Psi_{L_0}$};
		\node[blue] at (5.75,-1.05) {$\Psi_{\Lmax}$};
		
	\end{tikzpicture} 
	\caption{The $L$-dependent family of effective potentials~$\Psi_L$. 
	}
	\label{F:PWL}
\end{figure}

In fact, for the purposes of regularity statements, it is convenient also to define
\beq\label{DEF:A}
\mathbb{A}:=\big\{(E,L)\in(-\infty,0)\times(0,\infty)\,|\,E>\EminL\big\},
\eeq
which includes the action support $\I$ of the steady state with the exception of the curve $E=\EminL$. To any pair $(E,L)\in\A$, we associate the turning points of the flow, $r_-(E,L)<r_+(E,L)$, defined through the condition 
\beq\label{def:rpm(E)}
\Psi_L(r_\pm(E,L))=E,
\eeq 
see Figure~\ref{F:PWL}. Note that these points are defined even outside the support of the steady state and that, for each $L$, as $E\searrow \EminL$, we have $r_\pm(E,L)\to r_L$ (uniformly in $\eta$).

To simplify notation later on, we here define a function
\beq\label{DEF:EDIFF}
\dE=\dE(E,L):=E-\EminL.
\eeq

We further split the boundary of $\I$ into two portions, one along which the energy achieves its minimum (the trapping region) and the other portion corresponding to the vacuum boundary:
\beq\label{E:IBDRYDEF}
\pa\I_0=\pa\I\cap\{E=\EminL\},\qquad \pa\I_{\text{vac}}=\pa\I\cap\big(\{E=E_0\}\cup\{L=L_0\}\big).
\eeq


\subsubsection{The angle coordinate}


To denote the energy-momentum pairs $(E,L)$ inside the support of the steady state, we use the notation
\begin{align*}
\bI=(E,L)\in \mathcal I.
\end{align*}
To each energy-momentum pair $\bI=(E,L)\in \A$, we then associate the period function 
 \begin{equation}\label{E:PERIODDEF}
	T(\bI):= 2\int_{r_-(\bI)}^{r_+(\bI)}\frac{\diff r}{\sqrt{2E-2\Psi_L(r)}},\qquad \bI\in \A.
\end{equation}
We shall see below in Lemma~\ref{lemma:regularity} that $T$ is at least $C^3(\A)$. It follows from the fact that, for each $L$, the Hamiltonian defining the flow is non-degenerate (due to $\Psi_L''(r_L)>0$), that the period function extends continuously onto the curve $E=\EminL$, and hence is defined on the whole support $\I$, see, for example~\cite{Manosas02}. The standard theory of Hamiltonian dynamics also gives the existence of the area function, $A(\bI)$, such that $\pa_EA(\bI)=T(\bI)$:
\beq\label{E:AREADEF}
A(\bI):=2\int_{r_-(\bI)}^{r_+(\bI)}\sqrt{2E-2\Psi_L(r)}\diff r,\qquad \bI\in \A.
\eeq

  We parametrise the motion of the Hamiltonian flow by the angle $\theta\in\mathbb S^1$
defined through 
\begin{equation}\label{thetadef}
  \theta(r,w,L) = \frac1{T(\bI)}\int_{r_-(\bI)}^r\frac{1}{\sqrt{2E(r,w,L)-2\Psi_L(s)}} \diff s, \ \ (r,w,L)\in\Omega.
\end{equation}
When $w<0$, we simply set $\theta(r,w,L)=1-\theta(r,-w,L)$. 

We refer to the set of variables $(\th,\bI)$ as the \emph{action-angle variables}. We note that our definition of action-angle variables
is normalised in such a way that the phase-space volume element is not exactly preserved, but satisfies
instead 
\be
\diff (r ,w,L) = T(\bI) \diff (\theta,\bI). 
\ee


\subsubsection{Period and frequency}


Associated to the period function is the {\em frequency function} 
\begin{align}
\om(\bI):=\frac1{T(\bI)}, \ \ \bI\in\I.
\end{align}
It is shown in~\cite{HRSS2023} that the period function is strictly increasing with respect to $E$ when $\eta\leq\eta_0(\mu,\nu)\ll1$ and is bounded above and below on $\I$. Due to the disparity in significance for the $E$ and $L$ derivatives of the period function, we will consistently abuse notation by writing $'$ for $E$ derivatives. In particular, there exist positive constants $T_{\max}$, $T_{\min}$ and  $c_0$ such that, for all $\bI\in\I$, uniformly in $\eta\in[0,\eta_0]$,
\beq
0<T_{\min} \leq T(\bI)\leq T_{\max}
\eeq
and
\beq
T'(\bI)\geq \frac1{c_0},\ \ \ \om'(\bI)\leq -{c_0}.
\eeq
Thus $\om$ inherits the regularity of $T$ also and,
since $T(\cdot,L)$ is strictly increasing, $\om(\cdot,L)$ is strictly decreasing. Finally, we let
\begin{align}\label{E:OMEGAMINMAX}
\ommin : = \min_{\I} \om, \ \ \ommax:=\max_{\I} \om.
\end{align}


\subsection{Regularity properites}


The regularity of many of the quantities just defined will be crucial in all that follows. A priori, there are two possible obstructions to the regularity of each of these quantities, the first due to the limited regularity of the steady states at the vacuum boundary (when either $E=E_0$ or $L=L_0$) and the second at the trapping set, $\{(E,L)\,\big| E=\EminL\}$. It will turn out below that in fact only the former obstruction is really felt, and so the global regularity properties are constrained only by the exponents $\mu$ and $\nu$ appearing in~\eqref{E:SS}. The following lemma collects several of the key regularity properties in one place.

\begin{lemma}[Regularity lemma I]\label{lemma:regularity}
Let $\fmn$ be a steady state as in~\eqref{E:SS}--\eqref{E:SS2}, inducing a potential $\Psi_L(r)$ and let $\eta\in[0,\eta_0]$, where $\eta_0$ is the small constant introduced in Lemma~\ref{lemma:rL}. Then the potential enjoys the regularity, uniformly in $\eta$,
\beq
\Psi_L\in C^{N+2}(0,\infty),
\eeq
where we recall $N\geq 4$ is defined in~\eqref{def:N}. Moreover, the turning points of the flow satisfy
\beqa
r_\pm\in C^{N+2}(\mathbb A),
\eeqa
and, if $(R,W)$ are the solutions to the ODE system
\beqa\label{eq:RW}
\dot R=W,\ \ \ \dot W=-\Psi_L(R),\ \ \ \ (R,W)(0,\bI)=(r_-(\bI),0),
\eeqa
then also $(R,W)\in C^{N+1}(\R\times\mathbb A)$. Finally, the period function $T$ and area function $A$ satisfy, uniformly in $\eta$,
\beq
T,A\in C^{N+1}(\overline{\mathbb A}).
\eeq
\end{lemma}

\begin{proof}
The regularity claims on $\A$ are the content of \cite[Remark A.3.5]{Straub24}. To extend the regularity properties of the period and area functions up to the trapping set $E=\EminL$, one follows the argument of~\cite[Appendix A.4]{Straub24}.
\end{proof}

We will require much stronger regularity statements than this close to the trapping set (elliptic point). 

\begin{lemma}[Regularity lemma II]\label{lemma:analyticity}
Let the steady state $\fmn$ satisfy~\eqref{E:SS}--\eqref{E:SS2} and let $\eta\in[0,\eta_0]$, where $\eta_0$ is the small constant introduced in Lemma~\ref{lemma:rL}.
\begin{itemize}
\item[(i)] For each $L\in[L_0,\Lmax]$, the effective potential $\Psi_L(r)$ is real analytic on the open interval $(\Rmin,\Rmax)$. In particular, for every $L\in[L_0,\Lmax]$, $\Psi_L(r)$ is jointly analytic in $(r,L)$ on a uniform in $L$ neighbourhood around $(r_L,L)$ and, moreover, $r_L$ is analytic in $L$.
\item[(ii)] The functions $r_\pm(\bI)$ are locally real-analytic for $\bI\in\I\setminus\big(\{(E_0,L_0)\}\cup\pa\I_0\big)$.  The solution $(R,W)$ to the ODE system~\eqref{eq:RW} is locally real-analytic in $(s,\bI)$ for $s\in\R$, $\bI\in\I\setminus\big(\{(E_0,L_0)\}\cup\pa\I_0\big)$.
\item[(iii)] The period function $T(\bI)$ and area function $A(\bI)$ are analytic on $\I\setminus\{(E_0,L_0)\}$. 
\item[(iv)] The functions $(r,w)(\th,\bI)$ are locally analytic in $(\th,\bI)$ for $\th\in\R$, $\bI\in\I\setminus\big(\{(E_0,L_0)\}\cup\pa\I_0\big)$.
\end{itemize}
\end{lemma}

\begin{proof}
(i) We recall (for example from~\cite[Equations (2.2.27), (2.2.29)]{Straub24}; \cite[Lemma A.1]{HRSS2023}) that the density of the steady state is of the form
\beq
\rho_0(r)=\eta r^{2\ell}c_{k,\ell}\big(E_0-\Psi_{L_0}(r)\big)_+^{k+\ell+\frac32}.
\eeq
Observe that the mapping $x\mapsto x_+^{k+\ell+\frac32}$ is analytic except at $x=0$. On the other hand, $\Psi_{L_0}(r)=E_0$ only for $r=\Rmin,\Rmax$. Thus, for any  $r_0\in(\Rmin,\Rmax)$, there exists a neighbourhood $V$ of $r_0$ such that $E_0-\Psi_{L_0}(r)\geq c_0>0$ on $V$, and so locally $\rho_0=g(U_0(r))$ for an analytic function $g$, where we recall $U_0(r)$ is the potential of the steady state. Elliptic regularity theory for the Poisson equation $\Delta U_0=4\pi \rho_0=4\pi g(U_0(r))$  yields that $U_0$ (hence also $\Psi_L(r)$) is locally real analytic around $r_0$. The remaining claims in (i) follow easily.

(ii) We recall first that $r_\pm(\bI)\in(\Rmin,\Rmax)$ for all $\bI\in\I\setminus\{(E_0,L_0)\}$, so that $\Psi_L(\cdot)$ is analytic at $r_\pm(\bI)$ for all such $\bI$. From the defining relation $\Psi_L(r_\pm(\bI))=E$, whenever $E>\EminL$, we have   $\Psi_L'(r_\pm(\bI))\neq 0$  and conclude the claimed analyticity properties of $r_\pm(\bI)$.

We hence deduce the analyticity of the solution to the ODE system~\eqref{eq:RW}: as $(E,L)$ remain constant along the ODE flow, whenever $E>\EminL$ and $(E,L)\neq (E_0,L_0)$, the analyticity of the resulting flow is guaranteed by standard ODE theory for analytic nonlinearities with analytic dependence on parameters.

(iii) The analyticity of the period function on the interior of the domain $\I$ follows now from the identities~\cite[(A.3.8)--(A.3.9)]{Straub24}, which relate the derivatives of the period function to the functions $(R,W)$ whose analyticity was shown in (ii).
The only remaining claim concerns the regularity of the period function up to the line $E=\EminL$ on the support set $\I$. This follows from the properties of period functions for analytic Hamiltonians as in~\cite[Theorem A]{Manosas02}.

(iv) Finally, from the relation $r(\th,\bI)=R(\th T(\bI),\bI)$, $w(\th,\bI)=W(\th T(\bI),\bI)$ and (ii)--(iii) above, we conclude the proof of the lemma.
\end{proof}

In fact, the solution to the ODE system~\eqref{eq:RW} gives an alternative (equivalent) definition of the action-angle change of variables: for $(\th,\bI)\in\S^1\times\I$,
\beq
(r,w,L)(\th,\bI)=(R(\th T(\bI),\bI),W(\th T(\bI),\bI),L).
\eeq
In order to fix a region of $\bI$ on which we may employ the analytic regularity statements above, we define the tubular neighbourhood, for some $\de>0$ to be chosen later,
\beq\label{def:Ade}
\Ade:=\{\bI\in \mathbb{A}\,|\,E\in(\EminL,\EminL+\de),\,L\in[L_0,\Lmax]\}.
\eeq
Similarly, we define 
\beq\label{def:Ide}
\Ide:=\{\bI\in \Ade\,|\,\bI\in\I\},\qquad \Ide^c:=\I\setminus\Ide.
\eeq


\begin{lemma}\label{L:ELNONTRAPPINGREG}
Let the steady state $\fmn$ satisfy~\eqref{E:SS}--\eqref{E:SS2} and let $\eta\in[0,\eta_0]$, where $\eta_0$ is the small constant introduced in Lemma~\ref{lemma:rL}.
 Let $\de>0$, $g\in C^{\n}(\overline{\Om})$ and suppose $j_1,j_2,j_3\in\N$ are such that $j_1+j_2+j_3\leq \min\{\n,N+1\}$. Then there exists $C=C(\de,k,\mu,\nu)>0$,  independent of $\eta\in[0,\eta_0]$, such that, for all $(\th,\bI)\in \S^1\times \Ide^c$,
\beq
\big|\pa_\th^{j_1}\pa_E^{j_2}\pa_L^{j_3}g(\th,\bI)\big|\leq C.
\eeq
\end{lemma}


\begin{proof}
This is direct from Lemma~\ref{lemma:regularity} and the definition of $\Ide$.
\end{proof}



\subsection{Hamiltonian geometry, initial data, and the spectrum} \label{S:HAMILTON}


In the action-angle variables, the pure transport operator $\T$ transforms into $\om(\bI)\pa_\th$.
As a result, the initial value problem~\eqref{E:FULLLIN}--\eqref{E:FULLLININITIAL} takes the form
\begin{align}
\pa_t f+\om(\bI)\pa_\th \left(f+\eta U_{|\varphi'|f}\right)& =0 ,\label{E:LINEARFLOW} \\
f(0) & = \fin,\label{E:LINEARFLOWINITIAL}
\end{align}
where we retain the same notation for $f$ expressed as a function of $(\theta,\bI)$. 
For any $f\in L^2(\mathbb S^1\times \I)$,
we define the corresponding Fourier coefficients
\begin{align}
\widehat f(m,\bI) : = \int_{\mathbb S^1} f(\th,\bI)e^{-2\pi i m \theta}\diff\th, \ \ m\in\Z, \ \bI\in \I.
\end{align}
When no confusion can arise, we shall commonly suppress the other arguments, and simply write
$\fhat_m$. When other subscripts are present, for example with the initial data $f_0$, we place the Fourier mode after a comma, $\fhat_{0,m}$.

By the well-known Antonov coercivity (see e.g.~\cite{HRSS2023}), we have $(f,f)_{\mathscr H}\lesssim (\LL f,f)_{\mathscr H} \lesssim (f,f)_{\mathscr H}$, where $\LL$ is as in~\eqref{E:BBLOPERATOR}. The steady states $\fmn$ are spectrally stable,
and the operator $\mathscr L = \T \LL$ is formally antisymmetric with respect to the distance defined by $\langle f,f\rangle_{\LL} := (\LL f,f)_{\mathscr H}$.
Thus $\text{ran}\,\T $ is contained in the $\langle \cdot,\cdot\rangle_{\LL}$-orthogonal complement of the kernel of $\T \LL$  and it is not hard to see
that they are in fact the same.
Moreover, $\text{ran}\,\T$ is exactly the $\mathscr H$-orthogonal complement of $\text{ker} \, \T$ by~\cite{HRS2022}. 
To speak meaningfully of gravitational damping, we must therefore restrict our initial data to $\text{ran}\,\T$, which is equivalent to the statement that 
the angular average over the invariant torus $\mathbb S^1$
vanishes, i.e. 
\begin{align}\label{E:AVERAGEZERO}
\widehat f_0(0,\bI) = \int_{\mathbb S^1}\fin(\theta,\bI) \diff \theta = 0, \ \ \bI\in \I.
\end{align}
By integrating~\eqref{E:LINEARFLOW} over $\mathbb S^1$, it is easily seen that the property~\eqref{E:AVERAGEZERO} is dynamically conserved along the flow, i.e.
$\widehat f(t,0,\bI) = 0$, for all $t\ge0$,  $\bI\in \I$, so
$\text{ran}\,\T$ is an invariant subspace for the linearised flow.

A trivial but important consequence of the assumption~\eqref{E:AVERAGEZERO} is that the total mass $M[F_0]$ of the perturbation associated to $F_0$ vanishes, i.e.
\begin{align}\label{E:ZEROMASS}
M[f_0]: = \int_\Omega f_0 \diff(r,w,L) = \int_{\mathbb S^1\times \I} f_0 T \diff(\theta,\bI) = 0.
\end{align}


\begin{prop}\label{P:SPECTRUM}
Let $\mu,\nu$ satisfy~\eqref{E:SS2}. There exists $\eta_0(\mu,\nu)>0$ (which we take less than or equal to that of Lemma~\ref{lemma:rL}) such that, for all $\fmn$ satisfying~\eqref{E:SS} and $\eta\in[0,\eta_0]$,  the associated linearised operator $\mathscr L: D(\mathscr L)=D(\T)\to \H$ is skew-adjoint and its spectrum is purely essential, i.e.
\begin{align}
\sigma(\mathscr L)=\sigma_{\text{ess}}(\mathscr L)=\sigma(\T) = i(-\infty,-\lmin] \cup\{0\}\cup i[\lmin,\infty),
\end{align}
where, recalling~\eqref{E:OMEGAMINMAX} we let
\be\label{E:LAMBDAMINDEF}
\lmin := 2\pi \min_{\bI\in\I}\om(\bI) = 2\pi \ommin.
\ee
Here $0$ is an eigenvalue of infinite multiplicity.
\end{prop}


\begin{proof}
The action-angle representation of the operator $\mathscr L$ allows us to compute the essential spectrum; this result is contained in~\cite{HRSS2023,HRS2022}. Moreover, the strategy developed in~\cite{HRSS2023} was generalised to this 3-D setting in~\cite[Theorems 6.4.1 and 6.5.5]{Straub24} to show that $0$ is the only eigenvalue, and there are in particular no embedded eigenvalues.
\end{proof}



\subsection{Frequency formulation and Stone's formula}\label{SS:STONE}

By taking the Fourier transform in $\th$, system~\eqref{E:LINEARFLOW}--\eqref{E:LINEARFLOWINITIAL} transforms into
\begin{align}
&\pa_t\widehat{f} (m,\bI)+2\pi im\om(\bI)\Big(\widehat{f}(m,\bI)+\eta \widehat{U_{|\varphi'|f}}(m,\bI)\Big) =0, \ \ (m,\bI)\in\Z_*\times \I \label{E:LINFREQ1}\\
&\widehat f(0,m,\bI) = \widehat \fin(m,\bI), \ \ (m,\bI)\in\Z_*\times \I,\label{E:LINFREQ2}
\end{align}
where we restrict to  $m\in\Z_*$ due to~\eqref{E:AVERAGEZERO}.
We remark here that unlike the corresponding situation for so-called shear flows in fluid mechanics (see e.g~\cite{BeMa2015,Jia20}), the frequencies do not separate i.e.~the nonlocal term $\widehat{U_{|\varphi'|f}}(m,\bI)$ depends on the full ensemble $\fhat(\bI)=\langle\widehat f(m,\bI)\rangle_{m\in\Z_*} \in\ell^2(\Z_*)$.
We interpret the linearised system~\eqref{E:LINFREQ1}--\eqref{E:LINFREQ2} as a flow on the Hilbert space
$$
\mathcal{H}=\Big\{\fhat:\Z_*\times \I\to \mathbb C\,|\,\|\fhat\|_{\mathcal{H}}=\sum_{m\in\Z_*} \int_\I|\varphi'(\bI)||\fhat(m,\bI)|^2\diff \bI<\infty\Big\}.
$$
For $\ghat\in\mathcal{H}$, we define the operator $\L:D(\L)\to \mathcal H$ as
\begin{align}
(\L \ghat)(m,\bI)&= 2\pi m\om(\bI)\Big(\ghat(m,\bI)+\eta \widehat{U_{|\varphi'|g}}(m,\bI)\Big), \label{DEF:L}\\
D(\L) = \H^1 & := \Big\{\fhat:\Z_*\times \I\to \mathbb C\,|\,\|\fhat\|_{\mathcal{H}^1}=\sum_{m\in\Z_*} \int_{\I}|m|^2|\varphi'(\bI)|\fhat(m,\bI)|^2\diff \bI<\infty\Big\}.
\end{align}
Here $\widehat{U_{g}}$ is the $\th$-Fourier transform of the gravitational potential generated by $\mathcal{F}^{-1}\ghat$, i.e.
\beq
\Delta U_{g}=\frac{4\pi^2}{r^2} \int (\mathcal{F}^{-1}\ghat)(r,w,L)\,\diff (w,L), \ \ \lim_{r\to\infty} U_g(r) = 0;
\eeq
(note the change of coordinates from $(\th,\bI)$ to $(r,w,L)$ following the inversion of the Fourier transform). It is not hard to
check that $\mathcal L:\mathcal H^1\to\mathcal H$ is a self-adjoint operator
with respect to the distance
\begin{align}
(\fhat,\ghat)_L : = \int_{\I}|\varphi'(\bI)|(\LL\fhat,\ghat)_{\ell^2(\Z_*)}\diff\bI = \int_{\I}|\varphi'(\bI)|\sum_{m\in\Z_*}(\LL\fhat)(m,\bI) \overline{\ghat(m,\bI)}\diff \bI .
\end{align}

The linearised problem~\eqref{E:LINFREQ1}--\eqref{E:LINFREQ2} can now be recast in the form
\beq\label{E:FEQN}
\pa_t\widehat{f}+i\L\widehat{f}=0, \ \ \widehat f(0) = \widehat\fin\in \text{ran}\,\T.
\eeq
Therefore, by Stone's theorem the operator $\mathcal L\big|_{\text{ran}\,\T}$ generates a strongly continuous semigroup $e^{-i t \mathcal L}: \mathcal H^1\cap\text{ran}\,\T\to\mathcal H\cap\text{ran}\,\T$
and the following fundamental formula holds:
\beq\label{eq:spectral}
\widehat{f}(t,m,\bI)=\frac{1}{2\pi i}\lim_{\eps\to0+}\int_{\sigma_{\L}}e^{i\la t}\Big(\big(\L+\la-i\eps)^{-1}-\big(\L+\la+i\eps)^{-1}\Big)\widehat{f}_{0}(m,\bI)\,\diff \la,
\eeq
where we let
\beq\label{E:SIGMALDEF}
\sigma_{\L}:=\sigma(\L\big|_{\text{ran}\,\T})=(-\infty,-\lmin] \cup [\lmin,\infty).
\eeq



\begin{lemma}[Representation formula for the gravitational potential]\label{L:REPRESENTATIONS}
To any $f\in L^1(\Om)$ we associate the gravitational potential $U_f:[0,\infty)\to\mathbb R$ given as the unique solution of~\eqref{E:POTENTIAL}--\eqref{E:DENSITY}, where 
$\rho(R)=\frac{\pi}{R^2}\int_{\Om}f(R,w,L)\diff (w,L)$ is trivially extended by zero outside its spatial support $[\Rmin,\Rmax]$. The gravitational force $\pa_R U_f$
satisfies the formula
\begin{align}
\pa_R U_f(R) & = \frac{4\pi^2}{R^2} \int_{\Om} f(s,w,L) h_R(s)\diff(s,w,L), \label{E:PARTIALRUFFORMULA}
\end{align}
where the Green's function $ h_R $ is given by the characteristic function
\begin{align}
h_R(s) & : = \chi_{s\le R}(s) \label{E:HRDEF}.
\end{align}
Moreover, for any  $(m,R)\in\Z_\ast\times[\Rmin,\Rmax]$, we have the identity
\begin{align}\label{E:GREENFOURIER}
\widehat{h_R}(m,\bI) = \frac1{\pi m} \sin(2\pi m \th(R,\bI))\chi_{E\geq \Psi_L(R)}(\bI), \ \ \bI\in \I,
\end{align}
where $\th(R,\bI)\in[0,\frac12)$ is the unique angle with the property $r(\th(R,\bI),\bI)=R$ for each $\bI\in\I$.
\end{lemma}

\begin{proof}
We integrate~\eqref{E:POTENTIAL} over $[\Rmin,R)$ to get 
\begin{align}\label{E:UPOT}
\pa_RU_f(R) = \frac{4\pi^2}{R^2} \int_{\Rmin}^R \int f(s,w,L) \diff (w,L) \diff s, 
\end{align}
which is precisely~\eqref{E:PARTIALRUFFORMULA}. Formula~\eqref{E:GREENFOURIER} is shown in~\cite[Lemma 2.2]{HRSS2024}.
\end{proof}


\begin{remark}\label{R:BDRYPOT}
Since $U(r)=\frac cr$ for all $r\ge\Rmax$ and some constant $c\in\mathbb R$, it follows from~\eqref{E:UPOT}
that 
\[
c= - 4\pi^2 \int_{\Rmin}^{\Rmax} \int f(s,w,L) \diff (w,L)\diff s.
\]
\end{remark}


\subsubsection{Function spaces}\label{SS:SPACES}


For any $j\in\mathbb N_0$ we denote by $C^j(\oom)$ the usual space of $j$-times continuously differentiable functions equipped with the usual norm. If a complex-valued function $V$ is defined on $[\Rmin,\Rmax]$
we will abbreviate $\|V\|_{C^j_R} = \|V\|_{C^j([\Rmin,\Rmax])}$.

In order to develop an effective limiting absorption principle, we introduce notation for function spaces with mixed-derivative regularity in $\la$ and $(r,w,L)$, which we call $\C^k(\overline{\Om}\times\sL)$. 
These spaces are defined via the following norms:
\begin{align}\label{E:CKDEF}
\|F\|_{\C^k(\oom\times\sL)}&:=\sum_{j=0}^{k-1}\sum_{|\al|\leq k-j}\|\pa_\la^jD^\al F(\cdot;\la)\|_{L^\infty_\la C^{0}(\oom)}  
\end{align}
for multi-indices $\al$ in $(r,w,L)$ and $k\in\mathbb N$. When $F:[\Rmin,\Rmax]\times\sL\to\mathbb{C}$, we  write
\begin{align}
\label{E:CKRDEF}
\|F\|_{\C^k_R([\Rmin,\Rmin]\times\sL)} &:=\sum_{j=0}^{k-1}\sum_{\ell\leq k-j}\|\pa_\la^j\pa_R^\ell F(\cdot;\la)\|_{L^\infty_\la C^{0}([\Rmin,\Rmax])} 
\end{align}
and we typically abuse notation by suppressing the domain and simply writing $\mathcal{C}^k_R$.



\section{Canonical Normal Form and a new foliation of the action support}\label{S:CNF}


In order to understand the dynamics driven either by the full linearised operator or the pure transport operator near the elliptic (trapping) set, it is essential to gain a better understanding of the structure of the action-angle change of variables close to this set.  

{\em Standing assumption.} We assume throughout this section that we are working with the Hamiltonian flow induced by a steady state $\fmn$ satisfying~\eqref{E:SS}--\eqref{E:SS2}. Moreover, we assume that $\eta\in[0,\eta_0]$, where $\eta_0$ is the sufficiently small constant given by Proposition~\ref{P:SPECTRUM}.


\subsection{Canonical Normal Form}


In this section we introduce the Birkhoff-Poincar\'e normal form in order to derive precise regularity statements close to the elliptic point. We first recall that, foliating the support of the steady state by surfaces of constant $L$, the characteristic flow, which leaves $L$ invariant, is the Hamiltonian flow given  by 
\beqa\label{eq:Hamiltonian}
\dot r = w,\quad
\dot w =-\Psi_L'(r),
\eeqa
with Hamiltonian
 \beq
H(r,w,L)=\frac12 w^2+\Psi_L(r)-\EminL=E(r,w,L)-\EminL.
\eeq
From Lemma~\ref{lemma:rL}, it is easy to see that, for each $L$, the point $(r_L,0)$ is an isolated, non-degenerate elliptic point, uniformly in $L$. 

The goal of this section is the following theorem, which provides an analytic change of variables on a suitable region to put the Hamiltonian system into normal form. In order to ensure that there is no degeneracy of the radius of analyticity as $L\to\Lmax$, we consider the Hamiltonian system defined on the whole cylinder
\[\Om'=\{(r,w,L)\,|\,E(r,w,L)\leq \min\{\Psi_L(\Rmin),\Psi_L(\Rmax)\},\,L\in[L_0,\Lmax]\}. \]
We note in particular that $r\in[\Rmin,\Rmax]$ on $\Om'$ and, for each $\bar L>L_0$, we have $\overline{\Om\cap\{L=\bar L\}}\subset \textup{int}(\Om'\cap\{L=\bar L\})$ (uniformly when $\bar L$ is uniformly away from $L_0$). In the action-angle variables, this induces a corresponding neighbourhood $\I'$ of $\I\cap\A$ within $\mathbb{A}$. We recall the definitions~\eqref{def:I}--\eqref{DEF:A} of $\I$ and $\mathbb{A}$ so that $E>\EminL$ on $\I'$.

\begin{theorem}\label{T:NORMALFORM}
There exists a map $\Gamma\in C^{N+1}(\Omega';\R^3)$, diffeomorphic onto its image and bi-analytic on $\Omega'\cap\{(r,w,L)\,|\,r\in(\Rmin,\Rmax)\}$, that transforms the Hamiltonian flow~\eqref{eq:Hamiltonian} into the flow of the Hamiltonian $\widetilde{H}(u,v,L)=(A(\cdot,L))^{-1}(\pi(u^2+v^2))$, where the area function is given by~\eqref{E:AREADEF}.

Moreover, in action-angle variables, the map $\Gamma$ induces a map $\ga:\mathbb{S}^1\times\I'\to\Gamma(\Omega)$ which satisfies the identity
\beq\label{eq:gammazero}
\gamma(\th,\bI)=\sqrt{A(\bI)/\pi}\big(\cos(2\pi\th),\sin(2\pi\th)\big)
\eeq
and $\ga\in C^{N+1}(\mathbb{S}^1\times\I')$ and $\ga$ is analytic on $(\mathbb{S}^1\times\I')\cap\{(\th,\bI)\,|\,r(\th,\bI)\in(\Rmin,\Rmax)\}$.
\end{theorem}

Clearly the regularity properties of $\ga$ are inherited from those of $A$, and so can be taken uniformly in $\eta$ by Lemma~\ref{lemma:regularity}. The remainder of this subsection will be devoted to the proof of this theorem which will follow directly from Lemmas~\ref{L:CNF}--\ref{lemma:twist}.

From Lemma~\ref{lemma:analyticity}, the Hamiltonian $H$ is locally analytic around the curve of elliptic points $(r_L,0,L)$, so that also the area and period functions $A(\bI)$ and $T(\bI)$ are analytic on a uniform neighbourhood $\{\bI\,|\,E\in[\EminL,\EminL+\de],\,L\in[L_0,\Lmax]\}$ of the curve $(\EminL,L)$ by, for example,~\cite{Manosas02}. Observe that this neighbourhood remains inside the closure of $\I'$ and its image in $(r,w,L)$ coordinates remains inside $\Om'$.

We may therefore derive the change of variables to Birkhoff-Poincar\'e Canonical Normal Form (CNF), \cite[Theorem A]{Manosas02} (observe that the normal form change of coordinates may also be taken to depend continuously on $\eta$). More precisely, we obtain the analytic map 
\beq\label{def:Gamma}
\Ga(r,w,L)=(G_1(r,w,L),G_2(r,w,L),L)
\eeq
 from a tubular neighbourhood of $(r_L,0,L)$,  analytic with respect to $(r,w,L)$, into a tubular neighbourhood of $(0,0,L)$ in $(u_1,u_2,L)$-space such that the image of the Hamiltonian flow under this map is the flow 
\beqa\label{eq:CNFflow}
\dot u_1=&\,-u_2h(u_1^2+u_2^2,L),\quad
\dot u_2 = u_1 h(u_1^2+u_2^2,L),\quad
\dot L =  0,
\eeqa
where $h(\xi)=2\pi (A^{-1})'(\pi \xi)$. \cite[Theorem A]{Manosas02} also gives $A(E(r,w,L),L)=\pi\big(G_1(r,w,L)^2+G_2(r,w,L)^2\big)=\pi (u_1^2+u_2^2)$. The Hamiltonian of the normal form system~\eqref{eq:CNFflow} is given by
\beq
\widetilde H(u_1,u_2,L)=\frac12\int_0^{u_1^2+u_2^2}h(\xi,L)\,\dif \xi=(A(\cdot,L))^{-1}(\pi(u_1^2+u_2^2)),
\eeq
where the inverse function is defined for each fixed $L$.
Moreover, there exist analytic maps $\phi$ and $\tilde\phi$ with $\phi'(0),\tilde\phi'(0)\neq 0$ such that
\beq\label{eq:dEinuv}
\dE=H(r,w,L)=\phi(\widetilde H(\Ga(r,w,L)))=\tilde\phi(G_1(r,w,L)^2+G_2(r,w,L)^2).
\eeq
To show that the change of coordinates is bi-analytic, using~\eqref{eq:Hamiltonian},~\eqref{def:Gamma}, and the chain rule it is easy to see that 
\begin{align}
\dot u_1 \pa_w G_2 - \dot u_2 \pa_w G_1 = w (\pa_r G_1\pa_w G_2-\pa_rG_2\pa_wG_1) = w \det(D\Gamma). 
\end{align}
On the other hand, from~\eqref{eq:CNFflow} we can rewrite the left-hand side above as $-h(u_1^2+u_2^2)(u_2\pa_wu_2+u_1\pa_wu_1)$. Putting these two observations together we obtain the identity
\begin{align}
w \det(D\Gamma) = -\frac12 h(u_1^2+u_2^2)\pa_w(u_1^2+u_2^2).
\end{align}
Since $u_1^2+u_2^2=\frac1{\pi}A(E(r,w,L),L)$, it follows that $\frac1{w}\pa_w(u_1^2+u_2^2) = A'(\bI)\frac1{\pi w}\pa_w E(r,w,L)=\frac{1}{\pi}T(\bI)\neq0$.
In particular, as  $h(0)\neq0$, the determinant $D\Gamma(r_L,0,L)$ is non-degenerate and we may apply the inverse function theorem to conclude that there exists 
an inverse map $\Gamma^{-1}(u_1,u_2,L)=((\Gamma^{-1})_1(u_1,u_2,L),(\Gamma^{-1})_2(u_1,u_2,L),L)$ which is locally analytic in a neighbourhood of the curve $\{(0,0,L)\,|\,L\in[L_0,\Lmax]\}$, where  the diameter of this neighbourhood, restricted to the $(u_1,u_2)$-plane, may be taken bounded below, uniformly in $L$.

The analytic change of variables into CNF is rigid in the sense that it is unique up to a twist of the form
\beq\label{DEF:TWIST}
(\varrho,\vartheta)\mapsto (\varrho,\vartheta+F(\varrho^2)),
\eeq
where $(\varrho,\vartheta)$ are polar coordinates in the $(u_1,u_2)$ plane,~\cite[Proposition 4.4]{Manosas02}. We call any such change of coordinates an \textit{analytic twist}.

As in the statement of Theorem~\ref{T:NORMALFORM}, we define the map $\ga$, which is $\Ga$ expressed via action-angle variables:
\beq\label{def:gamma}
\Ga(r(\th,\bI),w(\th,\bI),L)=:\ga(\th,\bI)=(g_1(\th,\bI),g_2(\th,\bI),L).
\eeq

\begin{lemma}\label{L:CNF}
The map $\ga$, defined in~\eqref{def:gamma}, takes the form
\beqa
g_1(\th,\bI)=&\,a(\bI)\cos(2\pi\th)+b(\bI)\sin(2\pi\th),\\
g_2(\th,\bI)=&\,a(\bI)\sin(2\pi\th)-b(\bI)\cos(2\pi\th), \label{E:G1G2}
\eeqa
where $a(\bI)$ and $b(\bI)$ are analytic functions of $\dE^{\frac12}$, vanishing at $\dE=0$.
\end{lemma}

\begin{proof}
Since $u_i=G_i(r(s),w(s),L(s))$, $i=1,2$,  applying the $s$-derivative to this relation we conclude from~\eqref{eq:Hamiltonian} and the definition of $\T$ that
\[
\dot u_i(s) = \T G_i = \om(\bI)\pa_\th g_i, \ \ i=1,2.
\]
Therefore from~\eqref{eq:CNFflow},
\[
\pa_\th g_1 = - g_2 \frac{h(\pi^{-1}A(\bI))}{\om(\bI)}, \ \ \pa_\th g_2 = g_1 \frac{h(\pi^{-1}A(\bI))}{\om(\bI)}.
\]
Since $A'(\bI)=T(\bI)$ and $h(\xi) = 2\pi A^{-1}(\cdot,L)'(\pi\xi)$, it follows that $h(\pi^{-1}A) = \frac{2\pi}{A'(\bI)} =2\pi \om$. Plugging this into the above system, we see that the $g_i$ are harmonic oscillators in $\th\in\mathbb S^1$ for any fixed $\bI\in\I$ and thus~\eqref{E:G1G2} follows.
 To see that the coefficients $a(\bI)$ and $b(\bI)$ expand analytically in half powers of $\dE$, we note that 
 $g_i(0,\bI) = G_i(r_-(\bI),0,L)$, $i=1,2$. Since each $G_i$ is analytic and $r_-(\bI)$ expand analytically in  half powers of $\dE$ (this is easy to see from the analyticity of $\Psi_L$ and the defining relation~\eqref{def:rpm(E)}), the claim follows.
\end{proof}

To fix notation, for the coefficients $a(\bI)$ and $b(\bI)$ of Lemma~\ref{L:CNF}, we set
\beq
a(\bI)=\sum_{j=0}^\infty a_j(L)\dE^{\frac j2},\qquad b(\bI)=\sum_{j=0}^\infty b_j(L)\dE^{\frac j2},
\eeq
where $a_0=b_0=0$. As the area function $A$ is analytic, we write 
\beq
A(\bI)=\sum_{j=0}^\infty A_j(L)\dE^{\frac j2},
\eeq
where again $A_0=0$ and also $A_{2k-1}=0$ for every $k\in\N$.

\begin{lemma}\label{lemma:twist}
There exists an analytic twist function $F(\rho^2)$ as in~\eqref{DEF:TWIST} such that, after applying the twist, we may take the function $b(\bI)\equiv 0$ and the coefficients $a_{2j}(L)=0$ for all $j\geq 1$. In particular, $a(\bI)=\sqrt{A(\bI)}$ and  we obtain that $\ga$ is given explicitly on $\Ade$ by 
\beq\label{eq:gamma}
\ga(\th,\bI)=\Big(\sqrt{A(\bI)/\pi}\cos(2\pi\th),\sqrt{A(\bI)/\pi}\sin(2\pi\th),L\Big).
\eeq
Thus $\ga$ extends to the neighbourhood $\I'$, is analytic for $\{(\th,\bI)\,|\,r(\th,\bI)\in(\Rmin,\Rmax)\}$ and $\ga\in C^{N+1}(\S^1\times\I')$.
\end{lemma}

We note that, transferring the regularity statements back into the map $\Ga$, this concludes the proof of Theorem~\ref{T:NORMALFORM}.

\begin{proof}
Note first that, under the map $\Gamma$, the boundary of the analytic, open lower half-plane, which is the axis $(r,0)=\{\th=0\}\cup\{\th=\frac12\}$, is mapped to the concatenation of the two curves
\beq\label{axisimage}
\ell:=\{(a(\bI),-b(\bI))\}\cup\{(-a(\bI),b(\bI))\}
\eeq
which is analytic due to the bi-analyticity of $\Gamma$. 

We note that any analytic function of $\varrho^2=(u_1^2+u_2^2)$ is also an analytic function of $\dE$, where we consider $\dE$ as defined on the $(u_1,u_2)$-plane by $\dE=\tilde\phi(\varrho^2)$ from~\eqref{eq:dEinuv}. We therefore define the twist map $F$ inductively as an analytic function $\tilde F(\dE)$. 
We set the first iteration of the twist map as the constant
$$\tilde F^{(0)}(\dE) =  \tilde F_0,$$
where $\tilde F_0$ is chosen such that the twist applies a constant rotation in $(u_1,u_2)$ mapping the tangent to $\ell$ at the origin to the $u$-axis. Applying this twist leaves us in an analytic CNF and alters $b(E)$ to set $b_1=0$. The constraint $\pi(a(\bI)^2+b(\bI)^2)=A(\bI)$ then implies $a_1=(\pi^{-1}A_2)^{\frac12}\neq 0$ and so, from the analyticity of~\eqref{axisimage}, it is then straightforward to see $b_2=0$ by symmetry considerations.

Suppose now that the twist $\tilde{F}^{(j-1)}(\dE)$ has been defined  so that we have $b_1,\ldots, b_{2j}=0$. Observe that the angle made by the image of the half axis $(r_-(E),0)$ with the $u$ axis in the $(u,v)$-plane is given by
\beq
\arctan\Big(-\frac{b(\bI)}{a(\bI)}\Big)=-\frac{b(\bI)}{a(\bI)}+O\Big(\frac{b(\bI)^3}{a(\bI)^3}\Big)=-\frac{b_{2j+1}(\dE)^{j}+O(\dE)^{j+\frac12}}{a_1 + O(\dE)^{\frac12}} + O\big(\dE\big)^{3j}.
\eeq
We therefore update the twist function as
$$\tilde F^{(j)}(\dE)= \tilde F^{(j-1)}(\dE)+\frac{b_{2j+1}}{a_1}\dE^{j},$$
which ensures that $b_{2j+1}=0$ and makes no change to the coefficients of $a$ and $b$ of orders less than $2j$. Again from~\eqref{axisimage} and symmetry considerations, we deduce $b_{2j+2}=0$.

In total, we therefore arrive at the twist 
\beq
\tilde F(\dE)=\tilde F_0+\sum_{j=1}^\infty \frac{b_{2j+1}}{a_1}\dE^j,
\eeq
which converges as an analytic function of $\dE$ due to the convergence of the power series for $b$, and hence is also analytic as a function of $\rho^2$ due to~\eqref{eq:dEinuv}.

Moreover, the coefficients $a_{2j+1}$ are all specified, as $b\equiv 0$, by the relation $\pi a(\bI)^2=A(\bI)$.
The specific form~\eqref{eq:gamma} of $\ga$ now follows directly from the relation $\pi (g_1^2+g_2^2)=A$, and the remaining claims of the lemma follow from this formula and Lemmas~\ref{lemma:regularity} and~\ref{lemma:analyticity}.
\end{proof}


\subsection{Fourier expansions}


The goal of this section is to establish the following proposition relating the phase-plane variables $(r,w,L)$ to the action-angle variables $(\th,\bI)$.


\begin{prop}\label{prop:rwexpansions}
There exists $\de>0$ such that, for each $\bI\in\Ade\cap\I'$, defined above Theorem~\ref{T:NORMALFORM}, we have the action-angle expansions
\beqa\label{eq:rwexpansions}
r(\th,\bI)-r_L=&\,r_0(\bI)\dE+\sum_{m=1}^\infty r_m(\bI)\dE^{\frac{|m|}{2}}\cos(2\pi m\th),\\
w(\th,\bI)=&\,\sum_{m=1}^\infty w_m(\bI)\dE^{\frac{|m|}{2}}\sin(2\pi m\th),
\eeqa
where the coefficients $r_m(\bI)$ and $w_m(\bI)$ are analytic in $\bI$ on $\Ade\cap\I'$ and satisfy the estimates
\beq\label{ineq:rmwmbound}
|r_m(\bI)|+|w_m(\bI)|\leq M C^{|m|}
\eeq
for some $M,C>0$, independent of $m$, $\eta\in[0,\eta_0]$, and derivative estimate
\beq
|\pa_L^j\pa_E^kr_m|+|\pa_L^j\pa_E^kw_m|\leq M(k+j)!C^{|m|+k+j}.
\eeq
\end{prop}

We observe that the estimates on the coefficients in the expansions guarantee that, even as we approach the vertices $(\Emin^{L_0},L_0)$ or $(E_0,\Lmax)$ of $\I$, the radius of convergence of the series does not degenerate.


\begin{proof}[Proof of Proposition~\ref{prop:rwexpansions}]
Recalling the invertibility of $\Gamma$, we employ the identity~\eqref{eq:gammazero} and work from the identity (re-scaling the argument by $\sqrt{\pi}$ without loss of generality)
\beq
(r-r_L)(\th,\bI)=(\Gamma^{-1})_1(\sqrt{A(\bI)}\cos(2\pi\th),\sqrt{A(\bI)}\sin(2\pi\th),L).
\eeq
It is convenient to use the abbreviations 
\[
X= \cos(2\pi \th), \ \ Y = \sin(2\pi \th)
\]
and also to derive the coefficients for the $m\leq0$ modes, as  parity considerations produce the positive modes as well, using that $r-r_L$ is even in $\th$ and $w$ is odd.
We have
\beqa\label{eq:r-rLmode}
\widehat {(r-r_L)}(m,\bI) & =\int_{\mathbb S^1} (r-r_L)(\th,\bI)e^{-2\pi i m\th}\diff \th  \\
& = \int_{\mathbb S^1}(\Gamma^{-1})_1 (\sqrt{A(\bI)}X,\sqrt{A(\bI)}Y,L) \diff \th \\
& = \sum_{N=0}^\infty A(\bI)^{\frac N2}\int_{\mathbb S^1} \sum_{k+\ell = N} g_{k\ell}(L) X^k Y^\ell e^{-2\pi i m\th} \diff \th,
\eeqa 
where we have expanded $(\Gamma^{-1})_1$ using coefficients $g_{k\ell}(L)$ satisfying $|g_{k\ell}(L)|\leq MC^{k+\ell}$ uniformly in $L$ and $\eta$.
Therefore,
\begin{align*}
{}&\widehat {(r-r_L)}(m,\bI)  =\int_{\mathbb S^1} (r-r_L)(\th,\bI)e^{-2\pi i m\th}\diff \th  \\
& = \sum_{N=0}^\infty A(\bI)^{\frac N2}\int_{\mathbb S^1} \sum_{k+\ell = N} g_{k\ell} 2^{-N}i^{-\ell} (e^{2\pi i \th}+e^{-2\pi i \th})^k(e^{2\pi i\th}-e^{-2\pi i \th})^\ell e^{-2\pi i m\th} \diff \th \\
& = \sum_{N=0}^\infty A(\bI)^{\frac N2}2^{-N}\sum_{k+\ell = N}g_{k\ell} i^{-\ell}  \sum_{k_1=0}^k \sum_{\ell_1=0}^\ell (-1)^{N-\ell_1}  {k\choose k_1}{\ell\choose \ell_1} \int_{\mathbb S^1} e^{2\pi i \th(2k_1+2\ell_1-N-m)} \diff \th  \\
& = \sum_{N=0}^\infty A(\bI)^{\frac N2}2^{-N}\sum_{k+\ell = N}g_{k\ell}i^{-\ell}  \sum_{k_1=0}^k \sum_{\ell_1=0}^\ell (-1)^{N-\ell_1}  {k\choose k_1}{\ell\choose \ell_1} \delta_{2k_1+2\ell_1-N-m=0},
\end{align*}
We now collate the constraints on the powers from above to get
\begin{align}
&N+m=2(k_1+\ell_1), \ \ 
0\le k_1+\ell_1 \le k+\ell = N, \ \ N\ge |m|.
\end{align}
 By relabelling $j=(k_1+\ell_1)$, we derive (recall we have assumed $m\leq0$)
\begin{align}
{}&\widehat {(r-r_L)}(m,\bI) \label{eq:rmode1}\\
& = A(\bI)^{\frac{|m|}2} 2^{-|m|} \sum_{j=0}^\infty A(\bI)^{j}2^{-2j} \sum_{k+\ell = |m|+2j }g_{k\ell}i^{-\ell}  \sum_{0\le k_1\le k , 0\le\ell_1\le\ell \atop k_1+\ell_1 = j}   (-1)^{|m|+2j-\ell_1} {k\choose k_1}{\ell\choose \ell_1}.\notag
\end{align}
On the other hand, we have
\beqa
\Big|\int_{\mathbb S^1} \sum_{k+\ell = N} g_{k\ell} X^k Y^\ell e^{-2\pi i m\th} \diff \th \Big|\leq \sum_{k+\ell = N} \big|g_{k\ell}\big|   \leq (N+1)MC^N,
\eeqa
so that, comparing~\eqref{eq:rmode1} to~\eqref{eq:r-rLmode}, we deduce 
\beqa
\Big| 2^{-|m|} 2^{-2j} \sum_{k+\ell = |m|+2j }g_{k\ell}i^{-\ell}  \sum_{0\le k_1\le k , 0\le\ell_1\le\ell \atop k_1+\ell_1 = j }  (-1)^{|m|+2j-\ell_1} {k\choose k_1}{\ell\choose \ell_1}\Big|\leq MC^{|m|+2j},
\eeqa
where we have increased $C,M$.

In particular, setting
\beq
\tilde r_m(A)=2^{-|m|} \sum_{j=0}^\infty A^{j}2^{-2j} \sum_{k+\ell = |m|+2j }g_{k\ell}i^{-\ell}  \sum_{0\le k_1\le k , 0\le\ell_1\le\ell \atop k_1+\ell_1 = j}   (-1)^{|m|+2j-\ell_1} {k\choose k_1}{\ell\choose \ell_1},
\eeq
we see that the series defining $\tilde r_m(A)$ converges absolutely and is real analytic on a  $\tilde\delta$-neighbourhood of $A=0$ by choosing $\tilde\delta>0$ small enough that $\tilde\delta C^2<1$.  Thus, converting to $\bI$ variables, we have defined an analytic function  $\tilde r_m(\bI)$ on a tubular $\delta$-neighbourhood of $(\EminL,L)$ as required for some $\delta>0$. Since the area function $A(\bI)$ is an analytic function of
$\dE(\bI)$ the existence of $r_m(\bI)$ as stated in~\eqref{eq:rwexpansions} now easily follows. 
Moreover clearly 
$$|r_m(\bI)|\leq 2^{-|m|+1}\sum_{k+\ell=|m|}|g_{k\ell}|\leq MC^{|m|}$$
and the higher derivative estimates follow similarly. 
In order to check the claimed behaviour in the case $m=0$, it is now sufficient to observe that, at $\EminL$, $r=r_L$, so that the 0 mode vanishes, and hence $r_0(\bI)$ is an analytic function of $\dE$ vanishing at zero.

The second claim of~\eqref{eq:rwexpansions} follows by differentiating the identity for $r(\th,\bI)-r_L$ with respect to $\th$, recalling that $w(\th,\bI)=\frac1{T(\bI)}\pa_\th r(\th,\bI)$, and the analyticity of the period function with respect to $\dE$. 
\end{proof}

We now take advantage of Proposition~\ref{prop:rwexpansions} to obtain a quantitative control on the Fourier coefficients of multinomials of the forms $(r-r_L)^{k_1}w^{k_2}$ near trapping. 


\begin{lemma}\label{lemma:Tm}
Let $k_1,k_2\in \N_0$ with $k_1+k_2=k$ and consider the order $k$ polynomial $(r-r_L)^{k_1}w^{k_2}$. Upon taking Fourier transform with respect to $\th$, there exists $\de>0$ such that, for $m\in\Z$, $\bI\in\Ide$, 
\beq
\mathcal{F}_\th\Big((r-r_L)^{k_1}w^{k_2}\Big)(m,\bI)=\dE^{\max\big\{\frac{|m|}{2},\frac{|m|}{2}+\lceil\frac{k-|m|}{2}\rceil\big\}}\mathcal{T}_{m,k_1,k_2}(\bI),
\eeq
where $\mathcal{T}_{m,k_1,k_2}(\bI)$ is analytic for $\bI\in\Ide$ and, for $j_1,j_2\in\N\cup\{0\}$, satisfies
\beq
\big|\pa_E^{j_1}\pa_L^{j_2}\mathcal{T}_{m,k_1,k_2}\big|\leq M^k(j_1!)(j_2!)C^{\max\{|m|,k\}+j_1+j_2}.
\eeq
\end{lemma}


\begin{proof}
It is sufficient to prove the result in the case $k_1=k$, $k_2=0$ (i.e.~the polynomial is a monomial in $r-r_L$ only), as the case of monomials in $w$ follows by analogous arguments, and the general case follows from the usual convolution formulae for Fourier transform.

We recall first that the $m$-th Fourier mode of $(r-r_L)$ may be obtained by directly integrating~\eqref{eq:rwexpansions} to find
\beq
\hat{r}(m,\bI):=\mathcal{F}_\th(r-r_L)(m,\bI) = \frac12 r_m(\bI)\dE^{\frac{|m|}{2}},
\eeq
where the functions $r_m(\bI)$ are analytic and also $r_0(\EminL,L)=0$.
Thus, for any $k\in \N$, using the fact that Fourier transform of a product gives a convolution,
\beqa\label{eq:r^jexpansion}
2^{k}\mathcal{F}_\th(r-r_L)^k(m,\bI)=&\,\sum_{J\in S_{k,m}}\hat{r}(j_1,\bI)\cdots\hat{r}(j_k,\bI)\\
=&\,\sum_{J\in S_{k,m}}\dE^{\frac{|j_1|+\cdots+|j_k|}{2}}\dE^{n_J}r_{|j_1|}(\bI)\cdots r_{|j_k|}(\bI),
\eeqa
where we have supposed without loss of generality that $m\geq0$ and defined
$$S_{k,m}=\{J=(j_1,\ldots,j_k)\in \Z^k\,|\,j_1+\cdots +j_k=m\},\qquad n_J=\#\{i\,|\,j_i=0\}.$$
If $k\leq m$, then we see that the sum $|j_1|+\cdots+|j_k|+2n_J$ is minimised (over the set of indices summing to $m$) by choosing all $j_i$ of the same sign, none of them zero, and obtaining $|j_1|+\cdots +|j_k|+2n_J=m$, so that
\beqa
\mathcal{F}_\th(r-r_L)^k(m,\bI)=&\,2^{-k}\dE^{\frac{|m|}{2}}\sum_{j=0}^\infty\dE^j\sum_{\substack{J\in S_{k,m} \\ |j_1|+\cdots+|j_k|+2|n_J|=m+2j}}r_{|j_1|}(\bI)\cdots r_{|j_k|}(\bI),
\eeqa
where we have observed that the parity of $\sum_{i=0}^k|j_i|$ is always the same as the parity of $m$ and  a simple combinatorial estimate using~\eqref{ineq:rmwmbound} shows that
\beq
\mathcal{T}_{m,k,0}(\bI):=\sum_{j=0}^\infty\dE^j\sum_{\substack{J\in S_{k,m} \\ |j_1|+\cdots+|j_k|+2|n_J|=m+2j}}r_{|j_1|}(\bI)\cdots r_{|j_k|}(\bI)
\eeq
converges to an analytic function near $\dE=0$ (that is, on $\Ade$ with a possibly smaller $\de$) with 
$$\big|\mathcal{T}_{m,k,0}(\bI)\big|\leq M^kC^{|m|}.$$
 On the other hand, if $k>m$, then we first note that $|j_1|+\cdots |j_k|+2n_J$ has the same parity as $m$ and is at least $k$. If $k-m$ is even, then this sum is minimised by choosing all $|j_i|=1$  and $n_J=0$, and we get $|j_1|+\cdots+|j_k|+2n_J=k$. If $k-m$ is odd, then  the minimal value for the sum is $k+1$  and so a similar identity holds. Thus
 \beq
 \mathcal{F}_\th(r-r_L)^k(m,\bI)=\dE^{\max\{\frac{|m|}{2},\frac{|m|}{2}+\lceil\frac{k-|m|}{2}\rceil\}}\mathcal T_{m,k,0}(\bI)
 \eeq
for an analytic function $\mathcal T_{m,k,0}$ satisfying
$$\big|\mathcal{T}_{m,k,0}(\bI)\big|\leq M^kC^{\max\{|m|,k\}}.$$
The estimates on derivatives of $\mathcal{T}$ follows by the usual considerations for analytic functions, passing derivatives inside the summation.
\end{proof}



\subsection{A new foliation of the action support}\label{S:YZ}


We introduce a change of variables 
\[
[\Rmin,\Rmax]\times(\I\cap\{E\ge\Psi_L(R)\} )\ni (R,\bI)\mapsto (R,y,z) 
\] 
by setting
\beq\label{E:RYZ}
y=\om(\bI),\qquad z=E-\Psi_L(R),
\eeq
where we recall $\bI=(E,L)$. A schematic depiction of level sets in the new coordinate system is given by Figure~\ref{F:ELtriangle}.


\begin{lemma}\label{L:RYZ}
Let $\fmn$ be a steady state as in~\eqref{E:SS}--\eqref{E:SS2} such that $\eta\in[0,\eta_0]$, where $\eta_0$ is the small constant introduced in Section~\ref{S:AA}.
There exists a possibly smaller $\eta_0>0$ such that for all $0\le\eta\leq\eta_0$ the change of variables~\eqref{E:RYZ} is well-defined.
\end{lemma}


\begin{proof}
It is straightforward to compute
\begin{align}
\frac{\pa(R,y,z)}{\pa(R,\bI)}=\begin{pmatrix} 1 & 0 & 0 \\ 0 & \om' & \pa_L\om \\ -\Psi_L'(R) & 1 & -\frac1{2R^2} \end{pmatrix},\\
\Big|\frac{\pa(R,y,z)}{\pa(R,\bI)}\Big| = -\big(\frac{\om'}{2R^2}+\pa_L\om\big) = -\P_R \om,\label{E:DET}
\end{align}
where for any $R\in[\Rmin,\Rmax]$ we let
\be\label{E:PRDEF}
\P_R:= \frac1{2R^2}\pa_E+\pa_L. 
\ee
We note that $\om'(\bI)= -\frac{T'(\bI)}{T(\bI)^2}< -C<0$ and $\pa_L\om(\bI)=O(\eta)$. In particular the right-hand side of~\eqref{E:DET} is uniformly positive for sufficiently small $0<\eta\ll1$.
\end{proof}


\begin{remark}
The inverse Jacobian is clearly given by the formula
\beq\label{eq:drELdryz}
\frac{\pa(R,\bI)}{\pa(R,y,z)}=\frac{1}{\P_R\om} \begin{pmatrix} \P_R\om & 0 & 0 \\ \pa_L\om\Psi_L'(R) & \frac{1}{2R^2} & \pa_L\om \\ -\om'\Psi_L'(R) & 1 & -\om' \end{pmatrix}.
\eeq
\end{remark}


It will be important to track the relationship between $\pa_R$, $\pa_y$, and $\pa_z$-derivatives applied to functions
that depend on $(R,y,z)$ only through $\bI(R,y,z)$. 


\begin{lemma}
Let $F\in C^1(\overline\I)$ be given. Then  the following formulas hold:
\begin{align}
&\pa_y [F(E(R,y,z),L(R,y,z))] = \pa_y E \pa_E F + \pa_L F \pa_yL = (\P_R\om)^{-1} \P_RF,\notag\\
&\pa_R [F(E(R,y,z),L(R,y,z))]=\frac{\Psi_L'(R)}{\P_R\om}(\pa_L\om\pa_EF-\om'\pa_LF)=\Psi_L'(R)\pa_z [F(E(R,y,z),L(R,y,z))].\label{E:RZINTERCHANGE0}
\end{align}
Formally introducing the operator 
\begin{align}\label{E:PIDEF}
\Pi : = \pa_L\om\pa_E - \pa_E\om\pa_L = \P_R\om \pa_z,
\end{align}
for any $F\in C^j(\overline \I)$, $j\ge1$, the following formula holds:
\beq\label{E:RZINTERCHANGE}
\pa_R^j(F(E(R,y,z),L(R,y,z)))=\sum_{i=0}^jp_i^j(R,y,z)\Pi^iF,
\eeq
where the analytic coefficients $p_i^j(R,y,z)$ are defined via the incomplete Bell polynomials (see e.g.~\cite{Cvijovic11} for a definition) by
\beq\label{E:BELLDEF}
p_i^j(R,y,z)=B_{j,i}\Big(\frac{\Psi_L'(R)}{\P_R\om},\ldots,\pa_R^{j-i+1}\Big(\frac{\Psi_L'(R)}{\P_R\om}\Big)\Big).
\eeq
We use the convention $B_{j,i}=0$ if $i>j$. For any $n\le\reg$ we have the cancellation bound
\beq\label{E:PIJBOUND}
|\pa_y^np_i^j|\leq C|\Psi_L'(R)|^{\max\{2i-j,0\}}\leq C\dE^{\max\{i-\frac{j}{2},0\}}, \ 0\le i \le j, \ \ n\leq N.
\eeq
\end{lemma}


\begin{proof}
The first identities are trivial.
Formula~\eqref{E:RZINTERCHANGE} is derived directly from~\eqref{E:RZINTERCHANGE0} and the relation
\beq
\frac{\dif}{\dif x}\Big(B_{n,k}(a_1(x),\ldots,a_{n-k+1}(x))\Big)=\sum_{i=1}^{n-k+1}{n\choose i}a_i'(x)B_{n-i,k-1}(a_1(x),\ldots,a_{n-i-k+2}(x)).
\eeq 
Finally, estimate~\eqref{E:PIJBOUND}
 follows from the regularity of $\Psi_L$, the observation that 
\[B_{n,k}(x_1,\ldots,x_{n-k+1})= x_1^{\max\{2k-n,0\}}\widetilde{B_{n,k}}(x_1,\ldots,x_{n-k+1})\]
for some other polynomial $\widetilde{B}_{n,k}$, and the fact that $\Psi_L'(R)$ is independent of $y$.
\end{proof}

\begin{remark}
The operator $\Pi$ is simply a Poisson bracket with $\om(\bI)$ in $(E,L)$-coordinates and acts on functions defined on the action support $\I$. It is used to
describing how the higher $R$ derivatives are exchanged for the $z$ derivatives via formula~\eqref{E:RZINTERCHANGE}.
Bound~\eqref{E:PIJBOUND} encodes a gain of ``good" positive powers of $\dE$ in the vicinity of the trapping set and is used 
to develop the regularity theory in Section~\ref{SS:PRODUCTS}.
\end{remark}

\begin{remark}
We note that the map $\Om\cap\{w\ge0\}\ni (r,w,L)\to (R,y,z)$ is a bijection.
\end{remark}

For the remainder of this section we make the standing assumption that $\eta\in[0,\eta_0]$, where $\eta_0$ is the small constant introduced in Lemma~\ref{L:RYZ}. Constants appearing below will be taken uniformly in $\eta$.

\subsubsection{Domains and notation}


We use the above change of variables only on the region $\I\cap \{E\geq \Psi_L(R)\}$, for which $|w|=\sqrt{2(E-\Psi_L(R))}$ is well-defined. Recall from Lemma~\ref{L:REPRESENTATIONS} that this coincides with the support of 
the gravitational force Green's function $\widehat{h_R}$. For each $R\in[\Rmin,\Rmax]$, we define the set
\beq\label{def:JR}
\J_R:=\{(y,z)\,|\,(E(R,y,z),L(R,y,z))\in\I,\,E(R,y,z)>\Psi_{L(R,y,z)}(R)\}.
\eeq
In order to state the estimates close to the trapping region, we set
\beq
\J_{R,\de}:=\{(y,z)\,|\,(E(R,y,z),L(R,y,z))\in\Ide\}.
\eeq
Finally, to provide a unified notation for functions of $(R,y,z)$, we define the set
\beq\label{def:Upsilon}
\Upsilon:=\{(R,y,z)\,|\,R\in[\Rmin,\Rmax],\,(y,z)\in\J_R\},
\eeq
while the $\delta$ neighbourhood of the trapping set in $\Upsilon$ corresponds to
\beq
\Upsilon_\de:=\{(R,y,z)\in\Upsilon\,|\,\dE(R,y,z)<\de\}.
\eeq
Note that the projection of $\Upsilon_\de$ into the $R$ coordinate is a compact sub-interval of $(\Rmin,\Rmax)$.

\begin{remark}\label{R:YIBP}
For any fixed $z_0\geq 0$, one sees directly that the image of the level set $\{z=z_0\}$ in the $\bI$ coordinates has both end-points in $\pa\I_{\text{vac}}$. Compare Figure~\ref{F:ELtriangle}.
\end{remark}

We introduce notation for the vacuum boundary portion of $\J_R$ by setting
\beq
 \pa\J_R^{\text{vac}}=\pa\J_R\cap\big(\{E(R,y,z)=E_0\}\cup\{L(R,y,z)=L_0\}\big).
\eeq
We note that $\pa\J_R^{\text{vac}}$ is smooth except at one point, where it is merely Lipschitz (the image of the vertex $(E_0,L_0)\in \bar\I$). The remaining portion of the boundary of $\J_R$ is contained within the line $z=0$, that is, 
\beq\label{E:JRMOVINGBDRY}
\pa\J_R\setminus \pa\J_R^{\text{(vac)}}\subset\{(y,z)\,|\,z=0\}.
\eeq
Finally, to  fix notation, let us use $R$ for the point in $[\Rmin,\Rmax]$, $r=r(\th,\bI)$ for the function mapping from action-angle variables to phase-space coordinates $(r,w,L)$, and set, for $(y,z)\in \J_R$,
\beqa
\tilde r(\th,R,y,z):=r(\th,E(R,y,z),L(R,y,z)), \ \tilde w(\th,R,y,z):=w(\th,E(R,y,z),L(R,y,z)).
\eeqa
Now, from Proposition~\ref{prop:rwexpansions}, we have the following lemma
\begin{lemma}\label{lemma:rwyzexpansions}
There exists $\de>0$ such that, for $(\th,R,y,z)\in\S^1\times\Ude$, the following identities hold: 
\beqa\label{eq:rwyzexpansions}
(\tilde r-r_L)(\th,R,y,z)=&\,\tilde r_0(R,y,z)\dE(R,y,z)+\sum_{m=1}^\infty \tilde r_m(R,y,z)(\dE(R,y,z))^{\frac{m}{2}}\cos(2\pi m\th),\\
\tilde w(\th,R,y,z)=&\,\sum_{m=1}^\infty \tilde w_m(R,y,z)(\dE(R,y,z))^{\frac{m}{2}}\sin(2\pi m\th),
\eeqa
where the coefficients $\tilde r_m(R,y,z)$ and $\tilde w_m(R,y,z)$ are analytic on $\Ude$ and, for any multi-index $\al$ in $(R,y,z)$, satisfy the bounds
\beq\label{E:rmwmRyzbds}
|\pa^\al \tilde r_m|+|\pa^\al\tilde  w_m|\leq M(|\al|!)C^{|m|+|\al|}.
\eeq
\end{lemma}

\begin{proof}
This follows directly from Proposition~\ref{prop:rwexpansions}, noting that $E$ and $L$ are analytic functions of $(R,y,z)$ by the change of variables identity~\eqref{eq:drELdryz}, analyticity of $\om$, and noting $R$ is bounded uniformly away from zero.
\end{proof}

\begin{lemma}\label{lemma:Tmyz}
Let $m\in\Z_*$, $k_1,k_2\in\N$. There exists $\de>0$ such that the coefficient functions $\mathcal{T}_{m,k_1,k_2}$, as in Lemma~\ref{lemma:Tm}, considered as functions on $\Ude$, satisfy
\beq
\big|\pa_R^{j_1}\pa_y^{j_2}\pa_z^{j_3}\mathcal{T}_{m,k_1,k_2}\big|\leq M(j_1!)(j_2!)(j_3!)C^{\max\{|m|,k_1+k_2\}+j_1+j_2+j_3}.
\eeq
\end{lemma}

\begin{proof}
This follows directly from the analyticity of the change of variables from $(R,\bI)$ to $(R,y,z)$ and the estimates of Lemma~\ref{lemma:Tm}.
\end{proof}


\begin{remark}[Steady state properties in $(R,y,z)$ variables]\label{REM:SSDERIVS}
Recalling~\eqref{E:SS} 
we define
\beq
\tilde\varphi(R,y,z):=\varphi(E(R,y,z),L(R,y,z)),\ \ \ \tilde\varphi'(R,y,z):=\varphi'(E(R,y,z),L(R,y,z)).
\eeq
It is  straightforward to check that $\tilde\varphi'(R,y,z)\in C^\ell(\overline{\Upsilon})$ for $\ell\leq \al$ and its mixed $(R,y,z)$ derivatives of order up to $\lceil \al-1\rceil$ vanish on $\pa\J_R^{\text{vac}}$, where
\beq\label{def:alpha}
\al=\al(\mu,\nu):=\min\{\mu-1,\nu\}>0.
\eeq 
\end{remark}

\section{Regularity theory at trapping}\label{S:REGULARITY}


\subsection{Regularity of polynomials in $(R,y,z)$ variables}


We assume throughout this section that we are working with a steady state $\fmn$ satisfying~\eqref{E:SS}--\eqref{E:SS2} such that $\eta\in[0,\eta_0)$, where $\eta_0$ is the sufficiently small constant given by Lemma~\ref{L:RYZ}. Constants will be taken uniform in $\eta$ in the following.
The following lemma states some useful bounds and explain a cancellation important to our description of trapping~\eqref{eq:dydeltaE}--\eqref{eq:dRdeltaE}.


For $R\in[r_{L_0},r_{\Lmax}]$, we define $L_R\in[L_0,\Lmax]$ such that
\beq\label{E:LRDEF}
R=r_{L_R},\text{ i.e.~}\Psi_{L_R}(R)=\Emin^{L_R}.
\eeq
When $R\in[\Rmin,\Rmax]\setminus[r_{L_0},r_{\Lmax}]$, we take a smooth, monotone extension of this function, such that $r_{L_R}\neq \Rmin,\Rmax$ on this set and $|r_L-r_{L_R}|\leq |r_L-R|$.

\begin{lemma}\label{L:L-LR}
Let $R\in[\Rmin,\Rmax]$. 
There exists  $C>0$ such that for all $0\le\eta<\eta_0$, on the restricted action-angle domain $\I\cap\{E\geq\Psi_L(R)\}$,
 the following estimate holds:
\beqa\label{ineq:L-LR}
|L-L_R|\leq C|r_L-r_{L_R}|\leq C|r_L-R|\leq  C\dE^{\frac12},
\eeqa
where we recall~\eqref{DEF:EDIFF}.
Moreover, the following identities hold:
\begin{align}\label{eq:dydeltaE}
\pa_y(\dE(R,y,z))& =\frac{1}{2\P_R \om}\Big(\frac{1}{R^2}-\frac{1}{r_L^2}\Big), \\
\pa_R(\dE(R,y,z)) &=\frac{\Psi_L'(R)}{\P_R\om}\Big(\pa_L\om+\om'\frac{1}{r_L^2}\Big). \label{eq:dRdeltaE}
\end{align}
\end{lemma}


\begin{proof}
From Lemma~\ref{lemma:rL}, as $L_R\in[L_0,\Lmax]$ by definition, we obtain $|L-L_R|\leq C|r_L-r_{L_R}|$ and $|r_L-r_{L_R}|\leq |r_L-R|$ follows directly from the definition.  Estimate~\eqref{ineq:L-LR} follows as, on $\I$ and for $E\geq\Psi_L(R)$,
\beq
(r_L-R)^2\leq C(\Psi_L(R)-\Psi_L(r_L))= C(\Psi_L(R)-\EminL)\leq C\dE,
\eeq
 The
proof of~\eqref{eq:dydeltaE}--\eqref{eq:dRdeltaE} is a simple and direct calculation using $\pa_L \EminL=\frac12 r_L^{-2}$.
\end{proof}



\begin{remark}\label{R:CANCEL}
The identities~\eqref{eq:dydeltaE}--\eqref{eq:dRdeltaE} represent hidden cancellations in derivatives of $\dE$. These cancellations are quantified as a gain of integrability due to the estimate 
\[
|R-r_L|+|\Psi_L'(R)|\leq C\dE^{\frac12} \ \text{ on} \ \J_R. 
\]
Throughout, we will continue to write $\dE$ for the function $\dE(R,y,z)$.
\end{remark}

To organise the $y$ and $z$ derivatives of $\tilde r-r_L$ and $\tilde w$, we collect them in the following lemma. 

\begin{lemma}\label{L:RWBOUNDS}
Let $j_1,j_2\in \N_0$. Then there exists $\de>0$ such that the $(y,z)$ derivatives on $\Ude$ satisfy
\beq\label{ineq:dRdyr-rL}
\big|\pa^{j_1}_y\pa_z^{j_2}(\tilde r-r_L)\big|+\big|\pa^{j_1}_y\pa_z^{j_2}\tilde w\big|\leq C_k\dE^{\frac12-\frac{j_1}{2}-j_2}.
\eeq
\end{lemma}

\begin{proof}
We focus on the estimate for $\tilde r-r_L$ as the proof for $\tilde w$ follows similar lines. We begin by proving that  for any $(R,y,z)\in\Ude$, $\th\in\S^1$, $k\in\mathbb N$, the following identity holds:
\beqa\label{eq:pakyr}
{}&\pa_y^k(\tilde r-r_L)(\th,R,y,z)\\
&=\tilde r_{0,y}^{(k)}(R,y,z)+\sum_{j=0}^{\lceil\frac{k}{2}\rceil}\Big(\frac{1}{R^2}-\frac{1}{r_L^2}\Big)^{\max\{k-2j,0\}}\sum_{m=1}^\infty \tilde r_{m,y}^{(k,j)}(R,y,z)\dE^{\frac{m}{2}-k+j}\cos(2\pi m\th),
\eeqa
 where the functions $\tilde r_{m,y}^{(k,j)}$
  are analytic on $\Ude$ and inherit bounds from~\eqref{E:rmwmRyzbds}.
 
We proceed by induction. Using~\eqref{eq:dydeltaE}, we differentiate the expansion for $\tilde r-r_L$ in~\eqref{eq:rwyzexpansions} to obtain
\beqa\label{eq:dyr}
\pa_y&\,(\tilde r-r_L)(\th,E(y,z),L(y,z))\\
=&\,\Big(\frac{1}{R^2}-\frac{1}{r_L^2}\Big)\frac{1}{2\P_R\om}\sum_{m=1}^\infty \frac{m}{2}\tilde r_m(R,y,z)\dE^{\frac{m}{2}-1}\cos(2\pi m\th)\\
&+\frac{1}{\P_R\om}\frac12\Big(\frac{1}{R^2}-\frac{1}{r_L^2}\Big)\tilde r_0(R,y,z)+ \pa_y\tilde r_0(R,y,z)\dE+\sum_{m=1}^\infty \pa_y\tilde r_m(R,y,z)\dE^{\frac{m}{2}}\cos(2\pi m\th),
\eeqa
where we observe that the last term on the on the right of the last line is more regular than the middle line. By defining the coefficients in the obvious manner, this demonstrates~\eqref{eq:pakyr} in the case $k=1$.

To proceed, suppose that the identity~\eqref{eq:pakyr} holds for some $k$. Differentiating each term in the main sum with respect to $y$, if $k-2j>0$ (i.e.~$j<\lceil\frac{k}{2}\rceil$), we obtain
\begin{align*}
\pa_y&\,\bigg[\Big(\frac{1}{R^2}-\frac{1}{r_L^2}\Big)^{\max\{k-2j,0\}}\Big(\sum_{m=1}^\infty \tilde r_m^{(k,j)}(R,y,z)\dE^{\frac{m}{2}-k+j}\cos(2\pi m\th)\Big)\bigg]\\
=&\,\Big(\frac{1}{R^2}-\frac{1}{r_L^2}\Big)^{k+1-2(j+1)}\frac{2(k-2j)}{\P_R\om}\frac{\pa_Lr_L}{r_L^3}\Big(\sum_{m=1}^\infty \tilde r_m^{(k,j)}(R,y,z)\dE^{\frac{m}{2}-(k+1)+(j+1)}\cos(2\pi m\th)\Big)\\
&+\Big(\frac1{R^2}-\frac{1}{r_L^2}\Big)^{k+1-2j}\frac{1}{2\P_R\om}\Big(\sum_{m=1}^\infty \big(\frac{m}{2}-k+j\big)\tilde r_m^{(k,j)}(R,y,z)\dE^{\frac{m}{2}-(k+1)+j}\cos(2\pi m\th)\Big)\\
&+\Big(\frac1{R^2}-\frac{1}{r_L^2}\Big)^{k+1-2(j+1)}\Big(\sum_{m=1}^\infty \Big(\frac{1}{R^2}-\frac{1}{r_L^2}\Big)\pa_y\tilde r_m^{(k,j)}(R,y,z)\dE^{\frac{m}{2}-(k+1)+(j+1)}\cos(2\pi m\th)\Big).
\end{align*}
We note that the first and third lines on the right may be grouped together, as the coefficients remain analytic. 
In the case that $k-2j\leq 0$ (i.e.~$j=\lceil\frac{k}{2}\rceil$), the term corresponding to the first line on the right in the previous equation does not arise, and a slightly simpler formula ensues.
 It is then a simple exercise to group the coefficients and check the expansion~\eqref{eq:pakyr}.

From the identity~\eqref{eq:pakyr}, the  estimate~\eqref{ineq:dRdyr-rL} in the case $j_2=0$ follows directly as, on $\Jde$, we employ $|R-r_L|\leq C\dE^{\frac12}$ (from Lemma~\ref{L:L-LR}) and the convergence of the series, and on the complement region, where $\dE\geq \de$, we note all derivatives are bounded from Lemma~\ref{L:ELNONTRAPPINGREG} and analyticity of the change of variables to $(R,y,z)$.  

Now applying $\pa_z^{j_2}$, we observe simply that $|\pa_z\dE(R,y,z)|\leq C$ on $\Upsilon_\de$ to conclude.
\end{proof}


\subsection{Regularity of $C^k$ Taylor remainder functions in $(y,z)$ variables}
In this section, we collect estimates on the Fourier transform of Taylor remainders and their $y$ and $z$ derivatives. Throughout this section, we suppose that the function $g\in C^{\n}(\overline{\Om})$ is given, where, for simplicity of notation, we suppose $\n\leq N+1$, where we recall~\eqref{def:N}.

For a given $R\in[\Rmin,\Rmax]$, we define 
a Taylor polynomial and remainder as 
\beqa\label{eq:gTaylorsplitting}
g(r,w,L)=&\,\sum_{j_1+j_2+j_3\leq \n-1}\frac{\pa_r^{j_1}\pa_w^{j_2}\pa_L^{j_3}g(r_{L_R},0,L_R)}{j_1!j_2!j_3!}(r-r_{L_R})^{j_1}w^{j_2}(L-L_R)^{j_3}+\tilde g(r,w,L,R)\\
=:&\,\Poly_{\n}[g](r,w,L,R)+\Rem_{\n}[g](r,w,L,R),
\eeqa
where, for multi-indices $\al$ in $(r,w,L)$ with $|\al|\leq \n$,
\beq\label{ineq:dalphagtilde1}
\big|\pa^\al\Rem_{\n}[g](r,w,L,R)\big|\leq C\|g\|_{C^{\n}(\oom)}\big(|r-r_{L_R}|+|w|+|L-L_R|\big)^{\n-|\al|}
\eeq
and we recall~\eqref{E:LRDEF}. At any point $(r,w,L)$ such that $E(r,w,L)\geq 
\Psi_L(R)$, this estimate can be simplified by  using the triangle inequality to estimate $|r-r_{L_R}|\leq |r-r_L|+|r_L-r_{L_R}|$ and applying the inequalities~\eqref{ineq:L-LR} to obtain
\beq\label{ineq:dalphagtilde}
\big|\pa^\al\Rem_{\n}[g](r,w,L,R)\big|\leq C\|g\|_{C^{\n}(\oom)}\dE(r,w,L)^{\frac{\n-|\al|}{2}}.
\eeq

\begin{lemma}\label{L:RWLBOUND}
Let $g\in C^{\n}(\overline{\Om})$ and $\al$ a multi-index in $(r,w,L)$ such that $|\al|\leq \n$. Then there exists $C>0$ such that for all $R\in [\Rmin,\Rmax]$ and $(r,w,L)\in\oom$ such that $E(r,w,L)\geq \Psi_L(R)$, we have the estimate
\beq\label{ineq:gtildeest}
\big|\pa^\al\pa_\th^{\n-|\al|}\Rem_{\n}[g](r,w,L,R)\big|\leq C\|g\|_{C^{\n}(\oom)}\dE(r,w,L)^{\frac{\n-|\al|}{2}}.
\eeq
\end{lemma}

\begin{proof}
To verify this estimate, we first note that it is sufficient to establish this under the additional assumption $\bI\in\Ade$, as the regularity of the derivative operators away from the trapping set from Lemma~\ref{L:ELNONTRAPPINGREG} establishes the estimate directly on the complement of this set. 

To simplify the notation, we set $\tilde g = \Rem_{\n}[g]$ throughout the proof. We set $k_0=\n-|\al|$ and apply the generalised Faa di Bruno formula~\eqref{eq:genFdB} (treating $\bI$ as fixed) to derive
\beq
\pa_\th^{k_0}\pa^\al\tilde g=\sum_{1\leq|\la|\leq k_0}\pa^{\al+\la}\tilde g(r(\th,\bI),w(\th,\bI),L,R)\sum_{s=1}^{k_0}\prod_{j=1}^{k_0}\sum_{p_s((k_0),\la)}k_0!\frac{(\pa_\th^{\ell_j} (r-r_L),\pa_\th^{\ell_j}w)^{k_j}}{(k_j!)(\ell_j!)^{|k_j|}},
\eeq
where we slightly abuse notation to treat multi-indices $\ell_j$ as scalars and the multi-indices $\la$ count $(r,w,L)$ derivatives but always have a zero contribution for the $L$ derivatives. 
Thus, as $|\pa^{\al+\la}\tilde g(r(\th,\bI),w(\th,\bI),L)\big|\leq C\|g\|_{C^{\n}(\oom)}\dE^{\frac{\n-|\al|-|\la|}{2}}$ from~\eqref{ineq:dalphagtilde} and $|\pa_\th^{\ell_j} r(\th,\bI)|+|\pa_\th^{\ell_j} w(\th,\bI)|\leq C\dE^{\frac12}$ for all $\ell_j$ from Proposition~\ref{prop:rwexpansions}, then using the relation $\sum |k_j|=|\la|$ (compare~\eqref{def:ps}), we deduce
\beq\label{ineq:htildeest1}
\big|\pa_\th^{k_0}\pa^\al\tilde g\big|\leq C\|g\|_{C^{\n}}\dE^{\frac{\n-|\al|}{2}}.
\eeq 
Finally, we observe that the commutator $[\pa_\th^{k_0},\pa^\al]\tilde g$ picks up only lower order derivatives of $\tilde g$ with bounded coefficients, and hence we conclude. 
\end{proof}


Using this lemma and Lemma~\ref{L:RWBOUNDS}, in the next proposition we quantify the trapping degeneracy of the remainder with respect to the high-order $(y,z)$-derivatives.


\begin{prop}\label{prop:remainder}
Let $g\in C^{\n}(\oom)$ and consider the Taylor remainder $$\tilde g(\th,R,y,z)=\Rem_{\n}[g](\th,E(R,y,z),L(R,y,z),R)$$ as a function of $(\th,R,y,z)\in\S^1\times\Upsilon$.  Then, for all $j_1,j_2,j_3\in\N$ such that $j_1+j_2+j_3\leq \n$, there exists $C>0$ such that
\beqa
\big|\pa_y^{j_1}\pa_z^{j_2}\pa_\th^{j_3}\tilde g(\th,R,y,z)\big|\leq C\|g\|_{C^{\n}(\oom)}\dE(R,y,z)^{\frac{\n}2-\frac{j_1}{2}-j_2}.
\eeqa
\end{prop}

\begin{proof}
We again note that it is sufficient from Lemma~\ref{L:ELNONTRAPPINGREG} to establish the estimate only for $(R,y,z)\in\Ude$. We first set $\tilde h(r-r_L,w,L,R)=\pa_\th^{j_3} \Rem_{\n}[g](r,w,L,R)$. Now from~\eqref{ineq:gtildeest}, as $r_L$ is an analytic function of $L$, we have, for all multi-indices $\al$ in $(r,w,L)$ with $|\al|\leq k$,
\beq\label{ineq:htildeests}
\big|\pa^\al \tilde h\big|\leq C\|g\|_{C^{\n}(\oom)}\dE^{\frac{\n-|\al|}{2}}.
\eeq
To apply the generalised Fa\`a di Bruno formula,~\eqref{eq:genFdB}, we introduce the convention that multi-indices $\nu,\ell_j\in\N_0^4$ record $(\th,R,y,z)$ derivatives, while multi-indices $k_j\in\N_0^4$ record $(r,w,L,R)$ derivatives. Now we set the multi-index $\nu=(0,0,j_1,j_2)$ so that $\pa^\nu=\pa_y^{j_1}\pa_z^{j_2}$. Applying the generalised Fa\`a di Bruno formula,~\eqref{eq:genFdB}, we have
\beqa
\pa^\nu \tilde h=&\,\sum_{1\leq|\la|\leq |\nu|}(\pa^\la \tilde h)(r(\th,R,y,z)-r_L,w(\th,R,y,z),L(R,y,z),R)\\
&\times \sum_{s=1}^{|\nu|}\sum_{p_s(\nu,\la)}(\nu!)\prod_{j=1}^{s}\frac{(\pa^{\ell_j}(r-r_L),\pa^{\ell_j}w,\pa^{\ell_j}L,\pa^{\ell_j}R)^{k_j}}{(k_j!)(\ell_j!)^{|k_j|}}.
\eeqa
We note now from~\eqref{ineq:dRdyr-rL} that
\beq
\big|\pa^{\ell_j}(r-r_L)\big|+\big|\pa^{\ell_j}w\big|\leq C\dE^{\frac12-\frac12(\ell_j)_3-(\ell_j)_4},
\eeq
while $|\pa^{\ell_j}L|\leq C$, $\pa^{\ell_j}R=0$ as $(\ell_j)_2=0$ always. Thus, for any set of multi-indices $(k_1,\ldots,k_s;\ell_1,\ldots,\ell_s)\in p_s(\nu,\la)$,
\beqa
{}&\Big|\prod_{j=1}^{s}\frac{(\pa^{\ell_j}(r-r_L),\pa^{\ell_j}w,\pa^{\ell_j}L,0)^{k_j}}{(k_j!)(\ell_j!)^{|k_j|}}\Big|\leq C\dE^{\sum_{j=1}^s(\frac12-\frac12(\ell_j)_3-(\ell_j)_4)((k_j)_1+(k_j)_2)}\\
&\leq C\dE^{\sum_{j=1}^s(\frac12 - \frac12(\ell_j)_3-(\ell_j)_4)|k_j|},
\eeqa
where we recall that each of the $k_j$ and $\ell_j$ is a multi-index. In the last inequality, we have used that $0\prec \ell_j$ for each $j$ (see Appendix~\ref{A:FDB}), so that $|\ell_j|\geq 1$, and hence $\frac12-\frac{|\ell_j|}{2}\leq 0$.

Now,  from the definition of the partition set $p_s(\nu,\la)$, we have
\beqa
\sum_{j=1}^s|k_j|\ell_j=\nu=(0,0,j_1,j_2),\qquad \sum_{j=1}^s |k_j|=|\la|,
\eeqa
so that
\beqa
\Big|\prod_{j=1}^{s}\frac{(\pa^{\ell_j}(r-r_L),\pa^{\ell_j}w,\pa^{\ell_j}L,0)^{k_j}}{(k_j!)(\ell_j!)^{|k_j|}}\Big|\leq C\dE^{\frac12|\la|-\frac12j_1-j_2},
\eeqa
 and hence, as by Lemma~\ref{L:RWLBOUND}
 \beq
 |\pa^\la\tilde h(r-r_L,w,L,R)|\leq C\|g\|_{C^{\n}(\oom)}\dE^{\frac{\n-|\la|}{2}},
 \eeq
 we obtain
 \beqa
 |\pa^\nu\tilde h(\th,R,y,z)|\leq C_k\|g\|_{C^{\n}(\oom)} \dE^{\frac{\n}{2}-\frac{j_1}{2}-j_2}.
 \eeqa
\end{proof}


\subsection{Regularity of the Green's function}


The purpose of this section is to establish two separate estimates for the Green's function $\widehat{h_R}=(\pi m)^{-1}S_m(\th(R,y,z))$ in the $(R,y,z)$ coordinates. The first follows from a partial separation of variables
in the $(R,y,z)$ coordinates and is contained the following proposition.

\begin{prop}\label{prop:Greensdecomp}
There exists $\de>0$ such that, for  $(R,y,z)\in\Ude$, $m\in\Z_*$, the $m$-th Fourier mode $S_m(\th(R,y,z))$  admits the decomposition
\beq\label{eq:Greensdecomp}
S_m(\th(R,y,z))=(\dE(R,y,z))^{-\frac{|m|}{2}}H_m(R,y,z),
\eeq
where $H_m$ is analytic with respect to $(R,y,\sqrt{z})$ on $\Ude$. For any $k,j\in\N\cup\{0\}$, there exists a constant $C_{k,j}>0$, depending on the steady state, $j$ and $k$, such that, for all $(R,y,z)\in\Ude$,
\beq\label{ineq:dyRHm}
\big|\pa_R^j\pa_y^k H_m\big|\leq C_{k,j}C^{|m|}|m|^{k+j}\dE^{\max\{\frac{|m|-k-j}{2},0\}}.
\eeq
\end{prop}

\begin{proof}
We recall that the $m$-th mode of the Green's function is a $(\pi m)^{-1}$ multiple of
\beq
S_m(\th(R,\bI))=\sin\big(2\pi m \th(R,\bI)\big)=\textup{sgn}(m)\Im\Big(\big(\cos(2\pi\th(R,\bI))+i\sin(2\pi\th(R,\bI))\big)^{|m|}\Big).
\eeq
We recall~\eqref{def:gamma} and Lemma~\ref{lemma:twist} and set
$$\tilde\ga(\th,\bI)=g_1(\th,\bI)+ig_2(\th,\bI)=\sqrt{A(\bI)}\big(\cos(2\pi\th)+i\sin(2\pi\th)\big),$$
with a corresponding definition for $\tilde\Gamma(r,w,L)=G_1(r,w,L)+iG_2(r,w,L)$, compare~\eqref{def:Gamma}. We then have the decomposition
\beq
S_m(m,\bI)=\dE(\bI)^{-\frac{|m|}{2}}\Big(\frac{\dE(\bI)}{A(\bI)}\Big)^{\frac{|m|}{2}}\Im\Big(\big(\tilde \ga(\th(R,\bI),\bI)\big)^{|m|}\Big).
\eeq
By definition,
$\tilde\ga(\th(R,\bI),\bI)=\tilde\Ga(R,w(\th(R,\bI),\bI),L)$ and 
\beq\label{E:wRyz}
w(\th(R,E(R,y,z),L(R,y,z)),E(R,y,z),L(R,y,z)) = \sqrt {2z}.
\eeq
Thus, upon moving to $(R,y,z)$ variables, we have
\beq
S_m(m,\bI)=\dE(\bI)^{-\frac{|m|}{2}}\Big(\frac{\dE}{A(\bI)}\Big)^{\frac{|m|}{2}}\Im\Big(\big(\tilde \Gamma(R,\sqrt{z},L(R,y,z))\big)^{|m|}\Big).
\eeq
Setting 
\beq\label{E:HMDEF}
H_m(R,y,z):=\Big(\frac{\dE(R,y,z)}{A(R,y,z)}\Big)^{\frac{|m|}{2}}\Im\Big(\big(\tilde \Gamma(R,\sqrt{z},L(R,y,z))\big)^{|m|}\Big),
\eeq
we have arrived at the claimed decomposition~\eqref{eq:Greensdecomp} and the analyticity of $H_m$ with respect to $(R,\sqrt{z},y)$ follows from the definition, the analyticity of $A(\bI)$, of $\om(\bI)$, of $\Gamma$ (as a canonical normal form), and of the change of variables to $(R,y,z)$.

To prove the estimate~\eqref{ineq:dyRHm}, we first observe that as $\tilde\Gamma\circ L$ is analytic in $(R,y)$, 
 there exist constants $C,K>0$ such that 
$$\big|\pa_R^j\pa_y^k(\tilde\Gamma\circ L)\big|\leq Ck!j!K^{k+j},$$
by the standard theory of analytic functions.  On the other hand, as $\dE\lesssim A(\bI)\lesssim \dE$, by definition of $\tilde \Gamma$, we have the simple estimate
$$|\tilde\Gamma|\leq C\dE^{\frac12}.$$ 
Applying the generalised Faa di Bruno formula (see Appendix~\ref{A:FDB}) with multi-index $\nu=(j,k,0)$ (and taking $m> 0$ for simplicity of notation),
\beqa
\pa^\nu \big(\tilde\Gamma^m\big) =&\,\sum_{1\leq\la\leq |\nu|}m\cdots(m-\la+1) (\tilde \Gamma)^{m-\la} \sum_{s=1}^{|\nu|}\sum_{p_s(\nu,\la)}(\nu!)\prod_{j=1}^{s}\frac{(\pa^{\ell_j}(\tilde \Gamma))^{k_j}}{(k_j!)(\ell_j!)^{k_j}},
\eeqa
where we note that $\la$ and $k_j$ are natural numbers. Estimating, this is bounded by
\beqa
{}&\big|\pa^\nu \big(\tilde\Gamma\big)^m\big|\leq \dE^{\max\{\frac{m-|\nu|}{2},0\}}\sum_{1\leq\la\leq |\nu|}m\cdots(m-\la+1) C^{m-\la}\sum_{s=1}^{|\nu|}\sum_{p_s(\nu,\la)}(\nu!)\prod_{j=1}^{s}\frac{C^{k_j}K^{|\ell_j|k_j}}{(k_j!)}\\
&=K^{|\nu|}C^{|m|}\dE^{\max\{\frac{m-|\nu|}{2},0\}}\sum_{1\leq\la\leq |\nu|}m\cdots(m-\la+1)
 \sum_{s=1}^{|\nu|}\sum_{p_s(\nu,\la)}(\nu!)\prod_{j=1}^{s}\frac{1}{(k_j!)},
\eeqa
where we have used the identities $\sum_{j=1}^s|\ell_j|k_j=|\nu|$ and $\sum_{j=1}^s k_j=\la$. We observe that, again by the generalised Fa\`a di Bruno formula, the combinatorial sum on the right hand hand side is exactly
\beqs
\pa_s^j\pa_t^k\big((1-t)^{-m}(1-s)^{-m}\big)\big|_{(t,s)=(0,0)}=m\cdots(m+k-1)m\cdots(m+j-1)\leq C_{k,j} m^{k+j},
\eeqs
so that we have shown
\beq
\big|\pa_R^j\pa_y^k \big(\tilde\Gamma\big)^m\big|\leq C_{k,j} m^{k+j}K^{k+j}C^{|m|}\dE^{\max\{\frac{m-k-j}{2},0\}}.
\eeq
As $\dE^{\frac12} A(\bI)^{-\frac12}$ is analytic and uniformly positive as a function of $\bI$ it is also analytic as a function of $(R,y,z)$, and hence a similar argument to the above yields (for possible different $C$, $K$)
\beq
\Big|\pa_R^j\pa_y^{k}\Big(\frac{\dE(R,y,z)}{A(R,y,z)}\Big)^{\frac{|m|}{2}}\Big|\leq C^{|m|} K^{k+j}m\cdots(m+k-1)m\cdots(m+j-1).
\eeq
In order the close the estimate on $H_m$, we simply distribute derivatives by the product rule to conclude the proof of the lemma.
\end{proof}


\begin{lemma}\label{L:HRMAINEST}
Let $j,k\in\N$, $j+k\leq N+1$. Then there exists $C_{j,k}>0$ such that for $(R,y,z)\in\Upsilon$, the angle function $\th(R,y,z)$ satisfies the estimate
\beq\label{ineq:dyRtheta}
\big|\pa_R^j\pa_y^k\th(R,y,z)\big|\leq C_{j,k}(\dE(R,y,z))^{-\frac{k+j}{2}}.
\eeq
As a consequence,
\beq\label{ineq:dyRsin}
\Big|\pa_R^j\pa_y^k\big(\sin(2\pi m\th(R,y,z))\big)\Big|\leq C|m|^{k+j}\dE^{-\frac{k+j}{2}}.
\eeq
\end{lemma}

\begin{proof}
We begin by recalling from Theorem~\ref{T:NORMALFORM} that the canonical normal form change of variables is globally well-defined and $C^{N+1}$, so that, recalling~\eqref{E:wRyz}, we have
\beqas
{}&G_1(R,\sqrt{z},L(R,y,z))=\sqrt{A(R,y,z)}\cos(2\pi \th(R,y,z)),\\
 &G_2(R,\sqrt{z},L(R,y,z))=\sqrt{A(R,y,z)}\sin(2\pi \th(R,y,z)),
\eeqas
and hence
\beqs
\th(R,y,z)=\arctan\Big(\frac{G_2(R,\sqrt{z},L(R,y,z))}{G_1(R,\sqrt{z},L(R,y,z))}\Big).
\eeqs
Differentiating,
\[
2\pi \pa_y\th(R,y,z) = \frac{\pa_y G_2 G_1- G_2\pa_y G_1}{G_1^2+G_2^2}=\frac{\pa_y G_2 G_1- G_2\pa_y G_1}{A}.
\]
Moreover, from Theorem~\ref{T:NORMALFORM}, we easily find $|G_\ell|\leq C\dE^{\frac12}$, $\ell=1,2$ and $|\pa_R^j\pa_y^k (G_\ell(R,\sqrt{z},L(R,y,z))|\leq C_{k,j}$ for $k,j\in\N$, so that
$$\Big|\pa_R^j\pa_y^k \big(\pa_y G_2 G_1- G_2\pa_y G_1\big)\Big|\leq \begin{cases}
C_{j,k}, & j+k\geq 1,\\
C\dE^{\frac12}, & j=k=0.\end{cases}$$
On the other hand, recalling $A(\bI)$ is analytic in $\dE$ close to $\dE=0$ and $C^{N+1}$ globally by Lemma~\ref{lemma:regularity}, we have
\beq
\big|\pa_R^j\pa_y^k \big(A(\bI)^{-1}\big)\big| \leq C_{k,j}\dE^{-1-\frac{j+k}{2}},
\eeq
where we have used~\eqref{eq:dydeltaE}--\eqref{eq:dRdeltaE} and recalled $|r_L-R|\leq C\dE^{\frac12}$ on $\Upsilon$. Therefore, distributing derivatives, we easily conclude the claimed estimate~\eqref{ineq:dyRtheta}:
\beqa
\big|\pa_R^j\pa_y^k\th|\leq  C_{k,j}\dE^{-\frac{k+j}{2}}.
\eeqa
The second estimate,~\eqref{ineq:dyRsin}, now follows by the usual generalised Faa di Bruno argument.
\end{proof}


\subsection{Products of the form $\widehat{g}\widehat{h_R}$}\label{SS:PRODUCTS}


Given a function $g\in C^{\n}(\oom)$ and a point $R\in[\Rmin,\Rmax]$,  the regularity of the map
\[
\bI\mapsto \widehat g(m,\bI)\widehat{h_R}(m,\bI) = (\pi m)^{-1} \widehat{g}(m,\bI)\sin(2\pi m \th(R,\bI)),
\]
near the trapping region $\pa\I_0$~\eqref{E:IBDRYDEF} is central to the proof of decay of the gravitational force field. 
Recalling the decomposition~\eqref{eq:gTaylorsplitting}, we want to split the product as
\beqa
\widehat{g}(m,\bI)S_m(\th(R,\bI))=\widehat{\Poly_{\n}[g]}(m,\bI,R)S_m(\th(R,\bI)) +\widehat{\Rem_{\n}[g]}(m,\bI,R)S_m(\th(R,\bI)).
\eeqa
The advantage of this splitting is that, for the first term, the polynomial structure of the Taylor polynomial interacts well with the separation of variables for $S_m(\th(R,y,z))$ in $(R,y,z)$ coordinates, while in the second term, the remainder enjoys enhanced decay that can absorb singular factors emerging from the application of derivatives. However, in order to avoid artificial loss of regularity issues arising from higher $R$ derivatives applied to the splitting, we carefully use~\eqref{E:RZINTERCHANGE}  to convert $R$ derivatives on the regular function $\widehat{g}$ into $z$ derivatives, while still avoiding $z$ derivatives landing on the Green's function $S_m$.

\begin{lemma}[Product rule]\label{L:PARPRODUCTSPLITTING}
Let $g\in C^{\n}(\oom)$, $m\in\Z_*$, $j\leq \n$. Then, for all $(R,y,z)\in\Upsilon$, we have
\beqa\label{E:PARPRODUCTSPLITTING}
{}&\pa_R^j\Big(\widehat{g}(m,\bI(R,y,z))S_m(R,y,z)\Big)\\
&=\sum_{\ell=0}^j{j \choose \ell}\sum_{i=0}^\ell p_i^\ell\Pi^i\Big(\widehat{\Poly_{\n}[g]}(m,\bI(R,y,z),R)\dE^{-\frac{|m|}{2}}(\bI(R,y,z))\Big)\pa_R^{j-\ell}H_m(R,y,z)\\
&\ \ \ +\sum_{\ell=0}^j{j\choose \ell}\sum_{i=0}^\ell p_{i}^\ell\Pi^{i}\Big(\widehat{\Rem_{\n}[g]}(m,\bI(R,y,z),R)\Big)\pa_R^{j-\ell}S_m(R,y,z),
\eeqa
where $p_i^\ell$ are the analytic functions defined using Bell polynomials in~\eqref{E:BELLDEF}, and the Poisson bracket operator $\Pi$ is defined in~\eqref{E:PIDEF}. 
\end{lemma}

\begin{proof}
We begin by observing that $\widehat{g}(m,\bI(R,y,z))\dE(\bI(R,y,z))^{-\frac{|m|}{2}}$ satisfies the interchange of derivatives identity~\eqref{E:RZINTERCHANGE} since it depends on $(R,y,z)$ only through $\bI$. Thus, recalling also the splitting $S_m(R,y,z)=\dE^{-\frac{|m|}{2}}H_m(R,y,z)$ of Proposition~\ref{prop:Greensdecomp}, we apply the Leibniz rule to find
\begin{align*}
&\pa_R^j\Big(\widehat{g}(m,\bI(R,y,z))S_m(R,y,z)\Big)=\sum_{\ell=0}^j{j\choose \ell}\pa_R^{\ell}\Big(\widehat{g}(m,R,y,z)\dE^{-\frac{|m|}{2}}\Big)\pa_R^{j-\ell}H_m(R,y,z)\\
&=\sum_{\ell=0}^j{j\choose \ell}\sum_{i=0}^\ell p_i^\ell\Pi^i\Big(\widehat{g}(m,R,y,z)\dE^{-\frac{|m|}{2}}\Big)\pa_R^{j-\ell}H_m(R,y,z)\\
&=\sum_{\ell=0}^j{j\choose \ell}\sum_{i=0}^\ell p_i^\ell\Pi^i\Big(\big(\widehat{\Poly_{\n}[g]}(m,\bI,R) +\widehat{\Rem_{\n}[g]}(m,\bI,R)\big)\dE^{-\frac{|m|}{2}}\Big)\pa_R^{j-\ell}H_m(R,y,z),
\end{align*} 
where we have applied the splitting~\eqref{eq:gTaylorsplitting} in the last line. We observe that the contribution from the Taylor polynomial $\Poly_{\n}[g]$ is exactly the first line on the right of~\eqref{E:PARPRODUCTSPLITTING}, and hence we work now with the remainder term. We now distribute derivatives and  apply~\eqref{E:RZINTERCHANGE} to $\dE^{-\frac{|m|}{2}}$ to find
\begin{align}
&\sum_{\ell=0}^j{j\choose \ell}\sum_{i=0}^\ell p_i^\ell\Pi^i\Big(\widehat{\Rem_{\n}[g]}(m,\bI,R)\dE^{-\frac{|m|}{2}}\Big)\pa_R^{j-\ell}H_m(R,y,z) \notag \\
&=\sum_{\ell=0}^j{j\choose \ell}\sum_{i=0}^\ell p_i^\ell\sum_{\ell'=0}^i{i \choose\ell'} \Pi^{\ell'}\Big(\widehat{\Rem_{\n}[g]}(m,\bI,R)\Big)\Pi^{i-\ell'}\dE^{-\frac{|m|}{2}}\pa_R^{j-\ell}H_m(R,y,z) \notag \\
&=\sum_{\ell=0}^j\sum_{i=0}^\ell \sum_{\ell'=0}^i\sum_{k=0}^\ell{j\choose \ell}{\ell\choose k}p_{\ell'}^k \Pi^{\ell'}\Big(\widehat{\Rem_{\n}[g]}(m,\bI,R)\Big)p_{i-\ell'}^{\ell-k}\Pi^{i-\ell'}\dE^{-\frac{|m|}{2}}\pa_R^{j-\ell}H_m(R,y,z),
\label{E:BELL1}
\end{align}
where we have observed (for example from~\eqref{E:RZINTERCHANGE0}) that the operator $\Pi^\ell$ obeys the Leibniz rule and applied the Bell polynomial identity (for example from~\cite{Cvijovic11})
\[
B_{\ell,\ell'+i-\ell'}(x_1,\ldots,x_{\ell-i+1})=\frac{1}{{i\choose \ell'}}\sum_{k=0}^\ell{\ell \choose k}B_{k,\ell'}(x_1,\ldots,x_{k-\ell'+1})B_{\ell-k,i-\ell'}(x_1,\ldots,x_{\ell-k-(i-\ell')+1}).
\]

We note that ${j\choose \ell}{\ell\choose k}={j\choose k}{j-k\choose \ell-k}$
 and recall that, by definition, the Bell polynomials $B_{n,k}$ (and hence $p^n_k$) are identically zero if $k>n$ or if $n=0$, $k\neq 0$. We may thus rewrite the above sums  as
\begin{align*}
&\sum_{k=0}^j\sum_{\ell=k}^j{j\choose k}{j-k \choose \ell-k}\sum_{i=0}^\ell\sum_{\ell'=0}^i p_{\ell'}^k \Pi^{\ell'}\Big(\widehat{\Rem_{\n}[g]}(m,\bI,R)\Big)p_{i-\ell'}^{\ell-k}\Pi^{i-\ell'}\dE^{-\frac{|m|}{2}}\pa_R^{j-\ell}H_m(R,y,z)\\
&=\sum_{k=0}^j\sum_{k'=0}^{j-k}{j\choose k}{j-k \choose k'}\sum_{i=0}^{k+k'}\sum_{\ell'=0}^ip_{\ell'}^k \Pi^{\ell'}\Big(\widehat{\Rem_{\n}[g]}(m,\bI,R)\Big)p_{i-\ell'}^{k'}\Pi^{i-\ell'}\dE^{-\frac{|m|}{2}}\pa_R^{j-k-k'}H_m(R,y,z)\\
&=\sum_{k=0}^j\sum_{k'=0}^{j-k}{j\choose k}{j-k \choose k'}\sum_{\ell'=0}^{k+k'}p_{\ell'}^k \Pi^{\ell'}\Big(\widehat{\Rem_{\n}[g]}(m,\bI,R)\Big)\sum_{i=\ell'}^{k+k'}p_{i-\ell'}^{k'}\Pi^{i-\ell'}\dE^{-\frac{|m|}{2}}\pa_R^{j-k-k'}H_m(R,y,z).
\end{align*}
Observe now that if $\ell'>k$, then $p_{\ell'}^k\equiv 0$, so that we may restrict to $\ell'\leq k$ and re-write
\begin{align*}
&\sum_{i=\ell'}^{k+k'}p_{i-\ell'}^{k'}\Pi^{i-\ell'}\dE^{-\frac{|m|}{2}}=\sum_{n=0}^{k+k'-\ell'}p_n^{k'}\Pi^{n}\dE^{-\frac{|m|}{2}}
=\sum_{n=0}^{k'}p_n^{k'}\Pi^{n}\dE^{-\frac{|m|}{2}}=\pa_R^{k'}\dE^{-\frac{|m|}{2}},
\end{align*}
where we have also used that $p_n^{k'}=0$ for $n>k'$ and~\eqref{E:RZINTERCHANGE}.
We have therefore established that
\begin{align*}
&\sum_{\ell=0}^j{j\choose \ell}\sum_{i=0}^\ell p_i^\ell\Pi^i\Big(\widehat{\Rem_{\n}[g]}(m,\bI,R)\dE^{-\frac{|m|}{2}}\Big)\pa_R^{j-\ell}H_m(R,y,z)\\
&=\sum_{k=0}^j\sum_{k'=0}^{j-k}{j\choose k}{j-k \choose k'}\sum_{\ell'=0}^{k}p_{\ell'}^k \Pi^{\ell'}\Big(\widehat{\Rem_{\n}[g]}(m,\bI,R)\Big)\pa_R^{k'}\dE^{-\frac{|m|}{2}}\pa_R^{j-k-k'}H_m(R,y,z)\\
&=\sum_{k=0}^j{j\choose k}\sum_{\ell'=0}^{k}p_{\ell'}^k \Pi^{\ell'}\Big(\widehat{\Rem_{\n}[g]}(m,\bI,R)\Big)\pa_R^{j-k}\Big(\dE^{-\frac{|m|}{2}}H_m(R,y,z)\Big),
\end{align*}
as required.
\end{proof}


\begin{prop}\label{P:KEY}
Let $g\in C^{\n}(\oom)$. There exists $C>0$ such that, for all $j,k'\in\N$ with $j+k'\leq \min\{\n,N+1\}$, for all $(R,y,z)\in\Upsilon$, $m\in\Z_*$,
\beqa
{}&\Big|\pa_R^j\pa_y^{k'}\Big(\widehat{g}(m,E(R,y,z),L(R,y,z))S_m\big(\th(R,E(R,y,z),L(R,y,z))\big)\Big)\Big|\\
&\leq C\|g\|_{C^{\n}(\oom)}\Big(C^{|m|}|m|^{j+k'}\dE^{\max\{\frac{m-j-k'}{2},0\}}\mathbbm{1}_{\J_{R,\de}}+|m|^{-\n+j+k'}\dE^{\frac{\n-j-k'}{2}}\Big).
\eeqa
\end{prop}


\begin{proof}
We take a $\de>0$ sufficiently small and first establish the claimed estimate for $(R,y,z)\in\Ude$. Recalling Lemma~\ref{L:PARPRODUCTSPLITTING}, we first consider the Taylor polynomial contribution. First, we note that, after re-writing $r-r_{L_R}=r-r_L+r_L-r_{L_R}$, we adjust the coefficients of the polynomial by smooth functions of $L$ and $R$ to find
\beqa
\widehat{\Poly_{\n}[g]}(\th,\bI,R)=\sum_{j_1+j_2+j_3\leq k-1}g_{j_1,j_2,j_3}(L,R)(r-r_L)^{j_1}w^{j_2}(L-L_R)^{j_3},
\eeqa
where the coefficients $g_{j_1,j_2,j_3}(L,R)$ are smooth in $L$ and satisfy $|\pa_L^jg_{j_1,j_2,j_3}(L,R)|\leq C\|g\|_{C^{\n-1}(\oom)}$.
We now apply Lemma~\ref{lemma:Tm} and Proposition~\ref{prop:Greensdecomp} to find
\beqa
{}&\widehat{\Poly_{\n}[g]}\dE^{-\frac{|m|}{2}}=\sum_{j_1+j_2+j_3\leq \n-1}g_{j_1,j_2,j_3}(L,R)\dE^{\max\{0,\lceil\frac{j_1+j_2-|m|}{2}\rceil\}}\mathcal{T}_{m,j_1,j_2}(\bI(R,y,z)).
\eeqa
Therefore, substituting this into the first line on the right of~\eqref{E:PARPRODUCTSPLITTING} and applying $\pa_y^{k'}$, the resulting expression is  bounded by
\begin{align}
& C \sum_{j_1+j_2+j_3\leq \n-1}\sum_{\ell=0}^j\sum_{i=0}^\ell\sum_{n=0}^{k'}  \notag\\
& \quad \big|\pa_y^{n}\left(p^\ell_i \Pi^i\big(g_{j_1,j_2,j_3}(L,R)\dE^{\max\{0,\lceil\frac{j_1+j_2-|m|}{2}\rceil\}}\mathcal{T}_{m,j_1,j_2}(\bI(R,y,z))\big)\right)\big|  \big|\pa_y^{k'-n}\pa_R^{j-\ell}H_m\big|. \label{E:POLYBOUND}
\end{align}
We now observe that for any $i\ge0$,  by virtue of~\eqref{E:PIDEF}, we may write the operator $\Pi^i$ as
\begin{align}\label{E:PII}
\Pi^i =\sum_{i'=0}^i a^i_{i'}(R,y,z) \pa_z^{i'}
\end{align}
for some smooth functions $a^i_{i'}:\Upsilon\to\mathbb R$. Thus upon applying product rule, term $\big|\pa_y^{n}\left(p^\ell_i \Pi^i\big(g_{j_1,j_2,j_3}(L,R)\dE^{\max\{0,\lceil\frac{j_1+j_2-|m|}{2}\rceil\}}\mathcal{T}_{m,j_1,j_2}(\bI(R,y,z))\big)\right)\big|$ on the right-hand side of~\eqref{E:POLYBOUND} is bounded by the sum of the terms of the form
\[
C\big|\pa_y^{n_1}p^\ell_i \pa_y^{n_2}a^i_{i'}\pa_y^{n_3}\pa_z^{i'}\big(g_{j_1,j_2,j_3}(L,R)\dE^{\max\{0,\lceil\frac{j_1+j_2-|m|}{2}\rceil\}}\mathcal{T}_{m,j_1,j_2}(\bI(R,y,z))\big)\big| 
\]
where the nonnegative integers $n_i$ satisfy $n_1+n_2+n_3=n$. Thus by~Lemma~\ref{lemma:Tmyz}, bound~\eqref{E:PIJBOUND}, the smoothness of the functions $a^i_{i'}, g_{j_1,j_2,j_3}$, and
analyticity of $\dE^{\max\{0,\lceil\frac{j_1+j_2-|m|}{2}\rceil\}}$ as  a function of $(R,y,z)$, we obtain
\begin{align}
\big|\pa_y^{n}\left(p^\ell_i \Pi^i\big(g_{j_1,j_2,j_3}(L,R)\dE^{\max\{0,\lceil\frac{j_1+j_2-|m|}{2}\rceil\}}\mathcal{T}_{m,j_1,j_2}(\bI(R,y,z))\big)\right)\big| \le C\|g\|_{C^{\n-1}(\oom)} C^{|m|}.
\end{align}
We now estimate the remaining term in~\eqref{E:POLYBOUND}  with Proposition~\ref{prop:Greensdecomp} to bound
\[
\big|\pa_y^{k'-n}\pa_R^{j-\ell}H_m\big| \le  C^{|m|}|m|^{k'+j-n-\ell} \dE^{\frac12\max\{|m|-(k'+j-n-\ell),0\}}\le C C^{|m|}|m|^{j+k'}\dE^{\max\{\frac{m-j-k'}{2},0\}}.
\]
Combining these estimates, we conclude the bound 
\beqa
{}&\bigg|\pa_y^{k'}\left[\sum_{\ell=0}^j{j \choose \ell}\sum_{i=0}^\ell p_i^\ell\Pi^i\Big(\widehat{\Poly_{\n}[g]}(m,\bI(R,y,z),R)\dE^{-\frac{|m|}{2}}(\bI(R,y,z))\Big)\pa_R^{j-\ell}H_m(R,y,z)\right]\bigg|\\
&\leq C_{\n}\|g\|_{C^{\n-1}(\oom)} C^{|m|}|m|^{j+k'}\dE^{\max\{\frac{|m|-j-k'}{2},0\}}.
\eeqa

To handle the remainder term from~\eqref{E:PARPRODUCTSPLITTING}, we note that by the same product rule considerations as above, the 
expression 
\[
\lv\pa_y^{k'}\Big(\sum_{\ell=0}^j{j\choose \ell}\sum_{i=0}^\ell p_{i}^\ell\Pi^{i}\Big(\widehat{\Rem_{\n}[g]}(m,\bI(R,y,z),R)\Big)\pa_R^{j-\ell}S_m(R,y,z)\Big)\rv
\]
 is bounded by a constant multiple 
of the sum of the form
\begin{align}\label{E:REMONE}
\sum_{\ell=0}^j\sum_{i=0}^\ell\sum_{n=0}^{k'}\sum_{i'=0}^i\sum_{n_1+n_2+n_3=n} |\pa_y^{n_1}p^\ell_i| |\pa_y^{n_2}a^i_{i'}| |\pa_y^{n_3}\pa_z^{i'}\Big(\widehat{\Rem_{\n}[g]}(m,\bI(R,y,z),R)\Big) | |\pa_y^{k'-n}\pa_R^{j-\ell}S_m|.
\end{align}
Applying~\eqref{E:PIJBOUND}, Proposition~\ref{prop:remainder} and~\eqref{ineq:dyRsin}, each summand in~\eqref{E:REMONE} is bounded by
\begin{align}
& C \dE^{\max\{i-\frac{\ell}{2},0\}} \|g\|_{C^{\n}(\oom)} |m|^{-\n+n_3+i'} \dE^{\frac{\n}{2}-\frac{n_3}{2}-i'} |m|^{k'+j-n-\ell}\dE^{-\frac{k'+j-n-\ell}{2}} \notag\\
& = C\|g\|_{C^{\n}(\oom)} |m|^{-\n+j+k'} \dE^{\frac{\n-j-k'}{2}}  |m|^{(n_3-n)+(i'-\ell)} \dE^{\max\{i-\frac{\ell}{2},0\} + \frac{\ell}{2}-i' +\frac{n-n_3}{2} } \notag\\
& \le C\|g\|_{C^{\n}(\oom)} |m|^{-\n+j+k'}\dE^{\frac{\n-j-k'}{2}}, \label{E:ESTFINAL}
\end{align}
where we have used the bounds $n_3\le n$, $i'\le \ell$ to conclude that $|m|^{(n_3-n)+(i'-\ell)}\le1$, and the bound $i'\le i$ to conclude that
$\max\{i-\frac{\ell}{2},0\} + \frac{\ell}{2}-i' +\frac{n-n_3}{2} \ge \max\{i-\frac{\ell}{2},0\} + \frac{\ell}{2}-i \ge0$. This concludes the bound in the region $\Ude$. 

For $(R,y,z)\in\Ude^c$, we recall Lemma~\ref{L:ELNONTRAPPINGREG} and Lemma~\ref{L:HRMAINEST} to deduce directly that
\beq
\Big|\pa_R^j\pa_y^k\big(\widehat{ g}S_m\big)\Big|\leq C(\n,\de)\|g\|_{C^{\n}(\oom)} |m|^{-\n+j+k'}\leq C\|g\|_{C^{\n}(\oom)} |m|^{-\n+j+k'}\dE^{\frac{\n-j-k'}{2}}
\eeq
as required.
\end{proof}


\begin{remark}
We note that in the estimate~\eqref{E:ESTFINAL} a crucial cancellation took place reflecting the presence of powers of $\Psi_L'(R)$ through the bound~\eqref{E:PIJBOUND}. This bound
``talks" to the  indices of the Bell polynomials in a precise way that 
recovers half the powers of $\dE$ lost through the application of $z$-derivatives in Proposition~\ref{prop:remainder}.
\end{remark}

\section{Decay for the pure transport equation}\label{S:PT}

In this section we describe the decay rates of the gravitational potential associated with the solutions of the pure transport equation~\eqref{E:PTINTRO0}.

To facilitate notation, we introduce a function 
\begin{align}
q(R,\bI) &: = \frac{|\varphi'(\bI)|}{|\P_R\om(R,\bI)|}, \ \ (R,\bI)\in [\Rmin,\Rmax]\times\I, \label{E:LITTLEQDEF}
\end{align}
where we recall~\eqref{E:PRDEF};
this notation will be used throughout the rest of the paper. From Remark~\ref{REM:SSDERIVS}, we observe that $q$, defined on $\Upsilon$ via the change of variables~\eqref{E:RYZ}, satisfies 
\beq\label{E:QREG}
q(R,y,z)\in C^\reg(\overline{\Upsilon}),
\eeq
where we recall $\reg=\min\{\mu-1,\nu,\n\}$ from~\eqref{E:REGDEFMAIN}.

\begin{theorem}\label{thm:PTdecay}
Let $\fmn$ be a steady state as in~\eqref{E:SS}--\eqref{E:SS2} such that $\eta\in[0,\eta_0)$, where $\eta_0$ is the small constant introduced in Lemma~\ref{L:RYZ}.
Assume that $f_0\in C^{\n}(\oom)$ satisfies the orthogonality condition~\eqref{E:FINORTH} for some $\n\in\mathbb N$, $\n\ge2$. Let $\FPT(t,\cdot)$ denote the solution to the pure transport equation~\eqref{E:PTINTRO0}
with initial datum $\FPT(0,\cdot)=f_0(\cdot)$. 
Then there exists a $C>0$, depending on $\n,\mu,\nu$, such that
\beq
\|\pa_R^a U_{|\varphi'|\FPT}(t,\cdot)\|_{C^0}\leq C\|f_0\|_{C^{\reg}(\oom)}(1+t)^{-(\reg-a)}, \ \ a\in\{1,\dots,\reg\}.
\eeq
\end{theorem}


\begin{proof}
We use the solution formula~\eqref{E:FPTINTRO} together with
Lemma~\ref{L:REPRESENTATIONS} and the Plancherel theorem to conclude that for any $R\in[\Rmin,\Rmax]$
\begin{align}
R^2&\, \pa_RU_{|\varphi'|\FPT}(t,R)   = 4\pi^2 \int_{\I}\int_{\mathbb S^1} h_R(r(\theta,\bI))\, |\varphi'(\bI)| \FPT(t,\th,\bI) T \diff\theta\diff\bI \notag\\
& = 4\pi^2 \sum_{m\in\Z^\ast}\int_{\I} |\varphi'(\bI)|\widehat\FPT(t,m,\bI) \overline{\widehat{h_R}(m,\bI)} T \diff\bI\notag \\
& = 4\pi \sum_{m\in\Z^\ast}\frac1m \int_{\I\cap\{E\geq \Psi_L(R)\}} e^{-2\pi mi \om(\bI) t} |\varphi'(\bI)| \fhat_0(m,\bI) S_m(\th(R,\bI))  T \diff\bI, \label{E:GP1}
\end{align}
where in the last line we have also used that $\widehat{h_R}$ is supported on $E\geq \Psi_L(R)$, see Lemma~\ref{L:REPRESENTATIONS}.

We now change variables $(R,\bI)\to(R,y,z)$ in~\eqref{E:GP1}
and using~\eqref{E:LITTLEQDEF} for any $0\le b \le \reg$ we obtain
\begin{align}
{}& \pa_RU_{|\varphi'|\FPT}(t,R) = \frac{4\pi}{R^2} \sum_{m\in\Z^\ast}\frac1m \int_{\J_R} e^{-2\pi i m y t} q \fhat_0(m,R,y,z) S_m(\th(R,y,z))  T \diff (y,z) \notag\\
& = \frac{4\pi}{R^2}  \frac{1}{(2\pi it)^{b}}\sum_{m\in\Z^\ast}\frac1{m^{b+1}} \int_{\J_R} e^{-2\pi mi y t} \pa_y^{b}\left(qT \fhatm(R,y,z) S_m(\th(R,y,z)) \right) \diff (y,z) \notag\\
& =  \frac{4\pi}{R^2(2\pi it)^{b}}\sum_{m\in\Z^\ast}\frac1{m^{b+1}} \sum_{\ell=0}^{b} {b \choose \ell} \int_{\J_R} e^{-2\pi mi y t} \pa_y^{b-\ell}(qT) \pa_y^\ell\left(\fhatm(R,y,z) S_m(\th(R,y,z)) \right) \diff (y,z)\notag\\
& = \frac{4\pi}{R^2(2\pi it)^{b}}\sum_{m\in\Z^\ast}\frac1{m^{b+1}} \sum_{\ell=0}^{b} {b \choose \ell} \PP_m[\pa_y^{b-\ell}(qT), \pa_y^\ell\left( \fhatm S_m \right) ], \label{E:PTDECAY00}
\end{align}
where we have introduced the notation
\be
\PP_m[Q,V](t,R) = \int_{\J_R} e^{-2\pi i m y t} Q(R,y,z) V(R,y,z) \diff (y,z).
\ee
For any $0\le a \le \reg-b$, we apply $\pa_R^a$ to~\eqref{E:PTDECAY00} and we obtain
\begin{align}
\pa_R^{a+1}U_{|\varphi'|\FPT}(t,R)= &\frac{4\pi}{(2\pi it)^{b}}\sum_{m\in\Z^\ast}\frac1{m^{b+1}} \sum_{\ell=0}^{b}\sum_{\ell_1+\ell_2+\ell_3=a} {b \choose \ell} {a \choose \ell_1,\ell_2,\ell_3} \times \notag\\
& \qquad \times \pa_R^{\ell_2}(R^{-2})
\PP_m[\pa_R^{\ell_3}\pa_y^{b-\ell}(qT), \pa_R^{\ell_1}\pa_y^\ell\left( \fhatm S_m \right) ],
\end{align}
where we have used the notation ${b \choose \ell,j,n}=\frac{b!}{\ell ! j ! n!}$.
We note that the $R$-dependent portion of the boundary $\pa\J_R$ is contained in the image of the vacuum boundary $\pa \I_{\text{vac}}$ (compare~\eqref{E:JRMOVINGBDRY}) and so does not contribute to this expression as $\pa_R^{\ell_3}\pa_y^{b-\ell}(qT)$ vanishes at the vacuum boundary for all $\ell_3\le a$ and $\ell\le b$ (compare~\eqref{E:LITTLEQDEF} and Remark~\ref{REM:SSDERIVS}).
For any $\ell\in\{0,\dots,b\}$, $\ell_1\le a$, we use~Proposition~\ref{P:KEY}  and obtain
\begin{align}
&|\PP_m[\pa_R^{a-\ell_1}\pa_y^{b-\ell}(qT), \pa_R^{\ell_1}\pa_y^\ell\left( \fhatm S_m \right) ]| \le \int_{\J_R}|\pa_R^{a-\ell_1}\pa_y^{b-\ell}(qT)| | \pa_R^{\ell_1}\pa_y^\ell\big(\fhat_{0,m} S_m\big)| \diff(y,z) \notag\\
& \le C \|f_0\|_{C^{\ell+\ell_1}(\oom)}  \int_{\J_R} \big((C|m|)^{\ell+\ell_1}(C^2\dE)^{\max\{\frac{|m|-\ell-\ell_1}{2},0\}}\chardelta+1\big)\diff(y,z)\notag\\
& \le  C \|f_0\|_{C^{\ell+\ell_1}(\oom)} \big((C|m|)^{\ell+\ell_1}(C^2\de)^{\max\{\frac{|m|-\ell-\ell_1}{2},0\}}+1\big), \label{E:PTDECAY0}
\end{align}
where we have used $\|qT\|_{C^\reg(\bar{\Upsilon})}\leq C$ by~\eqref{E:QREG} in the second line and the bound $\dE\le \delta$ in the last.
Therefore
\begin{align}
&\sum_{m\in\Z^\ast}\frac1{|m|^{b+1}}\big| \sum_{\ell=0}^{b}\sum_{\ell_1+\ell_2+\ell_3=a} {b \choose \ell} {a \choose \ell_1,\ell_2,\ell_3} \PP_m[\pa_R^{\ell_3}\pa_y^{b-\ell}(qT), \pa_R^{\ell_1}\pa_y^\ell\left( \fhat_{0,m} S_m \right) ]\big| \notag\\
& \le
C  \|f_0\|_{C^{\ell+\ell_1}(\oom)}  \sum_{m\in\Z^\ast}\frac1{|m|^{b+1}} \sum_{j=0}^{\reg} \big((C|m|)^{j}(C^2\de)^{\max\{\frac{|m|-j}{2},0\}}+1\big)\leq C  \|f_0\|_{C^{\ell+\ell_1}(\oom)},
\end{align}
provided $C^2\de<1$. Combining this with~\eqref{E:PTDECAY00} and~\eqref{E:PTDECAY0} we obtain the claim.
\end{proof}

\section{Resolvent potential estimates}\label{S:RESOLVENTBOUNDS}


We assume throughout this section that we are working with a steady state $\fmn$ satisfying~\eqref{E:SS}--\eqref{E:SS2} such that $\eta\in[0,\eta_0)$, where $\eta_0$ is a 
sufficiently small constant so that the theory developed in Sections~\ref{S:AA}--\ref{S:REGULARITY} applies. 
Moreover, we denote by $f$ the solution to the initial value problem~\eqref{E:FULLLIN}--\eqref{E:FULLLININITIAL} with data satisfying the orthogonality condition~\eqref{E:FINORTH}
and $f_0\in C^{\n}(\oom)$, for $k\ge1$. We recall that~\eqref{E:FINORTH} implies that $\fhat_0(0,\bI)=0$, $\bI\in\I$.

\subsection{Gravitational (resolvent) potentials}
The goal of this section is to derive a representation formula for derivatives $\pa_R^a\pa_\la^b\Ueps$ of the gravitational resolvent potentials that is suitable for deriving estimates. This will be achieved below in Proposition~\ref{P:HORES}.

We first derive a  formula for $\big(\L+\la\pm i\eps)^{-1}\widehat{\fin}$ which is the first step towards resolvent potential identities. We recall that $\sL$ is given by~\eqref{E:SIGMALDEF}.


\begin{lemma}\label{L:RESOLVENTFORMULA}
Let $(\l,\k)\in\sL\times \mathbb R_+$ be given. We introduce the notation
\begin{align}\label{E:FDEF0}
\widehat{f^\pm_\eps}:=(\L+\la\pm i\eps)^{-1}\widehat{f}_0, \ \ \Ueps(R;\l) : = U_{|\varphi'|f^\pm_\eps}(R;\l), \ \ R\in[\Rmin,\Rmax],
\end{align}
where we recall the operator $\L$ from~\eqref{DEF:L}.
Then the following formula holds:
\beq\label{def:Fpm}
\widehat{f^\pm_\eps}(m,\bI;\la)=\frac{\widehat{f}_0(m,\bI)}{2\pi m\om(\bI)+\la\pm i\eps}-\frac{\eta \om(\bI)\widehat{\Ueps}(m,\bI,\la)}{\om(\bI)+ \frac{\la\pm i\eps}{2\pi m}}, \  (m,\bI)\in\Z_\ast\times\I.
\eeq
\end{lemma}


\begin{proof}
The proof is trivial and follows by rewriting~\eqref{E:FDEF0} in the form
\beq\label{E:FPMFORMULA}
2\pi m \om(\bI)\big(\widehat{f^\pm_\eps}(m,\bI;\la)+\eta\widehat{\Ueps}(m,\bI;\la)\big)+(\la\pm i\eps)\widehat{f^\pm_\eps}(m,\bI,\la)=\widehat{\fin}(m,\bI)
\eeq
for all $(m,\bI)\in\Z_*\times \I$. 
\end{proof}


\begin{lemma}\label{L:UOUTERRIM}
For any $(\l,\k)\in\sL\times \mathbb R_+$, the resolvent potential $U^\pm_\k(\cdot;\l)$ vanishes on the outer rim of the galaxy, i.e.
\begin{align}
U^\pm_\eps(\Rmax;\l) = 0, \ \ (\l,\k)\in\sL\times \mathbb R_+,
\end{align}
and hence, for $R\in[\Rmin,\Rmax]$,
\beq
U^\pm_\eps(R;\l)=-\int_{R}^{R_{\max}}\pa_r\Ueps(r;\l)\,\diff r.
\eeq
\end{lemma}

\begin{proof}
Equation~\eqref{E:FPMFORMULA} in $(R,w,L)$-variables reads
\beq\label{eq:Fresolventtransport}
\T\big(f^\pm_\eps(R,w,L;\l)+\eta U^{\pm}_{\eps}\big)+(\la\pm i\eps)f^\pm_\eps(R,w,L;\la)=\fin(R,w,L).
\eeq
We integrate over $\Om$ (recall that for each fixed $\eps>0$, the resolvents $f^\pm_\eps$ are regular) to obtain
\begin{align*}
(\la\pm i\eps)\int_\Om f^\pm_\eps(r,w,L;\la)\diff (r,w,L)=\int_\Om \fin(r,w,L)\diff (r,w,L) = M[\fin] = 0,
\end{align*}
by~\eqref{E:ZEROMASS}.
From Lemma~\ref{L:REPRESENTATIONS} and Remark~\ref{R:BDRYPOT} we thus conclude that for any $\l\in\sL$ we have
\begin{align*}
U^\pm_\eps(\Rmax;\l) =  - \frac{4\pi^2}{\Rmax} \int_\Om f^\pm_\eps (r,w,L) \diff (r,w,L) = - \frac{4\pi^2M[\fin]}{(\l\pm i\k)\Rmax}=0.
\end{align*}
\end{proof}


Our next goal is to derive a formal expression for the gradient of the gravitational potential along the flow. To that end 
we introduce the Plemelj singular operator, which will play a key role in our analysis.
\begin{definition}[Plemelj singular integral]
For any $(m,\l,\eps)\in \Z\times \sL\times \mathbb R_+$ and any $Q,\mathcal V\in C^0(\overline{\Upsilon})\cap C^1(\Upsilon)$ (recall $\Upsilon$ from~\eqref{def:Upsilon}), we formally introduce the operator
\begin{align}\label{E:PLEMDEF}
\PL[Q,\mathcal V](R) = \int_{\J_R} \frac{Q(R,\tilde\bI) \mathcal V(R,\tilde\bI)}{y+\frac{\l\pm i\eps}{2\pi m}} \diff (y,z).
\end{align}
\end{definition} 
This is a variant of the classical Plemelj singular integral -- the placeholder notation $Q$ and $\mathcal V$  distinguishes the regular contribution $Q$ from the
possibly singular contribution $V$. 


The operators $\PL$ act on functions $Q,\mathcal V$ seen as functions of $(R,y,z)$-variables. We will often use it in the setting where
$\mathcal V$ is in fact a mixed higher-order $(R,y)$-derivative of a product of the form $ \widehat V_m(\bI)S_m(\th(R,\bI))$. Note that in this case
$\bI=\bI(R,y,z)$ is viewed as a function of $(R,y,z)$.


For future use we record here the classical approximation of unity identity
\begin{align}
2\pi i g(x)&=-\lim_{\eps\to0}\int_\R\big[\frac{g(y)}{y-x+ i\eps }-\frac{g(y)}{y-x- i\eps }\big] \diff y 
=  \lim_{\eps\to0}\int_\R\frac{2i\eps g(y)}{(y-x)^2+\eps^2}\diff y \label{E:AI}, \ \ g\in L^1(\mathbb R).
\end{align}



\begin{lemma}\label{L:PARTIALRUEPSOP}
For any $(\l,\k,R)\in\sL\times \mathbb R_+\times[\Rmin,\Rmax]$, the resolvent gravitational force $\pa_R U^\pm_\eps$, satisfies 
\beq\label{E:PARTIALRUEPSOP}
\pa_R U^\pm_\eps(R;\la) = \frac2{R^2} \sum_{m\in\Z_\ast} \frac1{m^2}\PL[qy^{-1}, \fhatm S_m]
-\frac{4\pi\eta}{R^2}\sum_{m\in \Z_*}\frac1{m} \PL[q, {\wUepsm} S_m],
\eeq
where we recall $q(R,\bI)$ defined in~\eqref{E:LITTLEQDEF}.
\end{lemma}

We interchangeably use the notation $\wUepsm$ to denote the $m$-th Fourier coefficient $\widehat{\Ueps}(m,\bI)$
by analogy to the notation $\fhat_{0,m}$ for $\fhat_0(m,\bI)$.


\begin{proof}
We begin by recalling from Lemma~\ref{L:REPRESENTATIONS} that
\begin{align}
\pa_R \Ueps(R;\la)&=\frac{4\pi^2}{R^2}\int_\Om |\varphi'|f^\pm_\eps(s,w,L;\la) h_R(s,w,L)\,\diff(s,w,L) \notag\\
& =\frac{4\pi^2}{R^2}\sum_{m\in\Z_*}\int_\I|\varphi'|\widehat{f^\pm_\eps}(m,\bI;\la)\overline{\widehat{h_R}(m,\bI)} T(\bI)\diff \bI \notag\\
& = \frac{4\pi^2}{R^2}\sum_{m\in\Z_*}\frac1{\pi m}\int_\I |\varphi'|\widehat{f^\pm_\eps}(m,\bI;\la)S_m(\th(R,\bI)) T(\bI)\diff \bI,
\end{align}
where we have used the identity $\widehat{h_R}(m,\bI)=\frac1{\pi m}S_m(\th(R,\bI))$, $m\in\Z_\ast$ (cf.~Lemma~\ref{L:REPRESENTATIONS}).
Substituting in~\eqref{def:Fpm}, we find
\begin{align}\label{E:PARUFORM}
\pa_R \Ueps(R;\la)=\frac{4\pi^2}{R^2}\sum_{m\in\Z_*}\int_\I\Big(\frac{1}{2\pi^2 m^2}\frac{|\varphi'(\bI)|\fhat_0(m,\bI)}{\om(\bI)+\frac{\la\pm i\eps}{2\pi m}}-\frac1{\pi m}\frac{\eta\om(\bI)|\varphi'(\bI)|\widehat{\Ueps}(m,\bI,\la)}{\om(\bI)+\frac{\la\pm i\eps}{2\pi m}}\Big)S_m(\th(R,\bI))\,T(\bI)\,\diff \bI,
\end{align}
which gives the claim upon changing variables $\bI\to(y,z)$ as defined in Section~\ref{S:YZ} and recalling~\eqref{E:PLEMDEF}.
\end{proof}

Lemma~\ref{L:PARTIALRUEPSOP} suggests that we consider the following nonlocal operator $\FPM: C^1_R\to C^1_R$,
\begin{align}\label{E:FDEF}
\FPM[V](R) = V(R) -\eta \sum_{m\in\Z_\ast}\frac1m\int_R^{\Rmax}\frac{4\pi}{r^2} \PL[q, \widehat V_m S_m] \diff r.
\end{align}
Upon integrating~\eqref{E:PARTIALRUEPSOP} in $R$ (recalling Lemma~\ref{L:UOUTERRIM}), we conclude that
\begin{align}\label{E:FUEPS}
\FPM[\Ueps](R)= -\sum_{m\in\Z_\ast} \frac1{m^2} \int_R^{\Rmax} \frac2{r^2} \PL[qy^{-1}, \fhatm S_m] \diff r,
\end{align}
which therefore allows us to study the resolvent potential $\Ueps$ through invertibility properties of the operator $\FPM$. Recall here that 
$q$ is defined in~\eqref{E:LITTLEQDEF}.


\begin{lemma}\label{L:PARPAL1}
Let $a,b\in\mathbb N_0$ be given. Assume that $Q\in C^{a+b}(\overline{\Upsilon})$ is such that, for each $R\in[\Rmin,\Rmax]$, all derivatives $\pa_y^n\pa_R^sQ(R,y,z)$ with $n\leq b$, $s\leq a$ vanish at $\Jvac$ and let $V:\overline{\Upsilon}\times\sL\to\mathbb{C}$  be such that, for all $j+\ell\leq b$, $s\leq a$, we have $\pa_y^\ell\pa_R^s\pa_\la^j V\in C^0(\overline{\Upsilon}\times\sL)$.  Then 
\begin{align}
&\pa_R^a\pa_\l^b \PL[Q, V]  =\sum_{s=0}^a \sum_{\ell+j+n=b} {a \choose s } {b \choose \ell,j,n} \frac{(-1)^{b-j}}{(2\pi m)^{b-j}}  \PL[\pa_y^{n}\pa_R^{a-s}Q,\pa_y^\ell\pa_R^s\big(\pa_\l^jV\big)],\label{E:PARL}
\end{align}
where we have used the notation ${b \choose \ell,j,n}=\frac{b!}{\ell ! j ! n!}$.
\end{lemma}


\begin{proof} 
We first observe the simple identity that for any $b\in\mathbb N$ we have
\begin{align}
\pa_\l^b\Big(\frac{V}{y+\frac{\l\pm \eps}{2\pi m}}\Big) = \sum_{j=0}^b (-1)^{b-j} (b-j)! {b \choose j} \frac{1}{(2\pi m)^{b-j}} \pa_\l^j V (y+\frac{\l\pm i\eps}{2\pi m})^{-(b-j+1)}. \label{E:PALHIGH1}
\end{align} 
Thus,  
\begin{align}
&\pa_\l^b\PL[Q,V] = \int_{\J_R} Q \pa_\l^b\big(\frac{ V}{y+\frac{\l\pm i\eps}{2\pi m}}\big) \diff (y,z) \notag\\
& =  \sum_{j=0}^b (-1)^{b-j} (b-j)! {b \choose j} \frac{1}{(2\pi m)^{b-j}} \int_{\J}Q \pa_\l^j  V (y+\frac{\l\pm i\eps}{2\pi m})^{-(b-j+1)}\diff (y,z) \notag \\
& =\sum_{j=0}^b (-1)^{b-j} {b \choose j} \frac{1}{(2\pi m)^{b-j}}  \sum_{\ell=0}^{b-j}{b-j\choose \ell}\PL[\pa_y^{b-j-\ell}Q,\pa_y^\ell\pa_\l^j V],
\end{align}
where in the last line, we have used
 \[ (-1)^{b-j}(b-j)!(y+\frac{\l\pm i\eps}{2\pi m})^{-(b-j+1)}=\pa_y^{b-j}(y+\frac{\l\pm i\eps}{2\pi m})^{-1}\]
 and have recalled from Remark~\ref{R:YIBP} that, for each $z_0\ge0$, the level set $\{z=z_0\}$ is a smooth curve in $\J_R$ with both end-points on $\Jvac$ to see the vanishing of the boundary terms in the integration by parts due to the assumptions on $Q$.
We also note that 
${b \choose j}{b-j\choose \ell} = {b \choose j, \ell, b-j-\ell}$. A simple application of this identity and the Leibniz rule then gives
\begin{align}
&\pa_R^a\pa_\l^b \PL[Q,V] = \sum_{s=0}^a {a \choose s } \pa_\l^b\PL[\pa_R^{a-s}Q,\pa_R^sV] \notag\\
&=\sum_{s=0}^a {a \choose s } \sum_{\ell+j+n=b} {b \choose \ell,j,n} \frac{(-1)^{b-j}}{(2\pi m)^{b-j}}  \PL[\pa_y^{n}\pa_R^{a-s}Q,\pa_y^\ell\pa_R^s\pa_\l^j V],\notag
\end{align}
thus showing~\eqref{E:PARL}. We note that the $R$-dependent boundary $\pa\J_R$ does not contribute to this expression as $\pa_R^{a-s}Q$ vanishes at the vacuum boundary for all $s\le a$, compare~\eqref{E:JRMOVINGBDRY}.
\end{proof}


As a simple corollary of the previous lemma and Lemma~\ref{L:PARTIALRUEPSOP} we derive a formula for the high order derivatives of the resolvent potential, where we recall that, for each $\eps>0$, the resolvent operators are locally analytic in $\la$.


\begin{prop}\label{P:HORES}
For any $a,b\in\mathbb N_0$, $a\ge1$, such that $a+b\le\reg$ (recall the regularity index $\reg$ defined in~\eqref{E:REGDEFMAIN}) we have, for each $\eps>0$ and $R\in[\Rmin,\Rmax]$, $\la\in\sL$, the identity
\begin{align}\label{E:ABKEY}
\pa_R^a\pa_\l^b\Ueps =  \mathcal S_{a,b}^{\pm,\l,\eps}[f_0] + \mathcal R^{\pm,\l,\eps}_{a,b}[\Ueps],
\end{align}
where
\begin{align}\label{E:SDEF}
\mathcal S_{a,b}^{\pm,\l,\eps}[f_0]=& \frac2{R^2}\sum_{m\in\Z_\ast}\frac1{m^2}  \sum_{s=0}^{a-1} \sum_{\ell=0}^b c^{a-1;b}_{s;\ell}\frac{(-1)^{b}}{(2\pi m)^{b}}  \PL[\pa_y^{b-\ell}\pa_R^{a-1-s}(qy^{-1}),\pa_y^\ell\pa_R^s\big(\fhatm S_m\big)]  
\end{align}
and
\begin{align}
&\mathcal R^{\pm,\l,\eps}_{a,b}[V] = - 2(a-1)\frac1R \pa_R^{a-1}\pa_\l^bV  - (a-1)(a-2) \frac1{R^2}  \pa_{R}^{a-2}\pa_\l^bV \notag\\
&- \frac{4\pi \eta}{R^2} \sum_{m\in \Z_*}\frac1{m} \sum_{0\le s \le a-1\atop \ell+j+n=b} c^{a-1;b}_{s;\ell,j,n} \frac{(-1)^{b-j}}{(2\pi m)^{b-j}}  \PL[\pa_y^{n}\pa_R^{a-1-s}q,\pa_y^\ell\pa_R^s\big(\pa_\l^j \Vhat_m S_m\big)],
\label{E:RDEF}
\end{align}
and we recall~\eqref{E:LITTLEQDEF}.
Here we have used the abbreviations
\[
c^{a-1;b}_{s;\ell}:={a-1 \choose s } {b \choose \ell}, \ \ c^{a-1;b}_{s;\ell,j,n}:={a-1 \choose s } {b \choose \ell,j,n}.
\]
\end{prop}


\begin{proof}
From Lemma~\ref{L:PARTIALRUEPSOP} we have the identity
\begin{align}
R^2\pa_R \Ueps(R;\l) = 2 \sum_{m\in\Z_\ast} \frac1{m^2}\PL[qy^{-1},\fhat_{0,m}S_m]
-4\pi\eta\sum_{m\in \Z_*}\frac1{m} \PL[q, \wUepsm S_m].
\end{align}
Therefore from Lemma~\ref{L:PARPAL1} we conclude
\begin{align}
&\pa_R^{a-1}\pa_\l^b \big(R^2\pa_R U^\pm_\eps(r;\la)\big) \notag\\
&=  2\sum_{m\in\Z_\ast}\frac1{m^2}\pa_R^{a-1}\pa_\l^b \PL[qy^{-1},\fhat_{0,m}S_m]  
- 4\pi \eta \sum_{m\in \Z_*}\frac1{m}\pa_R^{a-1}\pa_\l^b \PL[q,\wUepsm S_m]  \notag\\
& = 2\sum_{m\in\Z_\ast}\frac1{m^2} \sum_{s=0}^{a-1} \sum_{\ell=0}^b {a-1 \choose s } {b \choose \ell} \frac{(-1)^{b}}{(2\pi m)^{b}}  \PL[\pa_y^{b-\ell}\pa_R^{a-1-s}(qy^{-1}),\pa_y^\ell\pa_R^s\big(\fhat_{0,m} S_m\big)] \notag\\
& \ \ \ \ - 4\pi \eta\sum_{m\in \Z_*}\frac1{m}\sum_{s=0}^{a-1} \sum_{\ell+j+n=b} {a-1 \choose s } {b \choose \ell,j,n} \frac{(-1)^{b-j}}{(2\pi m)^{b-j}}  \PL[\pa_y^{n}\pa_R^{a-1-s}q,\pa_y^\ell\pa_R^s\big(\pa_\l^j\wUepsm S_m\big)],\notag
\end{align}
which concludes the proof. Note that we used that $q$ defined in~\eqref{E:LITTLEQDEF} vanishes to order $\al\ge\reg$ at $\pa\J_R$, where $\al$ is as defined in~\eqref{def:alpha}. 
From Proposition~\ref{P:KEY} we obtained the required regularity of the product $\wUeps S_m$ with respect to $R$ and $y$ to apply Lemma~\ref{L:PARPAL1}, where we have recalled that for each $\eps>0$, the resolvent $\Ueps$ is regular.
\end{proof}


\begin{corollary}\label{C:PURELA}
For any $0\le b\le \reg$  we obtain the identity
\begin{align}\label{E:PURELA}
\FPM[\pa_\l^b\Ueps] & =-  \int_{R}^{\Rmax}\mathcal S^{\pm,\l,\eps}_{1,b}[f_0] \diff r  -\int_{R}^{\Rmax}\tilde{\mathcal R}^{\pm,\l,\eps}_{1,b}[\Ueps] \diff r \\
\Fminus[\pa_\l^b\Uepsplus-\pa_\l^b\Uepsminus] & =-  \int_{R}^{\Rmax}(\mathcal S^{+,\l,\eps}_{1,b}-\mathcal S^{-,\l,\eps}_{1,b})[f_0] \diff r 
-  \int_{R}^{\Rmax}(\mathcal R^{+,\l,\eps}_{1,b}-\mathcal R^{-,\l,\eps}_{1,b})[\Uepsplus] \diff r \notag\\
& \ \ \ \ -   \int_{R}^{\Rmax} \tilde{\mathcal R}_{1,b}^{-,\l,\eps}[\Uepsplus-\Uepsminus]\diff r, \label{E:PURELAPM}
\end{align}
where for any $a\ge1$ we let
\begin{align}
&\tilde{\mathcal R}^{\pm,\l,\eps}_{a,b}[V] = 
- \frac{4\pi \eta}{R^2} \sum_{m\in \Z_*}\frac1{m} \sum_{0\le s \le a-1\atop \ell+j+n=b, j<b} c^{a-1;b}_{s;\ell,j,n} \frac{(-1)^{b-j}}{(2\pi m)^{b-j}}  \PL[\pa_y^n\pa_R^{a-1-s}q,\pa_y^\ell\pa_R^s\big(\pa_\l^j \Vhat_m S_m\big)];
\label{E:TILDERDEF}
\end{align}
we recall the operators $\FPM$, $\mathcal S^{\pm,\l,\eps}_{1,b}$, and $\mathcal R^{\pm,\l,\eps}_{1,b}$ defined in~\eqref{E:FDEF},~\eqref{E:SDEF}, and~\eqref{E:RDEF} respectively and~\eqref{E:LITTLEQDEF}.
\end{corollary}


\begin{proof}
We let $a=1$ in Proposition~\ref{P:HORES} and obtain
\begin{align}
\pa_R \pa_\l^bU^\pm_\eps(R;\la) & = \mathcal S_{1,b}^{\pm,\l,\eps}[f_0] + \mathcal R_{1,b}^{\pm,\l,\eps}[\Ueps]  \notag\\
& = \frac{2}{R^2}\sum_{m\in\Z_\ast}\frac1{m^2} \sum_{\ell=0}^b  {b \choose \ell} \frac{(-1)^{b}}{(2\pi m)^{b}}  \PL[\pa_y^{b-\ell}(qy^{-1}),\pa_y^\ell\big(\fhat_{0,m} S_m\big)] \notag\\
& \ \ \ \ - \frac{4\pi\eta}{R^2} \sum_{m\in \Z_*}\frac1{m} \sum_{\ell+j+n=b} {b \choose \ell,j,n} \frac{(-1)^{b-j}}{(2\pi m)^{b-j}}  \PL[\pa_y^{n}q,\pa_y^\ell\big(\pa_\l^j\wUepsm S_m\big)]. \label{E:HIGHERB}
\end{align}
By absorbing the $j=b$ summand in $\mathcal R_{1,b}^{\pm,\l,\eps}$ into the left-hand side and integrating-in-$R$ we obtain~\eqref{E:PURELA} after recalling Lemma~\ref{L:UOUTERRIM}. Identity~\eqref{E:PURELAPM}
follows similarly, after taking the difference in~\eqref{E:HIGHERB}.
\end{proof}

We note for later convenience from Lemma~\ref{L:PARPAL1} that
\beq\label{E:RTILDEIDENTITY}
\tilde{\mathcal R}^{\pm,\l,\eps}_{a,b}[V]=-\frac{4\pi\eta}{R^2}\sum_{m\in\Z_*}\frac{1}{m}\pa_R^{a-1}\pa_\l^b\PL[q,\widehat{V}_m S_m]=\frac{1}{R^2}\pa_R^{a-1}\big(R^2\tilde{\mathcal R}^{\pm,\l,\eps}_{1,b}[V]\big).
\eeq


\subsection{Frequency splitting and Plemelj bounds}


Fix $\mu\in(0,\frac{\ommin}{\ommax})\subset (0,1)$, where we recall~\eqref{E:OMEGAMINMAX}. 
\begin{definition}[Near-resonant set]\label{D:RESSET}
To any $\l\in\sL$ (recall~\eqref{E:SIGMALDEF}) we associate the set
\begin{align}\label{E:RESFREQ}
\Res(\l) := \left\{m\in\Z_\ast\, \big|\, \exists \bI\in\I \ \text{ such that } \  \lv\om(\bI)+\frac{\l}{2\pi m}\rv<  \mu\ommin \right\}.  
\end{align}
We refer to $\Res(\l)$ as the {\em near-resonant set}.
We use the notation $\Res(\l)^c:=\Z_\ast\setminus \Res(\l)$ to refer to the set of {\em non-resonant frequencies}.
\end{definition}


The following lemma summarises important properties of the near-resonant set.


\begin{lemma}
For any $\ga\in[0,1]$ there exists a constant $C = C(\ga)>0$ such that  for any $(m,\l)\in \Res(\l)^c\times\sL$, $\bI\in \I$, we have 
\begin{align}
\frac1{\lv\om(\bI)+\frac{\l}{2\pi m} \pm  \frac{i\eps}{2\pi m}\rv} \le C\Big(\frac{ |m|}{|\l|}\Big)^\ga,\label{E:NONRESBOUND}\\
\big|\frac{1}{\om(\bI)+\frac{\la+i\eps}{2\pi m}}-\frac{1}{\om(\bI)+\frac{\la-i\eps}{2\pi m}}\big|\leq C\frac{\eps}{|m|}\frac{|m|^{2\ga}}{|\la|^{2\ga}}.\label{ineq:nonresfrediff}
\end{align}
Moreover, there exists a constant $C_0>0$ such that
\begin{align}
|\Res(\l)| \leq &\, C_0|\l|, \ \ \l\in\sL, \label{E:RLAMBDASIZE}
\end{align}
and there exists a constant $c_0=c_0(\ommin,\ommax)$ such that
\begin{align}
 |m|\geq &\,c_0|\la|\text{ for all } m\in \Res(\la), \ \l\in\sL.\label{E:RLAMBDAMIN}
\end{align}
\end{lemma}


\begin{proof}
For $m\in \Res(\l)$, by definition there exists an $\bI\in\I$ such that 
\[
-\mu\ommin-\ommax \le -\mu\ommin-\om(\bI)<\frac{\l}{2\pi m}<\mu\ommin - \om(\bI)\le(\mu-1)\ommin<0.
\]
It follows in particular that $\mu\om_0+\om_\ast>|\frac{\la}{2\pi m}|>(1-\mu)\om_0$, which implies~\eqref{E:RLAMBDASIZE}--\eqref{E:RLAMBDAMIN}.
For $m\in \Res(\l)^c$ we clearly have $|\om(\bI)+\frac{\l}{2\pi m}|>\mu\ommin$ for all $\bI\in\I$. Using this, bounds~\eqref{E:NONRESBOUND}--\eqref{ineq:nonresfrediff} are obvious.
\end{proof}


In the next theorem we provide estimates on the Plemelj operators~\eqref{E:PLEMDEF}, which will subsequently be used
as building blocks for the resolvent estimates.

\begin{theorem}\label{T:PLEMELJ}
Let $(\l,\k)\in \sL\times \R_{\ge0}$ and let 
$V\in C^{\n}(\overline{\Om})$ be given, where we recall~\eqref{E:SUPPORT} and suppose $\n\leq N+1$. Assume further that $Q\in C^1(\overline\Upsilon)$ is such that $Q$ and $\pa_y Q$ vanish on $\Jvac$ for each $R\in[\Rmin,\Rmax]$.

There exists a constant $C>0$ such that for any $\beta\ge1$ and any pair $(\ell,s)\in\mathbb N_{0}\times\mathbb N_0$, $\ell+s\leq\n-1$, we have the following bounds for all $R\in[\Rmin,\Rmax]$,
\begin{align}\label{E:PLEMELJ1}
&\sum_{m\in \Res(\l)} \frac1{|m|^\beta}\big|\PL[Q,\pa_y^\ell \pa_R^s\big(\widehat V_m S_m\big)](R)\big| \le  C|\l|^{1-\beta}\|Q\|_{C^1(\overline\Upsilon)} \|V\|_{C^{\n}(\oom)}.  \\
& \sum_{m\in \Res(\l)} \frac1{|m|^\beta}\big|\big(\PLplus-\PLminus\big)[Q,\pa_y^\ell \pa_R^s\big(\widehat V_m S_m\big)](R)\big| \le  C|\l|^{-\beta}\|Q\|_{C^1(\overline\Upsilon)} \|V\|_{C^{\n}(\oom)}. \label{E:PLEMELJ1PM}
\end{align}
Moreover, for any $\ga\in[0,1]\cap[0,\beta)$, there exists a constant $C=C_\ga>0$ such that
\begin{align}\label{E:PLEMELJ2}
&\sum_{m\in \Res(\l)^c} \frac1{|m|^\beta}\big|\PL[Q,\pa_y^\ell \pa_R^s\big(\widehat V_m S_m\big)](R)\big| \le C |\l|^{-\ga} \|Q\|_{C^0(\overline\Upsilon)}\|V\|_{C^{\n-1}(\oom)}. \\
& \sum_{m\in \Res(\l)^c} \frac1{|m|^\beta}\big(\PLplus-\PLminus\big)[Q,\pa_y^\ell \pa_R^s\big(\widehat V_m S_m\big)](R)\big| \le C \eps |\l|^{-2\ga} \|Q\|_{C^0(\Upsilon)}\|V\|_{C^{\n-1}(\oom)}. \label{E:PLEMELJ2PM}
\end{align}
In particular, for any $\ga\in[0,1]\cap[0,\beta)$ there exists a $C=C_\ga>0$ such that
\begin{align}\label{E:PLEMELJ3}
\sum_{m\in \Z_\ast} \frac1{|m|^\beta}\big|\Pl_{m,\l}[Q,\pa_y^\ell \pa_R^s\big(\widehat V_m S_m\big)](R)\big| 
\le C \max\{|\l|^{1-\beta},|\l|^{-\ga}\} \|Q\|_{C^1(\overline{\Upsilon})}\|V\|_{C^{\n}(\oom)}. \\
\sum_{m\in \Z_\ast} \frac1{|m|^\beta}\big|(\PLplus-\PLminus)[Q,\pa_y^\ell \pa_R^s\big(\widehat V_m S_m\big)](R)\big| 
\le C \max\{|\l|^{-\beta},|\l|^{-2\ga}\} \|Q\|_{C^1(\overline\Upsilon)}\|V\|_{C^{\n}(\oom)}. \label{E:PLEMELJ3PM}
\end{align}
\end{theorem}

\begin{proof}
{\em Proof of~\eqref{E:PLEMELJ1}.}
We rewrite
\begin{align}
\frac1{y+\frac{\l\pm i\eps}{2\pi m}} = \pa_y \plog, \ \ 
\plog(y):=\log (y+\frac{\l\pm i\eps}{2\pi m}). \label{E:PLOGDEF}
\end{align}
 Thus, using the vanishing properties of $Q$, integrating by parts we get
\begin{align*}
\PL[Q,\pa_y^\ell \pa_R^s\big(\widehat V_m S_m\big)](R) = - \int_{\J_R}\plog \big(\pa_y Q \pa_y^\ell \pa_R^s\big(\widehat V_m S_m\big) + Q \pa_y^{\ell+1} \pa_R^s\big(\widehat V_m S_m\big)\big) \diff (y,z),
\end{align*}
where the boundary integrals vanish by Remark~\ref{R:YIBP} and the vanishing of $Q$ on the vacuum boundary.
Therefore
\begin{align}
&\big|\PL[Q,\pa_y^\ell \pa_R^s\big(\widehat V_m S_m\big)]\big|  \notag\\
& \le \|Q\|_{C^1(\overline\Upsilon)} \int_{\J_R}|\plog| | \pa_y^\ell \pa_R^s\big(\widehat V_m S_m\big)| \diff(y,z)
+\|Q\|_{C^0(\overline\Upsilon)} \int_{\J_R}|\plog| | \pa_y^{\ell+1} \pa_R^s\big(\widehat V_m S_m\big)| \diff(y,z).
\end{align}
We now use~Proposition~\ref{P:KEY} and bound the second integral
\begin{align}
&\|Q\|_{C^0(\overline\Upsilon)}\int_{\J_R}|\plog| | \pa_y^{\ell+1} \pa_R^s\big(\widehat V_m S_m\big)| \diff(y,z) \notag\\
& \le C\|Q\|_{C^0(\overline\Upsilon)}\|V\|_{C^{\n}(\oom)}  \int_{\J_R}|\plog| \big((C|m|)^{\n}(C^2\dE)^{\max\{\frac{|m|-\n}{2},0\}} \mathbbm{1}_{\J_{R,\de}}+1\big)\diff(y,z)\notag\\
& \le  C\|Q\|_{C^0(\overline\Upsilon)}\|V\|_{C^{\n}(\oom)} \big((C|m|)^{\n}(C^2\de)^{\max\{\frac{|m|-\n}{2},0\}}+1\big),
\end{align}
where we have used the assumption $m\in \Res(\l)$ to ensure that the upper bound $\|\plog\|_{L^1(\J_R)}\lesssim1$ is independent of $m,\l$. 
We may bound the first integral analogously to obtain in total
\begin{align}
&\sum_{m\in \Res(\l)}\frac1{|m|^\beta}\big|\PL[Q,\pa_y^\ell \pa_R^s\big(\widehat V_m S_m\big)]\big|  \notag \\
& \le C\|Q\|_{C^1(\overline\Upsilon)}\|V\|_{C^{\n}(\oom)}\sum_{m\in \Res(\l)}\frac1{|m|^\beta} \Big((C|m|)^{\n}(C^2\de)^{\max\{\frac{|m|-\n}{2},0\}}+1\Big), \label{E:BETAPLM}
\end{align}
assuming $C^2\de<1$.
We now note that for any $A,A'\in\mathbb N$ and $0<\delta_0<1$ the series $\sum_{\ell\in\mathbb N}\ell^A\delta_0^{\frac{\ell-A'}{2}}$ converges by the ratio test.
Therefore, using~\eqref{E:RLAMBDASIZE}--\eqref{E:RLAMBDAMIN}
we infer that for any $N\in\mathbb N$ there exists a constant $C_{N,k}>0$
\begin{align}\label{E:NBOUND}
\sum_{m\in \Res(\l)}\frac1{|m|^\beta} (C|m|)^{\n}(C^2\de)^{\max\{\frac{|m|-\n}{2},0\}} \le  C_{N,\n} |\l|^{-N}.
\end{align}
Moreover,
\[
\sum_{m\in \Res(\l)}\frac1{|m|^\beta}  \lesssim |\l|^{1-\beta},
\]
where we have used~\eqref{E:RLAMBDASIZE}--\eqref{E:RLAMBDAMIN}. Using the above bounds in~\eqref{E:BETAPLM} we conclude
\begin{align}
\sum_{m\in \Res(\l)}\frac1{|m|^\beta}\big|\PL[Q,\pa_y^\ell \pa_R^s\big(\widehat V_m S_m\big)]\big| \le C |\l|^{1-\beta}\|Q\|_{C^1(\overline\Upsilon)}\|V\|_{C^{\n}(\oom)},
\end{align}
as claimed.
Here we choose $N= \beta-1$ in~\eqref{E:NBOUND}. 

{\em Proof of~\eqref{E:PLEMELJ1PM}.}
We have the classical identity
\begin{align}\label{E:KERNELH}
 H^\eps(y,m;\l):=\frac1{y+\frac{\la+i\eps}{2\pi m}}-\frac1{y+\frac{\la-i\eps}{2\pi m}} = \frac{-2i \frac{\eps}{2\pi m}}{\big(y+\frac{\l}{2\pi m}\big)^2 + \frac{\eps^2}{4\pi^2 m^2}},
\end{align}
which is a well-known approximation of unity kernel satisfying 
$
\int_\R |H^\eps(y,m;\la)| \diff y = 2\pi.
$
 Using this we have
\begin{align}
&\big|\big(\PLplus-\PLminus\big)[Q,\pa_y^\ell \pa_R^s\big(\widehat V_m S_m\big)]\big| = \big|\int_{\J_R} H^\eps(y,m;\l) Q \pa_y^\ell \pa_R^s\big(\widehat V_m S_m\big) \diff (y,z)\big| \notag\\
& \qquad \le C \|Q\|_{C^0(\overline\Upsilon)} \|V\|_{C^{\n}(\oom)} \int_{\J_R} |H^\eps(y,m;\l) | \big((C|m|)^{\n-1}(C^2\de)^{\max\{\frac{|m|-\n+1}{2},0\}}+|m|^{-1}\big) \diff(y,z) \notag\\
& \qquad \le  C \|Q\|_{C^0(\overline\Upsilon)} \|V\|_{C^{\n}(\oom)} \big((C|m|)^{\n-1}(C^2\de)^{\max\{\frac{|m|-\n+1}{2},0\}}+|m|^{-1}\big) \label{E:PLPM}
\end{align}
where we have used Proposition~\ref{P:KEY} in the next-to-last line. We observe the $\frac1{|m|}$ gain in the second term above, which in the resonant region will lead to an additional inverse power of $\la$
by comparison to~\eqref{E:PLEMELJ1}. Namely, using~\eqref{E:NBOUND} and~\eqref{E:RLAMBDASIZE}--\eqref{E:RLAMBDAMIN} we obtain
\begin{align}
 \sum_{m\in \Res(\l)} \frac1{|m|^\beta}\big|\big(\PLplus-\PLminus\big)[Q,\pa_y^\ell \pa_R^s\big(\widehat V_m S_m\big)](R)\big|
& \le C \|Q\|_{C^0(\overline\Upsilon)} \|V\|_{C^{\n}(\oom)} \sum_{m\in \Res(\l)} \frac1{|m|^{\beta+1}}   \notag\\
& \le C  \|Q\|_{C^0(\overline\Upsilon)} \|V\|_{C^{\n}(\oom)} |\l|^{-\beta},
\end{align}
which completes the proof of~\eqref{E:PLEMELJ1PM}.

{\em Proof of~\eqref{E:PLEMELJ2}--\eqref{E:PLEMELJ2PM}.}
In the non-resonant case the estimates simplify a little. By~\eqref{E:NONRESBOUND} and Proposition~\ref{P:KEY}, we have
\begin{align}
&\sum_{m\in \Res(\l)^c} \frac1{|m|^\beta}\big|\Pl_{m,\l}[Q,\pa_y^\ell \pa_R^s\big(\widehat V_m S_m\big)](R)\big| \notag\\
& \le C |\l|^{-\ga} \|Q\|_{C^0(\overline\Upsilon)} \sum_{m\in \Res(\l)^c} \frac1{|m|^{\beta-\ga}}\int_{\J_R}  \big| \pa_y^\ell \pa_R^s\big(\widehat V_m S_m\big)\big| \diff(y,z) \notag\\
& \le C |\l|^{-\ga} \|Q\|_{C^0(\overline\Upsilon)} \|V\|_{C^{\n-1}(\oom)} \sum_{m\in \Res(\l)^c} \frac1{|m|^{\beta-\ga}} \big((C|m|)^{\n-1}(C^2\de)^{\max\{\frac{|m|-\n+1}{2},0\}}+|m|^{-1}\big) \notag\\
& \le C |\l|^{-\ga} \|Q\|_{C^0(\overline\Upsilon)} \|V\|_{C^{\n-1}(\oom)}, 
\end{align}
where we used the assumptions $C^2\de<1$ and $\beta>\ga$  to infer that $\sum_{m\in\Z_\ast} |m|^{-1-\beta +\ga}<\infty$.
This shows~\eqref{E:PLEMELJ2}.
Using~\eqref{ineq:nonresfrediff} we similarly have
\begin{align}
&\sum_{m\in \Res(\l)^c} \frac1{|m|^\beta}\big|(\PLplus-\PLminus)[Q,\pa_y^\ell \pa_R^s\big(\widehat V_m S_m\big)](R)\big| \notag\\
& \le C\eps |\l|^{-2\ga} \|Q\|_{C^0(\overline\Upsilon)}  \sum_{m\in \Res(\l)^c} \frac1{|m|^{\beta+1-2\ga}}\int_{\J}  \big| \pa_y^\ell \pa_R^s\big(\widehat V_m S_m\big)\big| \diff(y,z) \notag\\
& \le C \eps|\l|^{-2\ga} \|Q\|_{C^0(\overline\Upsilon)} \|V\|_{C^{\n-1}(\oom)} \sum_{m\in \Res(\l)^c} \frac1{|m|^{\beta+1-2\ga}} \big((C|m|)^{\n-1}(C^2\de)^{\max\{\frac{|m|-\r+1}{2},0\}}+|m|^{-1}\big) \notag\\
& \le C \eps|\l|^{-2\ga} \|Q\|_{C^0(\overline\Upsilon)}\|V\|_{C^{\n-1}(\oom)}, 
\end{align} 
where we have used $\ga<\beta$ and $\ga\le1$ to ensure summability in the last bound. This shows~\eqref{E:PLEMELJ2PM}.

Claim~\eqref{E:PLEMELJ3} is a simple consequence of~\eqref{E:PLEMELJ1} and~\eqref{E:PLEMELJ2} while~\eqref{E:PLEMELJ3PM} follows from~\eqref{E:PLEMELJ1PM} and~\eqref{E:PLEMELJ2PM}.
\end{proof}

\begin{remark}
We note that the estimate for near-resonant frequencies $m\in \Res(\l)$ requires more regularity on $Q$ and $V$ than in the nonresonant case.
\end{remark}

The following proposition shows that 
the operator $\FPM$ defined in~\eqref{E:FDEF}
is indeed invertible, as a simple consequence of Theorem~\ref{T:PLEMELJ}.


\begin{prop}\label{P:INVERT}
Let $\l\in\sL$ be given. There exists an $\eta_0>0$ such that for all $0<\eta<\eta_0$
the operator $\FPM : C^1_R\to C^1_R$ is bounded (independent of $\l$) and invertible (with inverse bounded independent of $\l$).
\end{prop}

\begin{proof}
Let $V\in C^1_R$, where we recall the notational convention from Section~\ref{SS:SPACES}.
By Theorem~\ref{T:PLEMELJ} we have 
\begin{align}
& |\pa_R\FPM[V] -\pa_R V| \notag\\
 & \le \sum_{m\in \Res(\l)}\frac\eta{|m|}\frac{4\pi}{R^2} \big|\PL[q, \widehat V_m S_m]\big| + \sum_{m\in \Res(\l)^c}\frac\eta{|m|}\frac{4\pi}{R^2} \big|\PL[q, \widehat V_m S_m]\big|  \le C \eta \|V\|_{C^1_R}\notag
\end{align} 
where we recall that $q$ is given by~\eqref{E:LITTLEQDEF} and therefore
$\|q\|_{C^1(\overline\Upsilon)}\lesssim 1$ due to assumptions $\mu\ge2$, $\nu\ge1$ in~\eqref{E:SS2}.
Since $\FPM[V](\Rmax)= V(\Rmax)$ it follows that $\|\FPM[V]-V\|_{C^0_R} \leq C \|\pa_R\FPM[V] -\pa_R V\|_{C^0_R}
\le C \eta \|V\|_{C^1_R}$. Therefore, for sufficiently small $\eta>0$ we have $\|\FPM-{\bf Id}\|_{\mathcal L(C^1_R,C^1_R)}<1$ and the operator is invertible.
\end{proof}



\subsection{Resolvent potential estimates}


\begin{theorem}[Mixed derivative uniform resolvent bounds]\label{T:RESOLVENT1}
There exist $\eps$-independent constants $C=C(\reg)>0$ and $\eta_0>0$ such that for all $0<\eta<\eta_0$, 
\begin{align}\label{E:RESEST1}
\||\l|\, \Ueps\|_{\C^{\reg}_R}\le C \|f_0\|_{C^{\reg}(\oom)},
\end{align}
where we recall the $\C^\reg_R$-spaces~\eqref{E:CKRDEF}.
\end{theorem}

\begin{proof}
To estimate $\pa_\la^b\Ueps$ in $C^{\reg-j}_R$, we proceed by double induction in $b$ and the number of spatial derivatives in order to show
\begin{align}\label{E:ABIND0}
\|\pa_R^a\pa_\la^b\Ueps(\cdot;\l)\|_{C^0_R} \le C |\l|^{-1} \|f_0\|_{C^{\max\{1,a+b\}}(\oom)}, \ \ a+b\le \reg,\ \ b\leq \reg-1.
\end{align}
We recall $q$ from~\eqref{E:LITTLEQDEF} and the assumption~\eqref{E:SS2}. From~\eqref{E:LITTLEQDEF}--\eqref{E:QREG}, we have that $q,qy^{-1}\in C^K(\overline\Upsilon)$, and so bound derivatives of these functions by constants throughout.

{\em Step 1: $b=0$.}
We first show~\eqref{E:ABIND0} for $b=0$ by  induction on $a$. If $a=0,1$ the claim follows from~\eqref{E:FUEPS}, Proposition~\ref{P:INVERT},
and~\eqref{E:PLEMELJ3} with $V=f_0$. Note that we have also used the bound 
$
\|qy^{-1}\|_{C^1(\overline{\Upsilon})} \le C
$. 

Now for any $2\le a\le \reg$ from Proposition~\ref{P:HORES} we have
\be\label{E:UEPSHO}
\pa_R^a \Ueps= \mathcal S^{\pm,\l,\eps}_{a,0}[f_0] + \mathcal R^{\pm,\l,\eps}_{a,0}[\Ueps],
\ee
where the right-hand side is given by~\eqref{E:SDEF}--\eqref{E:RDEF}. 
From~\eqref{E:SDEF} we have 
\begin{align}
\big|\mathcal S^{\pm,\l,\eps}_{a,0}[f_0] \big|
& \le C \sum_{m\in\Z_\ast}\frac1{m^2} \sum_{s=0}^{a-1} \PL[\pa_R^{a-1-s}(qy^{-1}),\pa_R^s\big( \fhat_{0,m} S_m\big)] \notag\\
& \le C \sum_{s=0}^{a-1} |\l|^{-1} \|\pa_R^{a-1-s}q\|_{C^1(\overline{\Upsilon})} \|f_0\|_{C^{s+1}(\oom)}  \le C|\la|^{-1}  \|f_0\|_{C^a(\oom)},\label{E:ST1}
\end{align}
where we have used~\eqref{E:PLEMELJ3} with $\beta=2$, $\ga=1$,
and~\eqref{E:QREG}.

Similarly, we note that by~\eqref{E:RDEF} and~\eqref{E:PLEMELJ3} (with $\beta=1$, $\ga=0$) we have 
\begin{align}
|\mathcal R_{a,0}^{\pm,\l,\eps}[\Ueps]| & \le C \|\Ueps\|_{C^{a-1}_R} + C\eta\sum_{s\le a-1} \sum_{m\in \Z_\ast}\frac1{|m|} \big|\PL[\pa_R^{a-1-s}q,\pa_R^s\big(\wUepsm S_m\big)]\big| \notag\\
& \le C \|\Ueps\|_{C^{a-1}_R} + C \eta \sum_{s\le a-1} \|\pa_R^{a-1-s}q\|_{C^1(\overline{\Upsilon})} \|\Ueps\|_{C^{s+1}_R} \notag\\
& \le C |\la|^{-1}\|f_0\|_{C^{a-1}(\oom)} + C\eta  \|\Ueps\|_{C^a_R}, \label{E:ST3}
\end{align}
where we have also used~\eqref{E:QREG} and the inductive hypothesis applied to $ \|\Ueps\|_{C^{a-1}_R}$.

We now plug in~\eqref{E:ST1} and~\eqref{E:ST3} into~\eqref{E:UEPSHO}. Using the smallness of $0<\eta\ll1$ we conclude~\eqref{E:ABIND0} in the case $b=0$.

{\em Step 2: $1\le b\le\reg-1$.} Assume that~\eqref{E:ABIND0} is true for all $0\le b'\le b-1$. If $a=0,1$, we use Corollary~\ref{C:PURELA} and Proposition~\ref{P:INVERT} to conclude that
\begin{align}\label{E:LAJ1}
\|\pa_\l^b\Ueps\|_{C^1_R} \le C \|\mathcal S^{\pm,\l,\eps}_{1,b}[f_0]\|_{C^0_R} + C\|\tilde{\mathcal R}^{\pm,\l,\eps}_{1,b}[\Ueps]\|_{C^0_R},  
\end{align}
where $\mathcal S^{\pm,\l,\eps}_{1,b}$ and $\tilde{\mathcal R}^{\pm,\l,\eps}_{1,b}$ are given by~\eqref{E:SDEF} and~\eqref{E:TILDERDEF} respectively. From~\eqref{E:PLEMELJ3} (with $\beta\ge3$) we then obtain
\begin{align}
 \|\mathcal S^{\pm,\l,\eps}_{1,b}[f_0]\|_{C^0_R} & \le
 C \sum_{m\in\Z_\ast}\frac1{|m|^{2+b}}  \sum_{\ell=0}^b  \PL[\pa_y^{b-\ell}q,\pa_y^\ell\big( \fhatm S_m\big)]  \notag\\
 &\le C  \sum_{\ell=0}^b    |\l|^{-1} \|q\|_{C^{b-\ell+1}(\overline{\Upsilon})}\|f_0\|_{C^{\ell+1}(\oom)}  \le C |\l|^{-1} \|f_0\|_{C^{b+1}(\oom)}. \label{E:LAJ2}
\end{align}
Similarly,
\begin{align}
 \|\tilde{\mathcal R}^{\pm,\l,\eps}_{1,b}[\Ueps]\|_{C^0_R} & \le
C \eta \sum_{m\in \Z_*}\frac1{m} \sum_{\ell+j+n=b \atop j<b}  \frac{1}{|m|^{b-j}} \big| \PL[\pa_y^{n}q,\pa_y^\ell\big(\pa_\l^j\wUepsm S_m\big)]\big| \notag\\
& \le C\eta \sum_{\ell+j+n=b \atop j<b}  \sum_{m\in \Z_*}\frac1{m^2}  \big| \PL[\pa_y^{n}q,\pa_y^\ell\big(\pa_\l^j\wUepsm S_m\big)]\big| \notag\\
& \le C \eta |\l|^{-1} \sum_{\ell+j\le b \atop j<b}\|\pa_\l^j\Ueps\|_{C^{\ell+1}_R},\label{E:LAJ3}
\end{align}
where we have used~\eqref{E:PLEMELJ3} (with $\beta=2$, $\ga=1$) and the bound $b-j\ge1$ in the second line. Since $j\le b-1$ and $\ell+j+1\le b+1\le \reg$, we use the inductive assumption to conclude
$ \|\tilde{\mathcal R}^{\pm,\l,\eps}_{1,b}\|_{C^0_R}  \le C \eta |\l|^{-1}\|f_0\|_{C^{b+1}(\oom)}$.
As a consequence of this bound and~\eqref{E:LAJ1}--\eqref{E:LAJ2} we conclude
\begin{align}
\|\pa_\l^b\Ueps\|_{C^1_R}\le C |\l|^{-1} \|f_0\|_{C^{b+1}(\oom)}. 
\end{align}
To complete the proof of~\eqref{E:ABIND0}, we proceed again by induction in $a$.
The base cases $a=0,1$ have just been shown. Assuming that~\eqref{E:ABIND0} is true for all $a\le A$ for some $1\le A\le \reg-b-1$, we then use~\eqref{E:PLEMELJ3} and~\eqref{E:ABKEY} to 
show~\eqref{E:ABIND0} holds also for $a=A+1$ using an argument analogous to the proof of Step 1.
\end{proof}


\subsection{Regularity bootstrap}


The goal of this section is to prove the following improved resolvent bounds, which 
capitalise on the cancellations encoded inside the approximation of unity kernel $\frac1{y+\frac{\l+i\eps}{2\pi m}}-\frac1{y+\frac{\l-i\eps}{2\pi m}}$.


\begin{theorem}[Uniform resolvent bounds 2]\label{T:RESOLVENT2}
There exist $\eps$-independent constants $C=C(\reg)>0$ and $\eta_0>0$ such that for all $0<\eta<\eta_0$, 
\begin{align}\label{E:RESEST1PM}
\||\l|^2 (\Uepsplus-\Uepsminus)(\cdot;\l)\|_{\C^{\reg}_R}\le  C \|f_0\|_{C^{\reg}(\oom)},
\end{align}
where we recall the $\C^\reg_R$-spaces~\eqref{E:CKRDEF}.
\end{theorem}


\begin{proof}
From Lemma~\ref{L:PARTIALRUEPSOP} we have the basic identity
\begin{align}
\pa_R(\Uepsplus-\Uepsminus)(R,\l)
&= \frac2{R^2} \sum_{m\in\Z_\ast} \frac1{m^2}\big(\PLplus-\PLminus\big)[qy^{-1}, \fhatm S_m] \notag\\
& \ \ \ -\frac{4\pi\eta}{R^2}\sum_{m\in \Z_*}\frac1{m} \big(\PLplus-\PLminus\big)[q, \wUepsplusm S_m] \notag \\
& \ \ \  -\frac{4\pi\eta}{R^2}\sum_{m\in \Z_*}\frac1{m} \PLminus[q, (\wUepsplusm-\wUepsminusm) S_m], \label{E:PMFUN1}
\end{align}
where we recall~\eqref{E:LITTLEQDEF}.
We will prove the theorem by the same inductive strategy as in the proof of Theorem~\ref{T:RESOLVENT1} to show
\begin{align}\label{E:ABINDPM}
\|\pa_R^a\pa_\la^b(\Uepsplus-\Uepsminus)\|_{C^0_R} \le C  |\l|^{-2} \|f_0\|_{C^{\max\{1,a+b\}}(\oom)}, \ \ a+b\le \reg,\ \ b\leq \reg-1.
\end{align}

{\em Step 1: Let $b=0$.} 
Moving the right-most term in~\eqref{E:PMFUN1} to the left-hand side, we may rewrite the above identity in the form (recalling Lemma~\ref{L:UOUTERRIM})
\begin{align}
\Fminus[\Uepsplus-\Uepsminus](R;\l)&=- \sum_{m\in\Z_\ast} \frac1{m^2} \int_R^{\Rmax} \frac2{r^2} \big(\PLplus-\PLminus\big)[qy^{-1}, \fhatm S_m] \diff r \notag \\
& \ \ \ + \sum_{m\in \Z_*}\frac\eta{m} \int_R^{\Rmax} \frac{4\pi}{r^2} \big(\PLplus-\PLminus\big)[q ,\wUepsplusm S_m]\diff r,\label{E:BPM}
\end{align}
where we recall the operators $\FPM$ from~\eqref{E:FDEF}. We recognise that the right-hand side above allows for an enhanced decay at large $\la$-frequencies due to~\eqref{E:PLEMELJ3PM}, which we explain 
next.
By~\eqref{E:PLEMELJ3PM} with $\beta=2$ and $\ga=1$ we have
\begin{align}
\sum_{m\in\Z_\ast} \frac1{m^2}  \Big|\big(\PLplus-\PLminus\big)[qy^{-1}, \fhat_{0,m}S_m] \Big| \le C |\l|^{-2} \|f_0\|_{C^{1}(\oom)}.
\end{align}
We apply~\eqref{E:PLEMELJ3PM} to the second term on the right-hand side of~\eqref{E:BPM} with $\beta=1$ and $\ga=\frac12$  to obtain
\begin{align}
\sum_{m\in \Z_*}\frac\eta{|m|} \Big|\big(\PLplus-\PLminus\big)[q, \wUepsplusm S_m]\Big|
\le C\eta |\l|^{-1} \|\Uepsplus\|_{C^1_R} \le C\eta  |\l|^{-2} \|f_0\|_{C^{1}(\oom)},
\end{align}
where we have crucially used Theorem~\ref{T:RESOLVENT1} (more precisely,~\eqref{E:ABIND0} with $b=0$). From Proposition~\ref{P:INVERT} we now conclude that 
\begin{align}\label{E:ABOUND}
\|\Uepsplus-\Uepsminus\|_{C^1_R} \le C |\l|^{-2} \|f_0\|_{C^{1}(\oom)}.
\end{align}
This establishes the base case of~\eqref{E:ABINDPM} for $b=0$, $a=0,1$. Now we assume for induction that, for some $2\le a\le \reg$,~\eqref{E:ABINDPM} holds for $b=0$, $0\leq a'\leq a-1$. From 
 Proposition~\ref{P:HORES} we have
\be\label{E:UEPSHOPM}
\pa_R^a (\Uepsplus-\Uepsminus)= (\mathcal S_{a,0}^{+,\l,\eps}-\mathcal S_{a,0}^{-,\l,\eps})[f_0] + (\mathcal R_{a,0}^{+,\l,\eps}-\mathcal R_{a,0}^{-,\l,\eps})[\Uepsplus]
+ \mathcal R_{a,0}^{-,\l,\eps}[\Uepsplus-\Uepsminus],
\ee
where the operators $\mathcal S_{a,b}^{\pm,\l,\eps}$, $\mathcal R_{a,0}^{\pm,\l,\eps}$ are defined by~\eqref{E:SDEF}--\eqref{E:RDEF}. 
From~\eqref{E:SDEF} we have 
\begin{align}
\big|(\mathcal S_{a,0}^{+,\l,\eps}-\mathcal S_{a,0}^{-,\l,\eps})[f_0] \big|
& \le C \sum_{m\in\Z_\ast}\frac1{m^2} \sum_{s=0}^{a-1} \big|(\PLplus-\PLminus)[\pa_R^{a-1-s}(qy^{-1}),\pa_R^s\big( \fhat_{0,m} S_m\big)]\big| \notag\\
& \le C \sum_{s=0}^{a-1} |\l|^{-2} \|\pa_R^{a-1-s}q\|_{C^1(\bar{\Upsilon})} \|f_0\|_{C^{s+1}(\oom)} \notag\\
& \le C|\la|^{-2}  \|f_0\|_{C^a(\oom)},\label{E:ST1PM}
\end{align}
where we have used~\eqref{E:PLEMELJ3PM} with $\beta=2$, $\gamma=1$, and~\eqref{E:QREG}.
Similarly, we note that by~\eqref{E:RDEF} and~\eqref{E:PLEMELJ3PM} (with $\beta=1$, $\gamma=\frac12$) we have 
\begin{align}
| (\mathcal R_{a,0}^{+,\l,\eps}-\mathcal R_{a,0}^{-,\l,\eps})[\Uepsplus]| & \le  C\eta \sum_{s\le a-1} \sum_{m\in \Z_\ast}\frac1{|m|} \big|(\PLplus-\PLminus)[\pa_R^{a-1-s}q,\pa_R^s\big(\wUepsplusm S_m\big)]\big| \notag\\
& \le  C \eta \sum_{s\le a-1} |\l|^{-1} \|\pa_R^{a-1-s}q\|_{C^1(\bar{\Upsilon})} \|\Uepsplus\|_{C^{s+1}_R} \notag\\
& \le C\eta |\la|^{-2}\|f_0\|_{C^a(\oom)}, \label{E:ST3PM}
\end{align}
where we have crucially used~\eqref{E:ABIND0} to estimate $ \|\Uepsplus\|_{C^a_R}$ by $C|\l|^{-1}\|f_0\|_{C^a(\oom)}$. Finally, using the same 
strategy as in the proof of~\eqref{E:ST3} we obtain
\begin{align}
\big| \mathcal R_{a,0}^{-,\l,\eps}[\Uepsplus-\Uepsminus]\big| &\le C \|\Uepsplus-\Uepsminus\|_{C^{a-1}_R} + C\eta  \|\Uepsplus-\Uepsminus\|_{C^{a}_R} \notag\\
&\le C |\l|^{-2}\|f_0\|_{C^{a-1}(\oom)} + C\eta  \|\Uepsplus-\Uepsminus\|_{C^{a}_R},\label{E:ST4PM}
\end{align}
where we have used the inductive hypothesis in the second line. We use the above estimates in~\eqref{E:UEPSHOPM} and the smallness of $\eta\ll1$ to absorb the right-most term in~\eqref{E:ST4PM}
into the left-hand side. This concludes the proof of~\eqref{E:ABINDPM} in the case $b=0$.

{\em Step 2: $1\le b\le\reg-1$.}
 Assume that~\eqref{E:ABINDPM} is true for all $0\le b'\le b-1$. If $a=0,1$, we use the identity~\eqref{E:PURELAPM} in Corollary~\ref{C:PURELA} to conclude that
\begin{align}
\|\pa_\l^b(\Uepsplus-\Uepsminus)\|_{C^1_R} & \le C \|(\mathcal S^{+,\l,\eps}_{1,b}-\mathcal S^{-,\l,\eps}_{1,b})[f_0]\|_{C^0_R} + C\|(\mathcal R^{+,\l,\eps}_{1,b}-\mathcal R^{-,\l,\eps}_{1,b})[\Uepsplus]\|_{C^0_R} \notag\\  
& \ \ \ \ +C\|\tilde{\mathcal R}^{-,\l,\eps}_{1,b}[\Uepsplus-\Uepsminus]\|_{C^0_R},  \label{E:LAJ1PM}
\end{align}
where $\mathcal S^{\pm,\l,\eps}_{1,b}$, $\mathcal R^{\pm,\l,\eps}_{1,b}$, and $\tilde{\mathcal R}^{\pm,\l,\eps}_{1,b}$ are given by~\eqref{E:SDEF},~\eqref{E:RDEF},  and~\eqref{E:TILDERDEF} respectively. 
We use~\eqref{E:PLEMELJ3PM} (with $\beta\ge3$) and by analogy to~\eqref{E:LAJ2} we obtain
\begin{align}
\|(\mathcal S^{+,\l,\eps}_{1,b}-\mathcal S^{-,\l,\eps}_{1,b})[f_0]\|_{C^0_R} 
 & \le C |\l|^{-2} \|f_0\|_{C^{b+1}(\oom)}. \label{E:LAJ2PM}
\end{align}
Similarly, using~\eqref{E:PLEMELJ3PM} (with $\beta=1$, $\ga=\frac12$) we have
\begin{align}
\|(\mathcal R^{+,\l,\eps}_{1,b}-\mathcal R^{-,\l,\eps}_{1,b})[\Uepsplus]\|_{C^0_R}  & \le
C\eta \sum_{m\in \Z_*}\frac1{m} \sum_{\ell+j+n=b}  \frac{1}{|m|^{b-j}} \big| (\PLplus-\PLminus)[\pa_y^{n}q,\pa_y^\ell\big(\pa_\l^j\wUepsplusm S_m\big)]\big| \notag\\
& \le C \eta |\l|^{-1} \sum_{\ell+j\le b}\|\pa_\l^jU_\eps^+\|_{C^{\ell+1}_R}  \le C \eta |\l|^{-2} \|f_0\|_{C^{b+1}(\oom)}, \label{E:LAJ3PM}
\end{align}
where we have used 
Theorem~\ref{T:RESOLVENT1} (specifically~\eqref{E:ABIND0}) in the last inequality. 
Finally, analogously to~\eqref{E:LAJ3} we have
\begin{align}
\|\tilde{\mathcal R}^{-,\l,\eps}_{1,b}[\Uepsplus-\Uepsminus]\|_{C^0_R}
& \le C \eta \sum_{m\in \Z_*}\frac1{m} \sum_{\ell+j+n=b \atop j<b}  \frac{1}{|m|^{b-j}} \big| \PLminus[\pa_y^{n}q,\pa_y^\ell\big(\pa_\l^j(\wUepsplusm-\wUepsminusm) S_m\big)]\big| \notag\\
& \le C \eta |\l|^{-1} \sum_{\ell+j\le b \atop j<b}\|\pa_\l^j(\Uepsplus-\Uepsminus)\|_{C^{\ell+1}_R}  \le C\eta |\l|^{-3} \|f_0\|_{C^{b+1}(\oom)},\label{E:LAJ4}
\end{align}
where the last bound follows from the inductive assumption and the bounds $j\le b-1$ and $\ell+j+1\le b+1\le \reg$.
As a consequence of this bound and~\eqref{E:LAJ1PM}--\eqref{E:LAJ3PM} we conclude
\begin{align}
\|\pa_\l^b(\Uepsplus-\Uepsminus)\|_{C^1_R}\le C |\l|^{-2} \|f_0\|_{C^{b+1}(\oom)}. 
\end{align}
To complete the proof of~\eqref{E:ABINDPM}, we again argue inductively on $a$, as the base cases $a=0,1$ have just been shown. We employ
the identity
\be\label{E:UEPSHOAB}
\pa_R^a\pa_\l^b (\Uepsplus-\Uepsminus)= (\mathcal S_{a,b}^{+,\l,\eps}-\mathcal S_{a,b}^{-,\l,\eps})[f_0] + (\mathcal R_{a,b}^{+,\l,\eps}-\mathcal R_{a,b}^{-,\l,\eps})[\Uepsplus]
+ \mathcal R_{a,b}^{-,\l,\eps}[\Uepsplus-\Uepsminus],
\ee
to 
show~\eqref{E:ABINDPM} holds for $a=A+1$ using an argument analogous to the proof of Step 1. 
\end{proof}

\subsection{The limit $\eps\to0$}


In order to derive a priori bounds on the space-frequency derivatives of $\pa_\eps\Ueps$, we apply $\pa_\eps$ to~\eqref{E:ABKEY}.
The key observation is that $\pa_\eps(y+\frac{\l\pm i\eps}{2\pi m}) = \pm i \pa_\l(y+\frac{\l\pm i\eps}{2\pi m})$ and hence, for $\eps>0$,
\begin{align}\label{E:ABKEYEPS}
\pa_R^a\pa_\l^b\pa_\eps\Ueps = \pm i  \mathcal S_{a,b+1}^{\pm,\l,\eps}[f_0] \pm i  \tilde{\mathcal R}^{\pm,\l,\eps}_{a,b+1}[\Ueps] +\mathcal R^{\pm,\l,\eps}_{a,b}[\pa_\eps\Ueps],
\end{align}
where we recall~\eqref{E:SDEF},~\eqref{E:RDEF} and~\eqref{E:TILDERDEF}.

We start with the following crude bound on the Plemelj operators $\PL$.


\begin{lemma}\label{L:EPSILONCRUDE}
Let $\ell,s,j\in\N_0$ be such that $\ell+s+j\leq \reg$, $j\le\reg-1$, and let $\beta\ge1$ be given. There exists a constant $C>0$ such that for any $V\in \C^{\reg}(\oom\times\sL)$ (recall~\eqref{E:CKDEF}),
\begin{align}
\sum_{m\in \Res(\l)}\frac1{|m|^\beta} |\PL[Q,\pa_y^\ell\pa_R^s(\pa_\l^j\widehat{V}_mS_m)]| & \le  C \|Q\|_{C^0(\bar{\Upsilon})}\|\pa_\l^jV\|_{C^{\reg-j}(\oom)} |\l|^{\frac32-\beta}\eps^{-\frac12}. \label{E:CRUDERES}
\end{align}
Moreover, for any $\ga\in[0,1]\cap[0,\beta-1)$ there exists a constant $C=C(\ga)>0$ such that
\begin{align}
\sum_{m\in \Res(\l)^c}\frac1{|m|^\beta} |\PL[Q,\pa_y^\ell\pa_R^s(\pa_\l^j\widehat{V}_mS_m)]| & \le  C |\l|^{-\ga} \|Q\|_{C^0(\bar{\Upsilon})}\|\pa_\l^jV\|_{C^{\n-j}(\oom)}.\label{E:CRUDENONRES}
\end{align}
In particular, the following crude bound holds
\begin{align}
\sum_{m\in \Z_\ast}\frac1{|m|^\beta} |\PL[Q,\pa_y^\ell\pa_R^s(\pa_\l^j\widehat{V}_mS_m)]| & \le  C  \max\{|\l|^{\frac32-\beta},|\l|^{-\ga}\} \|Q\|_{C^0(\bar{\Upsilon})}\|\pa_\l^jV\|_{C^{\n-j}(\oom)} \eps^{-\frac12}. \label{E:CRUDE}
\end{align}
\end{lemma}


\begin{proof}
Let $m\in \Res(\l)$ be a near resonant frequency. Then for any $Q\in C^0(\overline{\Upsilon})$, from~\eqref{E:PLEMDEF}, we have
\begin{align}
|\PL[Q,\pa_y^\ell\pa_R^s(\pa_\l^j\widehat{V}_mS_m)]| \le \int_{\J_R} \frac{|Q||\pa_y^\ell\pa_R^s(\pa_\l^j\widehat{V}_mS_m)|}{(|y+\frac{\la}{2\pi m}|^2 + \frac{\eps^2}{4\pi^2 m^2})^{\frac12}} \diff (y,z).
\end{align}
We now use the simple lower bound
\[
\big|y+\frac{\la}{2\pi m}\big|^2 + \frac{\eps^2}{4\pi^2 m^2} \ge 2 \big|y+\frac{\la}{2\pi m}\big| \frac{\eps}{2\pi |m|} 
\]
to obtain 
\begin{align}\label{E:EPSBOUND1}
|\PL[Q,\pa_y^\ell\pa_R^s(\pa_\l^j\widehat{V}_mS_m)]| \le C\|Q\|_{C^0(\overline{\Upsilon})} \frac{|m|^{\frac12}}{\eps^{\frac12}}  \int_{\J_R}|\pa_y^\ell\pa_R^s(\pa_\l^j\widehat{V}_mS_m)| \big|y+\frac{\la}{2\pi m}\big|^{-\frac12} \diff (y,z).
\end{align}
We now invoke Proposition~\ref{P:KEY} to bound
\begin{align}
|\pa_y^\ell\pa_R^s(\pa_\l^j\widehat{V}_mS_m)|
&\le C \|\pa_\l^jV\|_{C^{\reg-j}(\oom)} \big((C|m|)^{\reg-j}(C^2\de)^{\max\{\frac{|m|-\reg+j}{2},0\}}+1\big) \le C \|\pa_\l^jV\|_{C^{\reg-j}(\oom)},\notag
\end{align}
where we have used the pointwise version of~\eqref{E:NBOUND} with $N\ge1$. Since $|y+\frac{\la}{2\pi m}|^{-\frac12}$
is integrable in the singular zone $|y+\frac{\l}{2\pi m}|\ll1$ and $|\frac{\l}{m}|$ is uniformly bounded from below and above for $m\in \Res(\l)$ we conclude from~\eqref{E:EPSBOUND1} that
\begin{align}
|\PL[Q,\pa_y^\ell\pa_R^s(\pa_\l^j\widehat{V}_mS_m)]| \le C \|Q\|_{C^0(\overline{\Upsilon})} |m|^{\frac12}\eps^{-\frac12}  \|\pa_\l^jV\|_{C^{\reg-j}(\oom)}.
\end{align}
Using~\eqref{E:RLAMBDASIZE}--\eqref{E:RLAMBDAMIN} we obtain for any $\beta\ge0$
\begin{align}
\sum_{m\in \Res(\l)}\frac1{|m|^\beta}|\PL[Q,\pa_y^\ell\pa_R^s(\pa_\l^j\widehat{V}_mS_m)]| \le C \|Q\|_{C^0(\overline{\Upsilon})}\|\pa_\l^jV\|_{C^{\reg-j}(\oom)} |\l|^{\frac32-\beta}\eps^{-\frac12},
\end{align}
which is~\eqref{E:CRUDERES}. 

Let now $m\in \Res(\l)^c$. Then from~\eqref{E:NONRESBOUND} we obtain for any $\ga\in[0,1]\cap[0,\beta-1),$
\begin{align}
& \sum_{m\in \Res(\l)^c}\frac1{|m|^\beta}|\PL[Q,\pa_y^\ell\pa_R^s(\pa_\l^j\widehat{V}_mS_m)]| \notag\\
& \le C\|Q\|_{C^0(\overline{\Upsilon})} \sum_{m\in \Res(\l)^c} \frac1{|m|^\beta}\big(\frac{|m|}{|\l|}\big)^\ga\int_{\J_R} |\pa_y^\ell\pa_R^s(\pa_\l^j\widehat{V}_mS_m)| \diff(y,z) \notag\\
 & \le C|\l|^{-\ga}\|Q\|_{C^0(\overline{\Upsilon})} \|\pa_\l^jV\|_{C^{\reg-j}(\oom)} \sum_{m\in \Res(\l)^c} |m|^{\ga-\beta} \notag\\
 & \le C |\l|^{-\ga} \|Q\|_{C^0(\overline{\Upsilon})}\|\pa_\l^jV\|_{C^{\reg-j}(\oom)},
\end{align}
where we have used the assumption $\ga-\beta<-1$.  This concludes the proof of~\eqref{E:CRUDENONRES}.
Bound~\eqref{E:CRUDE} follows from~\eqref{E:CRUDERES}--\eqref{E:CRUDENONRES}.
\end{proof}


We are now in position to prove the main claim of this subsection.


\begin{prop}\label{P:EPSUNIFORM}
For any $(\l,\eps)\in\sL\times(0,1)$ and any $\ga\in[0,1)$ there exists a constant $C=C(\reg,\ga)>0$ such that 
\begin{align}\label{E:EPSUNIFORM}
\||\l|^{\frac32}\pa_\eps\Ueps\|_{\C^{\reg}_R} \le  C\eps^{-\frac12} \|f_0\|_{C^{\reg}(\oom)}, \ \ 0\leq b\le\reg-1. 
\end{align}
Moreover, there exist functions $U_0^\pm(\cdot,\l)\in \C^{\reg}_R$ 
such that
\begin{align}\label{E:EPSUNIFORM2.1}
 \||\l|^{\frac32}\big(\Ueps(\cdot;\l)-U_0^\pm(\cdot;\l)\big)\|_{\C^{\reg}_R} \le C \eps^{\frac12} \|f_0\|_{C^{\reg}(\oom)} .
\end{align}
\end{prop}


\begin{proof}
We observe that for any pair of non-negative integers $(a,b)$ such that $a+b\le \reg$, $b\le\reg-1$, from~\eqref{E:TILDERDEF}  we get
\begin{align}
\big| \tilde{\mathcal R}^{\pm,\l,\eps}_{a,b+1}[\Ueps] \big| &\le C\eta \sum_{m\in\Z_\ast}\frac1{|m|^2} \sum_{0\le s \le a-1\atop \ell+j+n=b+1, j<b+1}
|\PL[\pa_y^n\pa_R^{a-1-s}q,\pa_y^\ell\pa_R^s\big(\pa_\l^j \wUepsm S_m\big)] |\notag\\
& \le C\eta |\l|^{-\frac12}\sum_{j=0}^b\|\pa_\l^j\Ueps\|_{C^{\reg-j}_R}\eps^{-\frac12}\leq  C\eta |\l|^{-\frac32}\|f_0\|_{C^{\reg}(\oom)}\eps^{-\frac12}, \label{E:RTILDEBOUND1}
\end{align}
where we have used Lemma~\ref{L:EPSILONCRUDE} with $\beta=2$ (and the observation that $b+1-j\ge1$), the bound $\|\pa_y^n\pa_R^{a-1-s}q\|_{C^0(\overline{\Upsilon})}\le C$ by~\eqref{E:QREG}, and Theorem~\ref{T:RESOLVENT1}.
Similarly, for the source term $ \mathcal S_{a,b+1}^{\pm,\l,\eps}[f_0]$ we obtain the bound
\begin{align}
\big|  \mathcal S_{a,b+1}^{\pm,\l,\eps}[f_0] \big| & \le C \sum_{m\in\Z_\ast}\frac1{|m|^3} \sum_{s\le a-1\atop\ell\le b+1}\big| \PL[\pa_y^{b+1-\ell}\pa_R^{a-1-s}q,\pa_y^\ell\pa_R^s\big(\fhatm S_m\big)] \big| \notag\\
& \le C |\l|^{-\frac32}\|f_0\|_{C^{\reg}(\oom)}\eps^{-\frac12}, \label{E:SOURCEBOUND1}
\end{align}
where we have used Lemma~\ref{L:EPSILONCRUDE} with $\beta=3$, $\ga=1$, the bound~\eqref{E:QREG}, and the observation that $b+1\ge1$). Finally, from~\eqref{E:RDEF} we obtain the bound
\begin{align}
\big|\mathcal R_{a,b}^{\pm,\l,\eps}[\pa_\eps\Ueps]\big| \le&\, C \|\pa_\l^b\pa_\eps\Ueps\|_{C^{a-1}}\delta_{a\ge2}\notag \\ &+C \sum_{m\in\Z_\ast}\frac\eta {|m|} \sum_{0\le s \le a-1\atop \ell+j+n=b}
\big|\PL[\pa_y^n\pa_R^{a-1-s}q,\pa_y^\ell\pa_R^s\big(\pa_\l^j \pa_\eps\wUepsm S_m\big)] \big|.
\end{align}
We now use~\eqref{E:PLEMELJ3} with $\beta=1$ to obtain
\begin{align}
\big|\mathcal R_{a,b}^{\pm,\l,\eps}[\pa_\eps\Ueps]\big| \le C \|\pa_\l^b\pa_\eps\Ueps\|_{C^{a-1}_R}\delta_{a\ge2} + C \eta \sum_{j\le b} \|\pa_\l^j\pa_\eps \Ueps\|_{C^{a+b-j}_R}.\label{E:REPSBOUND1}
\end{align}
We now use~\eqref{E:RTILDEBOUND1},~\eqref{E:SOURCEBOUND1},~\eqref{E:REPSBOUND1} to estimate the right-hand side of~\eqref{E:ABKEYEPS}. We thus obtain
\begin{align}
\|\pa_R^a\pa_\l^b\pa_\eps\Ueps \|_{C^0_R} \le C |\l|^{-\frac32}\|f_0\|_{C^{\reg}(\oom)}\eps^{-\frac12} + C \|\pa_\l^b\pa_\eps\Ueps\|_{C^{a-1}_R}\delta_{a\ge2}
+ C \eta \sum_{j\le b} \|\pa_\l^j\pa_\eps \Ueps\|_{C^{a+b-j}_R}.\notag
\end{align}
This is the key estimate, which via an inductive argument analogous to the one used in the proof of Theorem~\ref{T:RESOLVENT1} gives~\eqref{E:EPSUNIFORM}.

Thus for any fixed $(R,\l)\in[\Rmin,\Rmax]\times\sL$, any $a,b$ such that $a+b\le\reg$ and $b\le\reg-1$, and any $1>\eps_1>\eps_2>0$ we have
\begin{align}
&\left|\pa_R^a\pa_\l^bU^\pm_{\eps_1}(R;\l) - \pa_R^a\pa_\l^bU^\pm_{\eps_2}(R;\l)\right|  \le \int_{\eps_2}^{\eps_1} \left|\pa_\eps \Ueps(R;\l)\right| \diff\eps \notag\\
& \le C |\l|^{-\frac32}\|f_0\|_{C^{\reg}(\oom)} (\eps_1^{\frac12}-\eps_2^{\frac12})  \le  C |\l|^{-\frac32}\|f_0\|_{C^{\reg}(\oom)} (\eps_1-\eps_2)^{\frac12}, \notag 
\end{align}
where we have used~\eqref{E:EPSUNIFORM}.
Therefore the sequence $\eps\to U^\pm_{\eps}(R;\l)$ is Cauchy in $\C^{\reg}_R$ and hence has a limit $U_0^\pm\in\C^\reg_R$ as $\eps\to 0$.  Letting $\eps_2\to0$, we obtain~\eqref{E:EPSUNIFORM2.1}.
 \end{proof}



\section{Quantitative decay}\label{S:MAIN}


\begin{lemma}[Gravitational force resolvent representation]\label{L:LAP0}
Let $f(t,\cdot)$ be a solution to~\eqref{E:FULLLIN}--\eqref{E:FULLLININITIAL},~\eqref{E:FINORTH}. Then for any $R\in[\Rmin,\Rmax]$
\begin{align}
\pa_R U_{|\varphi'|f}&(t,R) = \pa_RU_{|\varphi'|\FPT}(t,R) \notag\\
& +\frac{2i\eta}{R^2}\sum_{m\in\Z_\ast}\frac1m\lim_{\k\to0} \int_{\sL} \int_{\I} e^{i\l t}|\varphi'(\bI)|\Big(\frac{\widehat{\Uepsminus}(m,\bI;\la)}{\om(\bI)+\frac{\la- i\eps}{2\pi m}}-\frac{\widehat{\Uepsplus}(m,\bI;\la)}{\om(\bI) +\frac{\la+ i\eps}{2\pi m}}\Big)S_m(\th(R,\bI))\diff \bI \diff \l, \label{E:FORCEFORMULA}
\end{align}
where $\FPT$ is the solution to the pure transport equation~\eqref{E:PTINTRO0} with the same initial data.
\end{lemma}

\begin{proof}
From Lemma~\ref{L:REPRESENTATIONS} and the Plancherel identity, 
\begin{align} 
\pa_RU_{|\varphi'|f}(t,R) & = \frac{4\pi^2}{R^2}\frac1{\pi }\sum_{m\in\Z_\ast}\frac1m\int_{\I} |\varphi'|\fhat(t,m,\bI)S_m(\th(R,\bI)) T(\bI)\diff \bI \notag\\
& = \frac{4\pi}{R^2}\frac1{2\pi i}\sum_{m\in\Z_\ast}\frac1m\lim_{\k\to0} \int_{\sL} \int_{\I} e^{i\l t}|\varphi'| \left(\widehat{f_\k^{-}}(m,\bI;\l)-\widehat{f_\k^+}(m,\bI;\l)\right) S_m(\th(R,\bI)) T(\bI)\diff \bI \diff \l,\notag
\end{align}
where in the second identity, we applied the representation formula~\eqref{eq:spectral}. Now by Lemma~\ref{L:RESOLVENTFORMULA} we may split the integral as
\beqa
&\pa_RU_{|\varphi'| f}(t,R)\\
&  =\frac{-2i}{R^2}\sum_{m\in\Z_\ast}\frac1m \lim_{\eps\to0} \int_{\sL} \int_{\I} e^{i\la t}  \Big(\frac{|\varphi'|\fhat_0 (m,\bI)}{2\pi m\om (\bI)+(\la -i\eps)}-\frac{|\varphi' |\fhat_0 (m,\bI)}{2\pi m\om(\bI)+(\la +i\eps)}\Big)S_m(\th(R,\bI)) T(\bI)\diff \bI \diff \l\\
&+\frac{2i\eta }{R^2}\sum_{m\in\Z_\ast}\frac1m\lim_{\k\to0} \int_{\sL} \int_{\I} e^{i\l t}\Big(\frac{\om(\bI)|\varphi'(\bI)|\widehat{\Uepsminus}(m,\bI;\la)}{\om(\bI) +\frac{\la- i\eps}{2\pi m}}-\frac{\om(\bI)|\varphi'(\bI)|\widehat{\Uepsplus}(m,\bI;\la)}{\om(\bI) +\frac{\la+ i\eps}{2\pi m}}\Big)S_m(\th(R,\bI)) T(\bI)\diff \bI \diff \l. \label{E:FORCEEXP1}
\eeqa
 We rewrite the first line in the form
\begin{align}
 &\frac{-2i }{R^2}\sum_{m\in\Z_\ast}\frac1m \lim_{\k\to0}  \int_{\sL} \int_{\I} e^{i\l t} |\varphi'|\fhat_0 (m,\bI) \frac{2i\eps}{(2\pi m\om (\bI)+\la)^2+\eps^2}
S_m(\th(R,\bI))T(\bI)\diff \bI \diff \l \notag\\
& = \frac{4\pi}{R^2}\sum_{m\in\Z_\ast}\frac1m \int_{\I}  e^{-2\pi i m \om(\bI)} |\varphi'|\fhat_0(m,\bI)S_m(\th(R,\bI))T(\bI)\diff \bI, 
\end{align}
where we have used~\eqref{E:AI} and recognise $\pa_RU_{|\varphi'|\FPT}(t,R)$ from~\eqref{E:GP1}.
For the second line of~\eqref{E:FORCEEXP1}, we simply recall $\om(\bI)=T(\bI)^{-1}$ and we obtain the claim.
\end{proof}


The next lemma shows that various boundary terms in our non-stationary phase argument arising at the end points of the spectrum (recall $\sL$,~\eqref{E:SIGMALDEF}) $\pm\lmin$, $\pm\infty$ in fact vanish
and are therefore no obstruction to decay. In short, vanishing at $\pm\lmin$ is due to presence of the vacuum boundary, while the vanishing at $\l=\pm\infty$ is due to decay-in-$|\lambda|$
in the resolvent estimates.


\begin{lemma}[Boundary terms vanish]\label{L:BDRYGOOD}
Let $(\l,\eps)\in \sL\times \R_{>0}$ and let $\Ueps$ be the resolvent potentials defined in~\eqref{E:FDEF0}. Then for any  $1\leq j\le\reg$, we have
\begin{align}
\lim_{\la\to\pm\infty} \big(\mathcal R_{1,j-1}^{+,\l,\eps}[\Uepsplus(\cdot;\l)]-\mathcal R_{1,j-1}^{-,\l,\eps}[\Uepsminus(\cdot;\l)]\big)&= 0 \label{E:BDRY1}\\
\lim_{\eps\to0}\big(\mathcal R_{1,j-1}^{+,\iota\lmin,\eps}[\Uepsplus(\cdot;\l)]-\mathcal R_{1,j-1}^{-,\iota\lmin,\eps}[\Uepsminus(\cdot;\l)]\big)
& = 0, \ \ \text{ for $\iota=+,-$},
\label{E:BDRY1.5}
\end{align}
where we recall the operators $\mathcal R_{1,b}^{\pm,\l,\eps}$ defined in~\eqref{E:RDEF}.
\end{lemma}

\begin{proof}
From Theorem~\ref{T:RESOLVENT1}, estimating analogously to~\eqref{E:LAJ3}, for any $b'\le\reg-1$, we have
\begin{align}
\big|\mathcal R_{1,b'}^{\pm,\l,\eps}[\Ueps]\big|
&\le C \eta |\l|^{-1} \|f_0\|_{C^{\reg}(\oom)}, \label{E:RONEB}
\end{align}
which implies the claim~\eqref{E:BDRY1} by letting $\l\to\pm\infty$.

Observe that just like in the first line of~\eqref{E:PLPM} we have 
\[
\big[\PLplus-\PLminus\big][qy^{-1}, \fhat_{0,m}S_m] = \int_{\J_R} H^\eps(y,m;\l) qy^{-1}  \fhat_{0,m}S_m \diff (y,z),
\]
where we recall~\eqref{E:KERNELH}. Therefore,
\begin{align}
&\lim_{\eps\to0} \bigg(\sum_{m\in\Z_\ast}\frac1{m^2} \int_R^{\Rmax} \frac2{r^2} \big[\PLplus-\PLminus\big][qy^{-1}, \fhat_{0,m}S_m] \diff r\bigg) \notag\\
&= 2\pi i \sum_{m\in\Z_\ast}\frac1{m^2}\int_R^{\Rmax} \frac2{r^2}\big(qy^{-1}\fhatm S_m\big)\Big|_{y=-\frac{\la}{2\pi m}} \diff r \notag\\
& = 0 \ \ \text{ for } \ \l=\pm\lmin. \label{E:BDRY2}
\end{align}
The last equality follows since from~\eqref{E:LAMBDAMINDEF} the value $y=|\frac{\lmin}{2\pi m}|$ is strictly outside the support of $|\varphi'(y,z)|$ for all $|m|\ge2$ (here recall definition~\eqref{E:LITTLEQDEF} of $q$). When $|m|=1$ the value $y=\frac{\lmin}{2\pi} = \ommin$ is
achieved only if $(y,z)\in\Jvac$ where $|\varphi'|$ vanishes as well. Furthermore, from the invertibility of the operators $\FPM$, we have
\begin{align}
\Uepsplus-\Uepsminus & = (\Fplus)^{-1}[\Fplus[\Uepsplus]-\Fminus[\Uepsminus]] - (\Fplus)^{-1}[(\Fplus-\Fminus)[\Uepsminus]] \notag\\
& = (\Fplus)^{-1}[\Fplus[\Uepsplus]-\Fminus[\Uepsminus]] - (\Fplus)^{-1}[(\Fplus-\Fminus)[\Uepsminus-U_0^-]] \notag\\
& \qquad -  (\Fplus)^{-1}[(\Fplus-\Fminus)[U_0^-]]. \label{E:SPLITTING}
\end{align}
By~\eqref{E:FUEPS} we have $\Fplus[\Uepsplus]-\Fminus[\Uepsminus] = - \sum_{m\in\Z_\ast}\frac1{m^2} \int_R^{\Rmax} \frac2{r^2} \big[\PLplus-\PLminus\big][qy^{-1}, \fhat_{0,m}S_m] \diff r$. Therefore, 
as $\eps\to0$ we note that for $\l=\pm\lmin$ the first term on the right-hand side of~\eqref{E:SPLITTING} converges to $0$ by~\eqref{E:BDRY2} and the boundedness of $(\Fplus)^{-1}$, see Proposition~\ref{P:INVERT}. The second term on the right of~\eqref{E:SPLITTING} goes to $0$
since $\|\Uepsminus-U_0^-\|_{C^0_R}\le C\eps^{\frac12}\|f_0\|_{C^{\reg}(\oom)}$ by Proposition~\ref{P:EPSUNIFORM}. The last term converges to zero as $\eps\to0$ again since by definition of $\F_\l^{\pm,\eps}$,~\eqref{E:FDEF}, we have
\begin{align}
(\Fplus-\Fminus)[U_0^-]=-\eta \sum_{m\in\Z_\ast}\frac1m\int_R^{\Rmax}\frac{4\pi}{r^2} (\PLplus-\PLminus)[q, \widehat U_{0,m} S_m] \diff r.
\end{align}
We now use an argument analogous to~\eqref{E:BDRY2} to conclude the convergence to $0$.
In conclusion, letting $\eps\to0$ in~\eqref{E:SPLITTING} we conclude that 
\begin{align}\label{E:SPLITTING2}
U_0^+(R,\pm\lmin)=U_0^-(R,\pm\lmin), \ \ R\in[\Rmin,\Rmax].
\end{align}
This is a starting point of an inductive argument using~\eqref{E:PURELA} that shows analogously that
\begin{align}\label{E:SPLITTING3}
\pa_\l^bU_0^+(R,\pm\lmin)=\pa_\l^bU_0^-(R,\pm\lmin), \ \ R\in[\Rmin,\Rmax], \ \ b\in\{0,\dots,\reg-1\}.
\end{align}
Indeed, to show~\eqref{E:SPLITTING3}, we suppose it holds for $0\leq b'\leq b-1$. By arguing as in~\eqref{E:SPLITTING} and employing~\eqref{E:PURELA}, it is easy to see that to show~\eqref{E:SPLITTING3} for $b$, it suffices to control $\tilde{\mathcal R}^{+,\l,\eps}_{1,b}[\Uepsplus]-\tilde{\mathcal R}^{-,\l,\eps}_{1,b}[\Uepsminus]$ at $\la=\pm\lmin$. We first split
\beqa\label{E:BDRYRSPLITTING1}
& \tilde{\mathcal R}_{1,b}^{+,\lmin,\eps}[\Uepsplus(\cdot;\lmin)]-\tilde{\mathcal R}_{1,b}^{-,\lmin,\eps}[\Uepsminus(\cdot;\lmin)] \\
&\qquad =\tilde{\mathcal R}_{1,b}^{+,\lmin,\eps}[(\Uepsplus-\Uepsminus)(\cdot;\lmin)]+(\tilde{\mathcal R}_{1,b}^{+,\lmin,\eps}-\tilde{\mathcal R}_{1,b}^{-,\lmin,\eps})[\Uepsminus(\cdot;\lmin)].
\eeqa
To estimate the second contribution on the right hand side, we first write
\beqas
{}&\lim_{\eps\to0}(\tilde{\mathcal R}_{1,b}^{+,\lmin,\eps}-\tilde{\mathcal R}_{1,b}^{-,\lmin,\eps})[\Uepsminus(\cdot;\lmin)]\\
&=\lim_{\eps\to0} \bigg(\sum_{m\in\Z_\ast}\frac1{m}\sum_{\substack{\ell+j+n=b \\ j<b}}c^{0;b}_{0;\ell,j,n} \int_R^{\Rmax} \frac2{r^2} \big[\PLplus-\PLminus\big][\pa_y^nq, \pa_y^\ell\pa_\la^j\wUepsminusm(\cdot;\lmin) S_m] \diff r\bigg), 
\eeqas
where we recall~\eqref{E:TILDERDEF}. Observe that since $b\leq K-1$, the series converges absolutely by~\eqref{E:PLEMELJ3} and Theorem~\ref{T:RESOLVENT1}. Thus we may pass the limit through the summation and conclude the limit is zero, as in~\eqref{E:BDRY2}.
To control the first term on the right hand side of~\eqref{E:BDRYRSPLITTING1}, we note that by an estimate analogous to~\eqref{E:LAJ3} we obtain
\begin{align}
\lv \tilde{\mathcal R}_{1,b}^{+,\lmin,\eps}[(\Uepsplus-\Uepsminus)(\cdot;\lmin)]\rv\le C \eta \lmin^{-1}\sum_{j=0}^{b-1}\|\pa_\l^j(\Uepsplus-\Uepsminus)(\cdot;\lmin)\|_{C^{\reg-j}_R},
\end{align}
which converges to $0$ as $\eps\to0$ by the inductive assumption, completing the proof of~\eqref{E:SPLITTING3}.

Now to control~\eqref{E:BDRY1.5}, we simply make the splitting as in~\eqref{E:BDRYRSPLITTING1} and argue as in the proof of ~\eqref{E:SPLITTING3} to conclude the proof of the lemma.
\end{proof}

\subsection{Proof of the main theorem}\label{S:PROOF}


\begin{proof}[Proof of Theorem~\ref{T:MAIN}]
{\em Step 1.} Before passing to the limit as $\eps\to0$, for any $a+b\le\reg$, $b\le\reg-1$ by the same argument as in Theorem~\ref{T:RESOLVENT2}, one  proves a uniform-in-$\eps$ bound 
\begin{align}
\|\Rplusa[\Uepsplus]-\Rminusa[\Uepsminus]\|_{C^0_{R}} \le C \eta |\l|^{-2} \|f_0\|_{C^{\reg}(\oom)}, \ \ 0<\eps<1.\label{E:RPMDIFF}
\end{align}
We now show that $\lim_{\eps\to0}\big(\Rplusa[\Uepsplus]-\Rminusa[\Uepsminus]\big)$ exists in $C^0$ and therefore the bound shown above carries over to the limit.
To that end we note the identity
\begin{align}
&\pa_\eps\left(\Rplusa[\Uepsplus-\Uepsminus]+(\Rplusa-\Rminusa)[\Uepsminus]\right) \notag\\
& = i \tilde{\mathcal R}_{a,b+1}^{+,\l,\eps}[\Uepsplus-\Uepsminus] + \Rplusa[\pa_\eps\Uepsplus-\pa_\eps\Uepsminus] \notag\\
& \ \ \ \ + i (\tilde{\mathcal R}_{a,b+1}^{+,\l,\eps}-\tilde{\mathcal R}_{a,b+1}^{-,\l,\eps})[\Uepsminus] + (\Rplusa-\Rminusa)[\pa_\eps\Uepsminus],
\end{align}
where we recall definition~\eqref{E:TILDERDEF} of $\tilde{\mathcal R}_{a,b+1}^{\pm,\l,\eps}$. 
We now argue analogously to the proof of Proposition~\ref{P:EPSUNIFORM}, using in particular~\eqref{E:RTILDEBOUND1},  to conclude that 
\begin{align}
\big|\pa_\eps\left(\Rplusa[\Uepsplus-\Uepsminus]+(\Rplusa-\Rminusa)[\Uepsminus]\right)\big|
\le C  \eps^{-\frac12} |\l|^{-\frac32} \|f_0\|_{C^{\reg}(\oom)}. 
\end{align}
Using the mean value theorem we infer that the limit $\lim_{\eps\to0}\big(\Rplusa[\Uepsplus]-\Rminusa[\Uepsminus]\big)$ exists in $C^0_R$ and satisfies the bound
\begin{align}\label{E:EPSUNIFORM2}
\|\lim_{\eps\to0}\big(\Rplusa[\Uepsplus]-\Rminusa[\Uepsminus]\big)\|_{C^0_R} \le C \eta |\l|^{-\frac32} \|f_0\|_{C^{\reg}(\oom)}.
\end{align}
By the dominated convergence theorem and the bound~\eqref{E:EPSUNIFORM2} we conclude that 
\begin{align}\label{E:DECAYCRUCIAL1}
\big|\lim_{\eps\to0} \int_{\sigma_{\L}} e^{i\l t}\left(\Rplusa[\Uepsplus]-\Rminusa[\Uepsminus]\right)\diff \l \big| \le C\|f_0\|_{C^{\reg}(\oom)}.
\end{align}

{\em Step 2.} 
We now show  for $b\leq K-1$ that
\begin{align*}
\pa_R U_{|\varphi'|f}(t,R) 
& =   \pa_R U_{|\varphi'|\FPT}(t,R) -\frac{i\eta}{2\pi} \frac{(-1)^b}{(it)^b}  \lim_{\eps\to0} \int_{\sigma_{\L}} e^{i\l t}\left(\Rplus[\Uepsplus]-\Rminus[\Uepsminus]\right)\diff \l .
\end{align*}
To prove this, by writing $e^{i\l t}= \frac1{(it)^b}\pa_\l^b\big(e^{i\l t}\big)$, we integrate by parts $b$ times with respect to $\la$ in the final term of~\eqref{E:FORCEFORMULA}, applying the identity~\eqref{E:RTILDEIDENTITY}, to see that
\beqa
{}&\frac{2i\eta}{R^2}\sum_{m\in\Z_\ast}\frac1m\lim_{\k\to0}\int_{\sL} \int_{\I} e^{i\l t}|\varphi'(\bI)|\Big(\frac{\widehat{\Uepsminus}(m,\bI;\la)}{\om(\bI)+\frac{\la- i\eps}{2\pi m}}-\frac{\widehat{\Uepsplus}(m,\bI;\la)}{\om(\bI) +\frac{\la+ i\eps}{2\pi m}}\Big)S_m(\th(R,\bI))\diff \bI \diff \l\\
&=\frac{2i\eta}{R^2}\sum_{m\in\Z_\ast}\frac1m\lim_{\k\to0}\int_{\sL} e^{i\la t}\Big(\mathbf{Pl}_m^{-,\l,\eps}[q, \wUepsminusm S_m]-\mathbf{Pl}_m^{+,\l,\eps}[q, \wUepsplusm S_m]\Big)\,\diff \l\\
&=-\frac{i\eta}{2\pi }\frac{(-1)^b}{(it)^b}  \lim_{\eps\to0} \int_{\sigma_{\L}} e^{i\l t}\left(\tilde{\mathcal R}_{1,b}^{-,\l,\eps}[\Uepsminus]-
{\tilde{\mathcal R}_{1,b}^{+,\l,\eps}}[\Uepsplus]\right)\diff \l \notag\\
& \ \ \ -\frac{i\eta}{2\pi} \lim_{\eps\to0}\sum_{j=1}^b (-1)^{j-1}\frac1{(it)^{j}}\big[\tilde{\mathcal R}_{1,j-1}^{-,\l,\eps}[\Uepsminus]-\tilde{\mathcal R}_{1,j-1}^{+,\l,\eps}[\Uepsplus]\big]_{\pa\sL},
\eeqa
where we note that each integration by parts is justified by the estimate~\eqref{E:RPMDIFF}. Sending $\eps\to0$ and applying the boundary estimate Lemma~\ref{L:BDRYGOOD}, we conclude the claimed identity.

Now applying a further $\pa_R^{a-1}$, we observe from the uniform convergence estimate~\eqref{E:EPSUNIFORM2} that the derivatives commute with the limit $\eps\to0$, so that, again from~\eqref{E:RTILDEIDENTITY}, we derive
\begin{align}
\pa_R^{a}U_{|\varphi'|f}&(t,R) = \pa_R^{a}U_{|\varphi'|\FPT}(t,R) \notag\\
&+\frac{i\eta}{2\pi} \frac{(-1)^b}{(it)^b} \sum_{s=0}^{a-1}{ a-1 \choose s}\pa_{R}^{a-1-s}(R^{-2}) \lim_{\eps\to0} \int_{\sigma_{\L}} e^{i\l t}\left(\tilde{\mathcal R}_{s+1,b}^{+,\l,\eps}[\Uepsplus]-
\tilde{\mathcal R}_{s+1,b}^{-,\l,\eps}[\Uepsminus]\right)\diff \l. 
\end{align}
Let now $a+b\le \reg$, $b\le \reg-1$ be given.
From bound~\eqref{E:DECAYCRUCIAL1},  we conclude that
\begin{align}
|\pa_R^a U_{|\varphi'|f}(t,R)| \le |\pa_R^aU_{|\varphi'|\FPT}(t,R)| + (1+t)^{-b} \|f_0\|_{C^{\reg}(\oom)},
\end{align}
which together with Theorem~\ref{thm:PTdecay} completes the proof.
\end{proof}

\begin{corollary}\label{C:SCATTERING}
Let the assumptions of Theorem~\ref{T:MAIN} hold and assume that $\reg\ge3$. Then there exists a profile $f_\infty$ such that, as $t\to\infty$, 
\[
\tilde f(t,\th,\bI):=f(t,\th+t\om(\bI),\bI) \to f_\infty \text{  in $\mathscr H$.}
\]
\end{corollary}

\begin{proof}
We note that $\pa_t \tilde f = (\pa_t + \om(\bI)\pa_\th)f(t,\th+t\om(\bI),\bI) = - \eta \mathcal T U_{|\varphi'|f}U_f (t,\th+t\om(\bI),\bI)$, and therefore
$|\pa_t\tilde f|\lesssim \eta (1+t)^{-\reg+1}$ by Theorem~\ref{T:MAIN}. In the last statement we used the identity $\T U_f = w\pa_R U_f$. This implies 
pointwise convergence of $\tilde f$ to some asymptotic profile $f_\infty$. Evaluating $\pa_t\|\tilde f - f_\infty\|_{\mathscr H}^2$, using Theorem~\ref{T:MAIN}
again, and the Cauchy-Schwarz inequality, we obtain the claim.
\end{proof}

\begin{remark}[Trace formula]
Using the uniform convergence~\eqref{E:EPSUNIFORM2} it is relatively straightforward to compute the limit as $\eps\to0$ on the right-hand side of~\eqref{E:FORCEFORMULA}.
To that end, we introduce the functions
\begin{align*}
\Sigma^{\pm}(m,R,y,z;\l) : = q(R,y,z)(\widehat{U_0^+}\pm\widehat{U_0^-})(m,R,y,z;\l) \sin(2\pi m (\th(R,y,z)))\chi_{[\underline\om(R,z),\overline\om(R,z)]}(y),
\end{align*}
for $m\in\Z$, $R\in[\Rmin,\Rmax]$, $(y,z)\in \J_R$, and $\l\in\sL$, 
where we recall the definition of $q$,~\eqref{E:LITTLEQDEF}, and we set, for $z_0\geq 0$,
\begin{align*}
 \overline\om(R,z_0)=\max_{\J_R}\om(E(R,y,z_0),L(R,y,z_0)), \ \  
\underline\om(R,z_0)=\min_{\J_R}\om(E(R,y,z_0),L(R,y,z_0)).
\end{align*}
We also recall the classical Plemelj singular integral formula 
\beq\label{E:PLEMELJ}
\lim_{\eps\to0}\frac{1}{2\pi i}\int_\R\frac{g(y)}{y-(x\pm \eps i)}\,\diff y=\frac{\pm g(x)+i\mathcal{H}g(x)}{2},
\eeq
where the Hilbert transform is, as usual, given by
$\mathcal{H}g(x)=\frac1\pi\lim_{\eps\to 0}\int_{|x-y|>\eps}\frac{g(y)}{x-y}\,\diff y$.

Using~\eqref{E:PLEMELJ} and the uniform estimates of Section~\ref{S:RESOLVENTBOUNDS}, one can prove the identity
\begin{align}
\pa_R U_{|\varphi'|f}(t,R) & = \pa_RU_{|\varphi'|\FPT}(t,R) + \int_{\sL}  e^{i\l t} \text{Tr}(R;\l) \diff \l,
\end{align}
where
\begin{align*}
\text{Tr}(R;\l) : = \frac{2\pi}{R^2}\int_0^{E_0-\Psi_{L_0}(R)} \sum_{m\in\Z_\ast}\frac1m \left\{\Sigma^+(m,R,-\frac{\la}{2\pi m},z;\l)   - i \H (\Sigma^-)(m,R,-\frac{\la}{2\pi m},z;\l)\right\} \diff z. 
\end{align*}
\end{remark}


\appendix



\section{Generalised Fa\`a di Bruno}\label{A:FDB}


We recall the generalised Fa\`a di Bruno formula from~\cite{Constantine96}. For multi-indices ${ \nu}=(\nu_1,\ldots,\nu_d)\in \N_0^d$ and vectors ${z}=(z_1,\ldots,z_d)\in \R^d$, we write
\beqas
|{\bf \nu}|=\sum_{i=1}^d\nu_i,\ \
\nu!=\prod_{i=1}^d(\nu_i!),\ \
\pa^\nu=\frac{\pa^{|\nu|}}{\pa_{x_1}^{\nu_1}\cdots\pa_{x_d}^{\nu_d}},\ \
z^\nu=\prod_{i=1}^dz_i^{\nu_i}
\eeqas
as usual. We use the following linear ordering on the space of multi-indices: $\nu_1\prec \nu_2$ if one of the following holds:
\begin{itemize}
\item $|\nu_1|<|\nu_2|$,\\
\item $|\nu_1|=|\nu_2|$ and $(\nu_1)_1<(\nu_2)_1$,\\
\item $|\nu_1|=|\nu_2|$, $(\nu_1)_1=(\nu_2)_1$,...$(\nu_1)_k=(\nu_2)_k$ and $(\nu_1)_{k+1}<(\nu_2)_{k+1}$ for some $k\geq 1$.
\end{itemize}

Suppose $f\in C^{|\nu|}(\R^m;\R)$ around $g(x_0)$ and $g\in C^{|\nu|}(\R^d;\R^m)$ around $x_0$. Then, setting $h=f\circ g$,
\beqa\label{eq:genFdB}
\pa^\nu h(x_0)=\sum_{1\leq|\la|\leq |\nu|}(\pa^\la f)(g(x_0))\sum_{s=1}^{|\nu|}\sum_{p_s(\nu,\la)}(\nu!)\prod_{j=1}^{s}\frac{(\pa^{\ell_j}g)^{k_j}}{(k_j!)(\ell_j!)^{|k_j|}},
\eeqa
where the multi-indices $k_i$ are $m$-dimensional, $\ell_j$ are $d$-dimensional, and the constraint set $p_s(\nu,\la)$ is given by
\beq\label{def:ps}
p_s(\nu,\la)=\big\{(k_1,\ldots,k_s;\ell_1,\ldots,\ell_s)\,|\,|k_i|>0,\ \ 0\prec\ell_1\prec\cdots\prec\ell_s,\ \ \sum_{i=1}^s k_i=\la,\ \ \sum_{i=1}^s|k_i|\ell_i=\nu\big\}.
\eeq

\section{Pure point potential}
We collect here some properties of the Hamiltonian dynamics associated to the pure point potential, arising from the steady state $\fmn$ in the case $\eta=0$. The following can all be found in \cite[Section 6.2.1]{Straub24}.

\begin{lemma}\label{L:PUREPOINT}
Let $\eta=0$ and consider the Hamiltonian flow generated by the pure point mass. Then the effective potential, for $L>0$, is given by
\beq
\Psi_L^0(r):=-\frac{M}{r}+\frac{L}{2r^2},\ \ \ r\in(0,\infty).
\eeq
This effective potential achieves a unique minimum at
\beq
r^0_L=\frac{L}{M},\qquad \Psi_L^0(r^0_L)=-\frac{M^2}{2L}=:E_{\min}^{L,0}
\eeq 
and the turning points for the flow are given by 
\beq
r^0_\pm(E,L)=\frac{-M\mp\sqrt{M^2+2EL}}{2E},
\eeq
so that  $\Psi_L^0(r^0_\pm(E,L))=E$ on the set $\A^0:=\{(E,L)\,|\,L>0,E_{\min}^{L,0}<E<0\}.$ 

Given $\kappa<0$ (compare~\eqref{E:KAPPA}), we define
\beq
R_{\min}^0=r_-^0(\kappa,L_0),\qquad  R_{\max}^0=r_0^+(\kappa,L_0).
\eeq
Solutions to the Hamiltonian ODE system $\dot{r}=w$, $\dot{w}=-(\Psi^0_L)'(r)$ are time-periodic for $(E,L)\in\A^0$ with period function
\beq
T^0(E,L)=\frac{\pi}{\sqrt{2}}\frac{M}{(-E)^{\frac32}}, \ \ \ (E,L)\in\A^0.
\eeq
Finally, we define the action-angle support, associated to $\kappa$, by 
\beq
\I^0:=\{(E,L)\,|\,L_0\leq L\leq L_{\max}^0,\ E_{\min}^{L,0}\leq E\leq \kappa\},
\eeq
where
\beq
L_{\max}^0=-\frac{M^2}{2\kappa}.
\eeq
\end{lemma}



\end{document}